\documentclass[english,11pt]{article}
\usepackage[german,french,english]{babel}
\usepackage[cp850]{inputenc}
\usepackage{latexsym,graphicx, fancybox}
\usepackage{graphicx}
\usepackage{booktabs}

\usepackage{amssymb, amsmath, amsfonts}
\usepackage{amsmath, amssymb}
\usepackage{color}
\usepackage{tikz}
\usepackage{latexsym}

\usepackage[colorlinks=true, allcolors=blue]{hyperref}    
\usepackage{tocloft}        
\usepackage{enumitem}

\usepackage{mathrsfs}

\usepackage{float}   
\usepackage{graphicx}  
\usepackage{hyperref}  
\usepackage{latexsym}  
\usepackage{xcolor}   
\usepackage{lipsum}   
\usepackage{subcaption}

\usepackage{amssymb}  
\usepackage{amsmath}  
\usepackage{amsbsy}
\usepackage{amsthm}
\usepackage{bbm}    
\usepackage{eucal}   
\usepackage{mathrsfs}  

\usepackage{multirow}  
\usepackage{ulem}    
\usepackage{dsfont}

\usepackage[round]{natbib}
\usepackage[textsize=small]{todonotes}

\usepackage{etoolbox}
\apptocmd{\thebibliography}{\setlength{\itemsep}{5.7pt}}{}{}

\usepackage{microtype}
	\usepackage{amsthm}
	
	\renewcommand{\d}{\mathrm{d}}
	\newcommand{\dd}{\,\mathrm{d}}
	\newcommand{\R}{\mathbb{R}}
	\newcommand{\N}{\mathbb{N}}
	\newcommand{\E}{\mathbb{E}}
	\renewcommand{\P}{\mathbb{P}}

	\newcommand{\p}{\mathbb{P}}
	\newcommand{\F}{\mathbb{F}}
	
	\newcommand{\cF}{{\mathcal F}}
	\newcommand{\cA}{{\mathcal A}}

	\newcommand{\as}{\mbox{{\rm a.s.}}}

	\usepackage{authblk}
	
	\newcommand{\T}{\top}
	\renewcommand{\c}{\alpha}
	\newcommand{\Mid}{{\ \Big|\ }}
	\newcommand{\Acal}{{\mathcal A}}
	\newcommand{\V}{\mathrm{Var}}

	\DeclareMathOperator{\diag}{diag}

	\newtheorem{Theorem}{Theorem}[section]
	\newtheorem{Definition}[Theorem]{Definition}
	\newtheorem{Proposition}[Theorem]{Proposition}

	\newtheorem{Lemma}[Theorem]{Lemma}
	\newtheorem{Corollary}[Theorem]{Corollary}
	\newtheorem{Remark}[Theorem]{Remark}

	\newtheorem{assumption}{Assumption}[section]

	\usepackage{thm-restate}
	\usepackage{float}
	\usepackage{hyperref}
	
\title{On explicit solutions to a class of quadratic BSDEJs driven by affine Volterra processes with jumps and applications\footnote{ The work of the first author is partially supported by
		Agence Nationale de la Recherche (ReLISCoP grant ANR-21-CE40-0001).  The second author is supported by \textit{Ecole Doctorale Sciences Mathematiques de Paris Centre}.}}

\author[1,2]{Sigui Brice Dro\footnote{E-mail: {\tt dro@lpsm.paris} }}
\author[1]{Emmanuel Gnabeyeu\footnote{E-mail: {\tt emmanuel.gnabeyeu\_mbiada@sorbonne-universite.fr} }}

\affil[1]{Laboratoire de Probabilit\'es, Statistique et Mod\'elisation, UMR 8001, Sorbonne Universit\'e and Universit\'e Paris Cit\'e, 4 pl. Jussieu, F-75252 Paris Cedex 5, France.}
\affil[2]{BPCE SA, 7 promenade Germaine Sablon Paris 75013, France.}

\begin{document}
\selectlanguage{english}
\maketitle

\vspace{-1cm}
\renewcommand{\abstractname}{Abstract}
\begin{abstract}
	In this paper we consider a class of quadratic BSDEs with jumps (quadratic BSDEJs) involving inhomogeneous affine Volterra processes and show that their solution can be reduced to solving a system of generalized inhomogeneous integral Riccati-Volterra ordinary differential equations with L\'evy jump compensators. 
    This yields a rich and flexible class of quadratic BSDEJs that are analytically tractable, in the sense that their solutions are explicit up to the solution of an associated integral Riccati-Volterra ODE with L\'evy jump compensator. 
    As an application, we provide analytically tractable solutions to the continuous-time Markowitz mean-variance portfolio selection problem within a multivariate class of affine Volterra models allowing jumps driven by an independent Poisson random measure. In this  non-Markovian and non-semimartingale market framework with unbounded random coefficients, the classical stochastic control approach cannot be directly applied to the associated optimization
	task. Instead, the problem is tackled using the martingale optimality principle by constructing a family of submartingale processes characterized via solutions to a novel Riccati backward stochastic differential equation with jumps (Riccati BSDEJ), particular subclass of the aforementionned quadratic BSDEJ.
	 Specifically, we obtain analytical closed-form expressions for the optimal feedback control 
	  as well as the mean-variance efficient frontier, both of which depend on the solution to the associated multivariate inhomogeneous Riccati-Volterra system, while the optimal value function is expressed using the solution to this original Riccati BSDEJ.
	 Furthermore, 
     numerical experiments\footnote{The code is available in a \href{https://github.com/Emma-Gnabeyeu/Theses_EMG/blob/main/Notebooks/Mean-VarOptim.ipynb}{Python notebook}} on a two-dimensional fake stationary rough Heston model introduce in~\cite{Gnabeyeu2026b}, is discussed and used to highlight the impact of stabilized rough volatilities on the Markowitz allocation problem.
\end{abstract}



\tableofcontents
\vspace{-.3cm}
\section{Introduction}
\vspace{-.3cm}

\medskip
\noindent
In stochastic control theory and mathematical finance, backward stochastic differential equations (BSDEs for short) have become fundamental tools since their general mathematical study by \cite{pardoux1990adapted}. These equations provide a powerful framework for modeling dynamic optimization problems, valuation and hedging of claim contingent in incomplete markets or solving control problems under uncertainty. The seminal work of El Karoui, Peng and Quenez \cite{el1997backward} recalled the connection between BSDEs and some partial differential equations, demonstrating the versatility of BSDEs across diverse applications including utility maximization, optimal portfolio selection and hedging in incomplete markets. Over the past three decades, BSDEs have evolved into one of the most active areas of research in applied mathematics and financial mathematics, with far-reaching implications for both theory and numerical point of view.

\medskip
\noindent
For linear BSDEs explicit closed-form solutions can be readily obtained, making them analytically tractable. However, many real-world problems in finance require BSDEs with nonlinear driver functions, particularly those exhibiting quadratic growth in the control variable $Z$. Such quadratic growth appears naturally in utility maximization problem or risk-sensitive control for example. The pioneering work of \cite{Kobylanski2000} adressed this type of BSDEs in the pure Brownian motion setting, establishing existence and uniqueness results under appropriate hypothesis. Nonetheless, the introduction of jumps which are essential for discontinuous phenomena in financial markets significantly complicates the analysis.\\
The extension of quadratic BSDEs to jump-diffusion models has been extensively studied by Morlais (\cite{morlais2009utility}, \cite{Morlais2010}). When jumps components are present, finding explicit solutions in the nonlinear case becomes substantially more challenging, and the complexity of the resulting equations typically necessitates rescourse to numerical schemes. Among others, we have 
\cite{lejay2014numerical} which is based on discretization Brownian motion and the Poisson measure and \cite{gnoatto2022deep} that use deep learning to solve them in high dimensions, but these methods remains computationally intensive and do not yield analytical insights into the structure of the solution. The lack of explicit solutions creates a significant bottleneck for theorical and pratical implementation in financial applications.\\
A major breakthrough in the pursuit of explicit solutions came from the recogniton that imposing additional structure on some inputs data could yields analytically tractable classes of BSDEs. In \cite{richter2014explicit}, A. Richter shows that within the class of Markovian affine volatility models, it is possible to reduce certain classes of BSDEs to a system of matrix-valued  ordinary differential equations (ODEs), thereby obtaining explicit solutions in a constructive manner. 

\medskip
\noindent
The class of Markovian volatility model is a little bit restrictive as the empirical observation show 
that implied and realized volatilities of major financial indices exhibit sample paths with low H\"older regularity~\citep{GatheralJR2018}, significantly rougher than those generated by classical Brownian-motion-driven models, has fundamentally reshaped the modeling of asset price dynamics and fostered the development of rough volatility models. 
Building on the widespread practical success of the celebrated~\cite{Heston1993} stochastic volatility model, several rough extensions have been developed.  Among the most prominent is the rough Heston model introduced by~\cite{el2019characteristic}, which is rooted in insights from market microstructure. This framework was subsequently generalized to the Volterra Heston model in~\cite{abi2019affine}, and more recently to the so-called \textit{fake stationary Volterra Heston model} proposed in~\cite{EGnabeyeuPR2025, EGnabeyeuR2025} with the aim of providing a unified and consistent framework that captures both short- and long-maturity behaviors, while allowing robust fitting across the entire term structure. This broader class of models encompasses the aforementioned specifications and is constructed by modeling the volatility process as a stochastic Volterra equation of convolution type with a time-dependent diffusion coefficient. 
Moreover, empirical studies based on high-frequency financial data (see e.g., ~\cite{TodorovTauchen2011}) strongly indicate that realistic asset volatility should account for both path roughness and discontinuous movements induced by jumps; see also~\cite{Xia2022} for a detailed discussion of volatility derivative models that simultaneously incorporate rough volatility and volatility jumps. Motivated by these empirical findings, \cite{dro2026optimal} extends affine stochastic Volterra models of stock volatility to accommodate jump dynamics.

\medskip
\noindent
In the present paper, we build on this line of research by considering quadratic BSDEs with jumps in this substantially richer modeling framework, which incorporates both roughness and memory effects, as well as Brownian and Poisson sources of randomness. Specifically, we study quadratic BSDEJs driven by inhomogeneous affine Volterra processes with jumps.
Affine Volterra models generalize affine classical processes by incorporating path-dependent effects through Volterra-type Kernels, thereby capturing memory and rough-path phenomena.  The literature on stochastic Volterra equations with jumps is instead very recent, and shows a growing interest. \cite{AlfonsiSzulda2025} establish conditions upon the kernel  and coefficients for the strong existence and pathwise uniqueness of a non-negative c\`adl\`ag solution to stochastic volterra integral equation with jumps. \cite{BondiLivieriPulido2024, dro2026optimal} derived, under an affine structure imposed upon the coefficients, a semi-explicit formula for the Fourier-Laplace transform of the solution.\\
Our main contribution is to show that quadratic BSDEJs under affine Volterra dynamics with jumps can be reduced to solving a system of generalized inhomogeneous integral Riccati-Volterra ordinary differential equations. This reduction yields a rich and flexible class of analytically tractable quadratic BSDEJs, where solutions are explicit up to the solution of an associated integral Riccati-Volterra ODE with L\'evy jump compensators.

\medskip
\noindent
As a key application, we provide analytically tractable solutions to the continuous-time Markowitz mean-variance optimization problem within a multivariate class of affine Volterra models allowing jumps driven by a Poisson random measure as proposed in~\cite{dro2026optimal}. The mean-variance criterion in portfolio allocation problem pioneered by~\cite{Markowitz1952}'s seminal work 
is one of the classical problems from mathematical finance
in which investment decisions rules are made according to a trade-off  between the return of the investment and the associated risk. 
Due to its intuitive appeal and analytical tractability, the Markowitz mean variance framework has become a cornerstone of mo\-dern portfolio management in both theory and practice.
Over the past decades, an extensive body of literature has sought to extend the original static formulation to a continuous-time setting. Early contributions focused on the classical Black-Scholes framework in complete markets, most notably the seminal work of~\cite{ZhouLi2000}. Subsequent research broadened the scope to more general market environments featuring random coefficients and multiple assets; see, for
instance,~\cite{LimZhou2002,Lim2004,JeanblancEtAl2012,ChiuWong2014,Shen2015}. 
These works approach the problem by drawing on tools from convex optimization, stochastic linear-quadratic (LQ) control theory, and backward stochastic differential equations (BSDEs). One of our goals is to extend the results of~\cite{HanWong2020a,AbiJaberMillerPham2021} by incorporating jumps into the affine stochastic Volterra integral equations describing the volatility process following a verification procedure of the martingale optimality principle.

\medskip
\noindent {\sc \textbf{Organization of the Work.}}  
\noindent The rest of the paper is organized as follows.  Section~\ref{sect-Note_prelim} introduces the necessary notation and provides an overview and characterization of affine Volterra processes with jumps in \(\R^d_+\), which will be used throughout the paper. In section~\ref{sect:Main_solBDSEJ}, we study explicit solutions to a class of quadratic BSDEs with jumps driven by these Volterra processes and their connection with generalized inhomogeneous integral
Riccati-Volterra equation with L\'evy jump compensator. This section is divided into three parts.
We first derive heuristically candidate expressions for the Brownian stochastic integrand, denoted by $Z$, and the Poisson stochastic integrand variable, denoted by $U$, in Sections~\ref{subsect:Intuition_Brownian} and~\ref{subsect:Intuition_jumps}, under a specific (degenerate) multidimensional structure induced by a distortion transformation. We then consider the general case in Section~\ref{subsect:Gen_Case}, where we provide explicit solutions through a verification argument, without imposing any restrictions on the aforementioned structure.
The results are then applied to the classical Markowitz portfolio optimization problem  in Section~\ref{sect-appl}, within a multi-asset financial market in which volatility is modeled by a multivariate class of \textit{Volterra square-root processes with jumps} driven by an independent Poisson random measure.
For such a market model, we derive an explicit solution to the continuous-time Markowitz mean-variance optimization problem under a more general correlation structure between asset returns and their respective volatilities. 
The analysis relies both on the BSDEJ results developed in Section~\ref{sect:Main_solBDSEJ}, whose solutions are explicitly characterized in terms of associated Riccati--Volterra equations with jumps, and on a verification argument based on the martingale optimality principle.
In Section~\ref{Sec:Num}, we demonstrate the practical implications of our findings through numerical experiments based on a two-dimensional fake stationary rough Heston volatility model introduced in~\cite{Gnabeyeu2026b} in order to maintain a consistent modeling framework across time scales, from short to long maturities.
Finally, Section~\ref{sect:proofMresult} is devoted to some additional technical proofs of the results.
  
\section{Notations and characterization of affine Volterra processes with jumps}\label{sect-Note_prelim} 
\noindent We first introduce the notation and basic framework in Section~\ref{sect-Notation}, and then present streamlined versions of affine Volterra processes with jumps in Section~\ref{sect-prelim}.

\subsection{Notations.}\label{sect-Notation}
\noindent We start with the notation we use subsequently. From now on, let $T > 0$ be fixed.

\smallskip
\noindent $\bullet$ Denote $\mathbb{T} = [0, T] \subset \mathbb{R}_+$, ${\rm Leb}_d$ the Lebesgue measure on $(\R^d, {\cal B}or(\R^d))$, $d\geq1$.

\noindent $\bullet$ $\mathbb{X} := {\cal C}([0,T], \R^d)  (\text{resp.} \quad {\mathcal C_0}([0,T], \R^d))$ denotes the set of continuous functions (resp. null at 0)  from $[0,T]$ to $\R^d $ and ${\cal B}or(\mathbb{X})$ denotes the  Borel $\sigma$-field of $\mathbb{X}$ induced by the $\sup$-norm topology. 

\smallskip
\noindent $\bullet$ Let $E$ represents a polish space and  $\mathcal{N}(0,1)(dx)$ denote the density of the standard gaussian law.

\smallskip 
\noindent $\bullet$ For $p\in(0,+\infty)$, $L_{\mathbb H}^p(\P)$ or simply $L^p(\P)$ denote the set of  $\mathbb H$-valued random vectors $X$  defined on a probability space $(\Omega, {\cal A}, \P)$ such that $\|X\|_p:=(\E[\|X\|_{\mathbb H}^p])^{1/p}<+\infty$. 

\smallskip 
\noindent $\bullet$ $X\perp \! \! \!\perp Y$  stands for independence of random variables, vectors or processes $X$ and $Y$.   ${\bold{1}_d}$ denotes the vector in $\R^d$ with all components equal to $1$.

\smallskip 
\noindent $\bullet$  We will use in this paper the matrix norm \(|A| = \operatorname{tr}(A^{\top}A)\); for all \(A\in \R^{d\times d}\) and the vector norm \(|u|^2 = \|u\|^2 = u^\top u,\;\forall\, u\in \R^{d}\) whenever the context is clear.
Here $\| \cdot \|$ (or $| \cdot |$) denote the Euclidian norm on $\R^d$. 

\smallskip 
\noindent $\bullet$ Let \([0,T]\) be a finite time horizon (\(T<\infty\)). Throughout the paper, a weak solution is understood to be defined on a filtered probability space $(\Omega,\mathcal F,\mathbb P,\mathbb F=(\mathcal F_t)_{t\ge0})$~\footnote{We equip $(\Omega,\cF,\P)$ with a filtration $\mathbb F$ satisfying the usual conditions i.e. it is right-continuous and $\P-$complete.}. The filtered probability space is assumed to support the following independent random elements:
\begin{itemize}
	\item a \( 2d\)-dimensional $\mathbb{F}$-Brownian motion $ \big(B, B^\top\big) = \big( (B_t)_{t \geq 0}, (B^\top_t)_{t \geq 0}\big)$ for \(d\geq1\) ;
	\item an independent $\mathbb{F}$-Poisson random measure $N$ on $(E, \mathcal{B}or(E))$ with compensator 
	$\pi(X_t, de)\,dt$, where the Levy measure $\pi$ is positive and satisfies
	$\pi(X_t, de) = \nu_0(de) + X_t^\top \, \nu(de) = \nu_0(de) + \sum_{i=1}^{d} X_t^i \, \nu_i(de) $ given an affine Volterra process $X$ 
	and a finite positive measure $\nu_k$ satisfying 
	\begin{align}
		\nu_k(\{0\}) = 0, \mbox{ for }k = 0,...,d .  
	\end{align}
\end{itemize}
\noindent
We denote by $\tilde{N}(dt,de) := N(dt,de) - \pi(X_t, de)\,dt$ its compensated measure and $\mathcal{P}$ the $\sigma$-algebra of predictable processes.
\medskip
\noindent Our problem will be defined under a given complete filtered probability space \((\Omega,\mathcal{F},\mathbb{F} = \{\mathcal{F}_t\}_{0\leq t\leq T}, \P)\). The filtration \(\mathbb{F}\) is not necessarily the augmented filtration generated by \( (B, B^\top,N)\) ;
thus, it can be a strictly larger filtration. Here \(\P\) is a real-world probability measure from which a family of equivalent probability measures can be generated. 
We define 
\begin{align*}
	L^{\infty}_{\F}([0,T], \R^d) 
	&= \left\{ Y:\Omega \times [0,T]\to \R^d,\; \F\text{-prog.~measurable and bounded a.s.} \right\} \\
	L^p_{\F}([0,T], \R^d) 
	&= \left\{ Y:\Omega \times [0,T]\to \R^d,\; \F\text{-prog.~measurable s.t.~} \E\Big[ \int_0^T |Y_s|^p ds \Big] < \infty \right\} \\
	\mathbb{S}^{\infty}_{\F}([0,T], \R^d) 
	&= \left\{ Y:\Omega \times [0,T]\to \R^d,\; \F\text{-prog.~measurable s.t.~} \sup_{t\leq T} |Y_t(\omega)|< \infty \text{ a.s.} \right\}\\
	\mathbb{S}^{p}_{\F}([0,T], \R^d) 
	&= \left\{ Y:\Omega \times [0,T]\to \R^d,\; \F\text{-prog.~measurable s.t.~} \E\Big[\sup_{0\leq t\leq T} |Y_t|^p\Big] < \infty \right\}\\
	L^p_{\F}([0, T],\tilde{N}) 
	&= \left\{ 
	\begin{aligned}
		U: \Omega \times [0,T]\times E \to \R,\;
		\mathcal{P} \otimes \mathcal{B}or(E)\text{-measurable such that \qquad} \\
		\mathbb{E} \left[\int_0^T \int_E |U_t(e)|^p \, \pi(X_t, de) \, dt \right] < \infty\qquad \qquad\qquad
	\end{aligned}
	\right\}.
\end{align*}
In the following, unless otherwise specified, $V$ denotes the Volterra volatility defined in \eqref{VolSqrt_}.
Here $|\cdot|$ denotes the Euclidian norm on $\R^d$, \(d\geq1\) for short.  Classically, for $p \in (1, \infty  )$, we define $L^{p, loc}_{\F}([0,T], \R^d)$ as the set of progressive 
processes $Y$ for which there exists a sequence of increasing stopping times 
$\tau_n \uparrow \infty$ such that the stopped processes $Y^{\tau_n}$ are in $L^{p}_{\F}([0,T], \R^d)$ for every $n \geq 1$, and we recall that it consists of all progressive processes $Y$ s.t. 
$ \int_0^T |Y_t|^p dt$ $<$ $\infty$, a.s. Likewise for $\mathbb{S}^{p, loc}_{\F}([0,T], \R^d)$. To unclutter notation, we write $L^{p, loc}_{\F}([0,T])$  instead of $L^{p, loc}_{\F}([0,T], \R^d)$ when 
the context is clear. 

\smallskip 
\noindent $\bullet$ Let $M$ be a c\`adl\`ag local martingale vanishing at $0$. The stochastic
exponential $\mathcal{E}(M)$ is defined by
\[
\mathcal{E}(M)_t
= \exp\!\left(M_t - \tfrac{1}{2}\langle M^c \rangle_t\right)
\prod_{0 < s \leq t}\bigl(1 + \Delta M_s\bigr)\exp(-\Delta M_s),
\quad t \in [0,T],
\]
where $M^c$ denotes the continuous martingale part of $M$,
$\langle M^c \rangle$ its quadratic variation, and $\Delta M_s := M_s - M_{s-}$
the jump of $M$ at time $s$. When $\Delta M_t > -1$ for all $t \in [0,T]$
$\mathbb{P}$-a.s., $\mathcal{E}(M)$ is a positive local martingale thus a supermartingale.

 \smallskip 
\noindent $\bullet$ We define the compensated exponential function and its scaled version by
\begin{equation}\label{eq:ExpCom_functs}
\text{for all } x\in\mathbb{R},\qquad
h(x)=e^x-x-1,
\qquad
h_\lambda(x)=\frac{h(\lambda x)}{\lambda}.
\end{equation}

\medskip
\noindent $\bullet$ For \(T>0\), write \( ( \mathbb{D}([0,T], \mathbb{R}), \text{J}_1) \) as  the usual Skorokhod space \footnote{It refers to the set of stochastic processes defined on $[0,T]$ that are right-continuous with left limits (c\`adl\`ag). This space generalizes the classical space $\mathcal C([0,T], \mathbb{R})$, which is equipped with the uniform convergence norm.} i.e. the space of  all $\R$-valued c\`adl\`ag functions or processes on \( [0,T] \) equipped with the Skorokhod topology \( J_1 \) (see, e.g., Billingsley \cite{Billingsley1999}).

\smallskip 
\noindent $\bullet$ Let \(\mathcal{M}\) denote the space of all $(\R_+, {\cal B}or(\R))$-measurable functions \(m\) on \(\mathbb{R}_+\) such that the restriction \(m|_{[0, T]}\), for any \(T > 0\), is a \(\mathbb{R}\)-valued finite measure (i.e. the restriction $m|_{[0,T]}$ with $T > 0$ is well-defined). For \(m \in \mathcal{M}\) and a compact set \(E \subset \mathbb{R}_+\), we define the total variation of \(m\) on \(E\) by:

\begin{equation}
	|m|(E) := \sup \left\{ \sum_{j=1}^N |m(E_j)| : \{E_j\}_{j=1}^N \text{ is a finite measurable partition of } E \right\}.
\end{equation}
\noindent We assume that the set of measure $m \in \mathcal{M}$ on \(\R_+\) is of locally bounded variation.

\smallskip 
\noindent $\bullet$ Convolution between a function and a measure. Let \(f : (0, T] \to \mathbb{R}\) be a measurable function and \(m \in \mathcal{M}\). Their convolution (whenever the integral is well-defined) is defined by
\begin{equation}\label{eq:convolmeasure}
	(f * m)(t)= \int_{[0,t)} f(t - s) \, dm(s) = \int_{[0,t)} f(t - s) \, m(ds) = (f\stackrel{m}{*}\mathbf{1})_t, \quad t \in (0, T].
\end{equation}

\smallskip 
\noindent $\bullet$ For a measurable function \( \varphi: \mathbb{R}^+ \to \mathbb{R} \), \(\forall p \geq 1,\) we denote: 

\centerline{$
	\| \varphi \|^p_{L^p([0,T])} := \int_0^{T} |\varphi(u)|^p \, du,  \; \displaystyle \|\varphi\|_{\infty}=\|\varphi\|_{\sup} := \sup_{u\in \mathbb{R}^+}|\varphi(u)| \; \text{and} \; \displaystyle \|\varphi\|_{\infty,T}=\|\varphi\|_{\sup,T} := \sup_{u\in [0,T]}|\varphi(u)|.
	$}

\subsection{Preliminaries and characterization of affine Volterra processes with jumps}\label{sect-prelim} 
Fix $d\in\mathbb{N}$. Let $\mathbb{R}^d$ be the state space that
we consider, let $g_0\in L^1_{\mathrm{loc}}(\mathbb{R}_+;\mathbb{R}^d)$ and
$K\in L^2_{\mathrm{loc}}(\mathbb{R}_+;\mathbb{R}^{d\times d})$ be a matrix-valued
kernel.  Let
$\sigma=\diag(\sigma_1,\ldots,\sigma_d)$ with \(\sigma^i\) a (locally) positive bounded Borel function and $D$ a $\R^{d\times d}$-valued matrix.
We denote by $\mu : \mathbb{R}_{+} \to \mathbb{R}^d_{+}$ a positive function.
Let $E$ be a measurable space of marks, and let $\eta: E\to\mathbb{R}^d$
be a measurable function specifying the size of the jump associated with a mark $e\in E$. 
We also introduce a \textit{characteristic triplet}
$(b,a,\pi)$ consisting of the measurable maps
$b:\mathbb{R}^d\to\mathbb{R}^d$,
$a:\mathbb{R}^d\to\mathbb{R}^{d\times d}$ and the transition kernel
$\pi(x,de)$ from $\mathbb{R}^d$ to $\mathbb{R}^d$. We require this triplet to
be affine on $\mathbb{R}^d$, meaning that, 
\begin{equation}\label{eq:affineProcess}
\forall\,t\geq0,\;\forall x\in \R^d,\quad b(t,x) = \mu(t) + D x, \quad
a(t,x) =  \sigma(t) \diag(x) \sigma^\top(t), \quad
\pi(x,de) = \nu_0(de)
+ x^\top \nu(de) 
\end{equation}
Here we denote by \(\nu:=(\nu_1, \ldots,\nu_d)^\top\) and
$(\nu_k)_{k=0,\ldots,d}$ are positive finite measures on $E$ satisfying $\int_{E} \eta(e)\eta(e)^\top\,\nu_k(de) < \infty$ and
$\nu_k(\{0\})=0$ so that $\pi$ is a finite Borel measure on the complete separable metric space $E$ .
Throughout the paper, we denote by
$X=(X_t)_{t\ge0}$ a predictable process with trajectories in
$L^1_{\mathrm{loc}}(\mathbb{R}_+;\mathbb{R}^d)$ and such that
$X\in \mathbb{R}^d$, $\mathbb{P}\otimes dt$-a.e. It is defined on a filtered
probability space $(\Omega,\mathcal{F},\mathbb{F},\mathbb{P})$ where the
filtration $\mathbb{F}=(\mathcal{F}_t)_{t\ge0}$ satisfies the
usual conditions and $\mathcal{F}_0$ is the trivial
$\sigma$-algebra on $\Omega$. Moreover, we assume that $X$ solves the
following affine stochastic Volterra equation of convolution type
\begin{equation}\label{eq:Volterra1}
X_t
=
g_0(t)
+
\int_0^t K(t-s)\, dZ_s,
\qquad
\mathbb{P}\text{-a.s., for a.e. } t\geq0.
\end{equation}
where $Z$ is a $d$-dimensional semimartingale starting at $0$ whose
differential characteristics with respect to the Lebesgue measure are
$(b(t,X_t),a(t,X_t),\pi(X_t,de))$, $t\ge0$. These characteristics are
taken with respect to the ``\textit{truncation function}
$\chi\big(\eta(e)\big)=\eta(e)$, $e\in E$'', which can be chosen because $Z$ is a
special semimartingale due to~\cite[Proposition 2.29, Chapter II]{jacod2013limit},
and the local integrability of the trajectories of $X$. 
Similar arguments as in~\cite[section 2.]{BondiLivieriPulido2024} proves that the integral on the right hand side of~\eqref{eq:Volterra1} is well-defined $\P$-a.e. for $t\in\R_+$ so that~\eqref{eq:Volterra1}  admit the following canonical representation
\begin{align}\label{eq:volterra2}
	X_t = g_0(t) + \int_0^t K(t-s) \big(\mu(s)+D X_s ds + \sigma(s) \sqrt{\diag(X_s)}dW_s + \int_{E}\eta(e)\tilde{N}(ds,de)\big).
\end{align}
Here $\eta$ specifies how a mark $e \in E$ 
generated by the random measure $N$ is translated into a jump $\eta(e)$ in 
each component of the vector~\eqref{eq:volterra2}. It is worth mentioning that the common 
random measure \(N\) models and reflects the systematic jump in the sense that, all components of \(X\) are exposed to the same shocks, but respond to these shocks heterogeneously, each according to its own jump-size profile \(\eta_i\).

\medskip\noindent
We assume the existence and non-negativity of a c\`adl\`ag $\R^{d}_+$-valued weak solution \(X\) for the Volterra SDEs~\eqref{eq:Volterra1}--~\eqref{eq:volterra2} for any initial condition $g_0: [0,T] \to \R^{d}_+$ defined on some filtered probability space $(\Omega,\mathcal F, (\mathcal F)_{t\geq 0}, \mathbb P)$ and satisfying among others for a constant \(C > 0\)
	\begin{equation}\label{eq:SupMoment_X}
		\sup_{0\leq t\leq T} \E\left[ \|X_t\|^p \right] \leq C \big(1+ \E\big[ \sup_{0\leq t\leq T} \|g_0(t)\|^p \big]\big) < \infty, \quad p > 0.
	\end{equation}
For instance, weak existence of $X$ such that~\eqref{eq:SupMoment_X} holds is established under
suitable assumptions on the kernel $K$, the jump components $(\eta,\pi)$ and the specifications of $g_0$ as stated in the
following assumption and the remarks therein.
\noindent We always work under the assumption below, which applies to the inhomogeneous Volterra equation~\eqref{eq:Volterra1}--~\eqref{eq:volterra2}.
\begin{assumption}[On Volterra convolutive kernels and jump]\label{assump:kernelJumpVolterra}\textcolor{white}{.} 
	\begin{enumerate}
		\item[(i)] 
		Assume that \(K\) is diagonal with  strictly positive scalar kernels $K_i$ on the diagonal for \(i=1,\ldots,d\) that is completely monotone on $(0, \infty)$ and satisfies for any $T > 0$ 
        \begin{itemize}
			\item The integrability assumption: The following is satisfied for some $\widehat\theta_i\in (0,1]$.
			\vspace{-.2cm}
			\begin{equation}\label{eq:contKtilde}
				(\widehat {\cal K}^{cont}_{\widehat \theta})\;\;\exists\,\widehat{\kappa_i}< +\infty,\;\forall\bar{\delta}\!\in (0,T],\; \widehat \eta(\delta) :=  \sup_{t\in [0,T]} \Big[\int_{(t-\bar{\delta})^+}^t \hskip-0,25cm K_i\big(t-u\big)^2 du\Big]^{\frac12}\le \widehat \kappa_i \,\bar{\delta}^{\,\widehat \theta_i}.
				\vspace{-.2cm}
			\end{equation}
			\item  The continuity assumption: \(({\cal K}^{cont}_{\theta}) \;\;  
			\exists\, \kappa_i< +\infty,\; \exists \; \theta_i\in (0,1] \; \text{such that}\; \forall \,\bar{\delta}{\in (0, T)}\)
			\vspace{-.2cm}
			\begin{equation}\label{eq:Kcont}
				({\cal K}^{cont}_{\theta}) \; \forall \,\bar{\delta}{\,\in (0, T)},\; \eta(\bar{\delta}):= \sup_{t\in [0,T]} \left[\int_0^t |K_i(\big(s+\delta)\wedge T\big)-K_i(s)|^2ds \right]^{\frac 12} \le  \kappa_i\,\bar{\delta}^{\theta_i}.
				\vspace{-.2cm}
			\end{equation}
			
		\end{itemize}  
		\noindent
		\item[(ii)]  Assume that, for all $i = 1, \ldots, d$ and all $k = 0, \ldots, d$, the jump sizes satisfy
		\[
		\nu_k\big(\{ e \in E : \eta^i(e) < 0 \}\big) = 0,
		\]
		and are square-integrable, i.e.,
		\begin{equation}\label{eq:IntJumpSize}
		\int_E \big(\eta^i(e)\big)^2 \, \nu_k(de) < \infty.
		\end{equation}
		\noindent
    \end{enumerate}
\end{assumption}
\noindent Note that the drift $b(t,x) = \mu(t)+ Dx$ in~\eqref{eq:volterra2} is clearly Lipschitz continuous in $x\in\R^d$,  uniformly in time $t\!\in [0,T]$. Moreover, the drift term \(b\), the diffusion coefficient \(\sigma(t, x):=\sigma(t)\sqrt{\diag(x)}\) together with the jump components $(\eta,\pi)$ under conditions~\ref{assump:kernelJumpVolterra}~$(ii)$ are of linear growth, i.e. there is a constant \(C_{b,\sigma, \eta} > 0\) such that
\begin{equation}\label{eq:LinearGrowth}
	\|b(t, x)\| + |\sigma(t, x)| + \int_E \|\eta(e)\|^2\,\pi(x,de) \leq C_{b,\sigma, \eta}(1 + \|x\|),
	\quad \text{for all } t \in [0, T] \text{ and } x \in \mathbb{R}^d.
\end{equation}
\begin{Remark}\label{rm:Kernels}
	\noindent 1.  Assumption~\ref{assump:kernelJumpVolterra}~$(i)$ covers, for instance,  constant non-negative kernels,  fractional kernels  of the form $\frac{t^{\alpha-1}}{\Gamma(\alpha)}\mathbf{1}_{\mathbb{R}_+}$ with $\alpha  \in(\frac12,1]$, exponentially decaying kernels ${\rm e}^{-\beta t}$ with $\beta>0$ and more generally the exponentially decaying fractional kernels  \(K(t) = \frac{t^{\alpha-1}}{\Gamma(\alpha)} e^{-\beta t}\mathbf{1}_{\mathbb{R}_+}
	\) with \( \alpha \in \left( \tfrac{1}{2}, 1 \right] \) and \( \beta \ge 0 \) ( see e.g. 
    ~\cite{EGnabeyeu2025, GnabeyeuPages2026b}, \cite[ Example 2.2 ]{GnabeyeuPages2026}). These kernels satisfy $(\widehat {\cal K}^{cont}_{\widehat \theta})$ and $({\cal K}^{cont}_{\theta})$, for $\alpha >1/2$ with $\theta = \widehat \theta = \min\bigl( \alpha-\frac12,\; 1\bigr)$.

\noindent 2. The existence of a solution $X$ to~\eqref{eq:volterra2} satisfying~\eqref{eq:SupMoment_X} is guaranteed under suitable specifications of $g_0$, as detailed below for the case $d=1$:
    \begin{itemize}[label=\(\bullet\)]
 \item
 If $g_0(t)=V_0 $, for some $V_0 \in \R_+$, then under Assumption \ref{assump:kernelJumpVolterra}, and the linear growth~\eqref{eq:LinearGrowth}, when $\sigma(t)\equiv\sigma$ on $[0,T]$, ~\cite[Theorem 2.7. or Theorem 5.2.]{AlfonsiSzulda2025} 
 establishes the strong existence and pathwise uniqueness of a non-negative c\`adl\`ag solution of the stochastic Volterra equation (\ref{eq:volterra2}) such that \eqref{eq:SupMoment_X} holds.
 \item \noindent If $g_0(t)=V_0 \varphi(t)$, for some $\varphi: [0,T]\to\R_+$ and $\R_+$-valued random variable $V_0$, then
 in the case $\eta\equiv0$ and fractional kernel $K=K_{\alpha}$ for $\alpha\in[\frac12,1)$, Equation~\eqref{eq:volterra2} admits a unique in-law positive continuous weak solution, obtained as the scaling limit of suitably time-modulated Hawkes processes with heavy-tailed kernels in the nearly unstable regime; see~\cite{EGnabeyeuR2025} and its multidimensional extension $d\geq 1$ is straightforward at least when \(D\) is diagonal. 
Moreover, suppose that \( V_0^i \in L^p(\mathbb{P}) \) for some suitable \( p \in (0, +\infty) \), such that
		the process $t \to v_0^i(t) :=V_0^i \varphi^i(t)$ is strictly positive and  absolutely continuous and $(\mathcal F_t)$-adapted.
		Moreover, for some $\delta_i > 0$, for any $p > 0$, there exists \(C:=C_{T,p}\) such that
		\begin{equation*}
			\mathbb{E} \,\!\Big(\sup_{t \in [0,T]} |v_0^i(t)|^p\Big) < +\infty,\quad
			\mathbb{E}\!\big[\,|v_0^i(t') - v_0^i(t)|^p\,\big] 
			\le C \Big( 1 + \mathbb{E}\,\big[\sup_{t \in [0,T]} |v_0^i(t)|^p\big] \Big) |t' - t|^{\delta_i p}.
		\end{equation*}
 then, a solution \( t \mapsto V_t^i \) to  Equation~\eqref{VolSqrt_}  starting from  $V_0^i$ 
has a \( \big( \delta_i \wedge \theta_i \wedge \widehat \theta_i - \epsilon \big) \)-H\"older pathwise continuous modification  on $\R_+$ for sufficiently small \( \epsilon > 0 \).
   \end{itemize}
\end{Remark}

\medskip
\noindent From now on, we
continue to consider $g_0$ as the initial input curve, with the contribution of the drift term in~\eqref{eq:volterra2} incorporated into its definition. Likewise, we redefine the
 $d-$dimensional semimartingale $Z$, starting from $0$, by \(Z_t := \int_{0}^{t}\big(D X_s ds + \sigma(s) \sqrt{\diag(X_s)}dW_s + \int_{E}\eta(e)\tilde{N}(ds,de) \big) ,\) for all \( t \geq 0\) so that Equation~\eqref{eq:Volterra1} reads
\begin{equation} 
	\label{VolSqrt2} 
	\begin{aligned}
		X_t = 
        g_0(t) + \int_0^t K(t-s) dZ_s\quad\text{with}\quad g_0(t) := g_0^{old}(t) + \int_0^t K(t-s)\mu(s) ds.
	\end{aligned}
\end{equation}
Finally, we consider the $\R^d$-valued process for \( s\geq t,\) 
\begin{equation} 
	\label{eq:processg}
    \begin{aligned}
	g_t(s) &= g_0(s) + \int_0^t K(s-u) \big(D X_u du + \sigma(u) \sqrt{\diag(X_u)}dW_u + \int_{E}\eta(e)\tilde{N}(du,de)\big)\\
    &=  g_0(s) + \int_0^t K(s-u) dZ_u.
\end{aligned}
\end{equation}
One notes that for each, $s\leq T$, $(g_t(s))_{t\leq s}$ is the adjusted forward process 
    		\begin{equation}\label{eq:Condprocessg}
    		g_t(s) = \; \mathbb E^\P\Big[  X_s - \int_t^s K(s-u)DX_udu \Mid \cF_t\Big],\quad \P \text{- a.s., for a.e.}\; s > t.
    		\end{equation}
This adjusted forward process is commonly used (see, e.g.,~\cite{BondiLivieriPulido2024,Gnabeyeu2026b, dro2026optimal}) to elucidate the affine structure of affine Volterra processes with c\`adl\`ag trajectories.
\noindent The process in~\eqref{VolSqrt2} is non-Markovian and non-semimartingale in general.

\bigskip\noindent 
We associate the affine process $X$ with admissible parameter set
$(a,b,\pi)$ in~\eqref{eq:affineProcess} to a BSDEJ.
To allow for a more flexible
financial modeling, our real-valued BSDEJ will have the following form
\begin{equation}\label{eq:BSDEJ_general}
Y_t = F(\int_0^T X_s ds)
   -\int_t^T Z_s^{\top}\,dW_s
   -\int_t^T\int_{E}U_s(e)\tilde{N}(ds,de)
   +\int_t^T
      f(s,X_s,Y_s,Z_s,U_s)\,ds,
\end{equation}
for $t\in[0,T]$, where the terminal condition $Y_T:=F(\int_0^T X_s ds)$ is allowed to depend
on the affine process $X$, thereby permitting greater flexibility in financial modeling, in particular for the pricing of variance swaps (see, e.g., \cite[Propositions 3.2 and 3.4]{dro2026optimal}). Recall that $W$ is the
Brownian motion of the underlying affine process $X$ and the generator
is a deterministic Borel measurable function
\[
f:[0,T]\times \mathbb{R}^d_+\times\mathbb{R}\times \mathbb{R}^d\times\mathbb{R}
\longrightarrow\mathbb{R}.
\]
It is worth noting that, unlike in~\cite{richter2014explicit}, the BSDEJ~\eqref{eq:BSDEJ_general} is associated with an affine Volterra process with jumps (see, e.g.,~\cite{BondiLivieriPulido2024,dro2026optimal}).
\begin{Definition}\label{def:Solut_bsde}
    A solution to BSDEJ \eqref{eq:BSDEJ_general} is a family of adapted processes
$(Y,Z,U)$ with values in
$\mathbb{R}\times\R^d \times\mathbb{R}$ such that the equation \eqref{eq:BSDEJ_general} is a.s. satisfied, the processes $Z$ and $U$ are predictable processes
such that  $\left(Z, U \right)\in L^2_{\F}([0,T], \R^d) \times L^2_{\F}([0, T],\tilde{N})$, the integrability condition $\int_0^T|f(s,X_s,Y_s,Z_s,U_s)|\,ds<\infty$ holds true almost surely and the mapping $t\mapsto Y_t$ is c\`adl\`ag.
\end{Definition}

\begin{Proposition}\label{prop:Transform} Consider that there exists a solution triplet $\left(Y, Z, U\right) \in \mathbb{S}^{p}_{\F}([0,T],\R) \times L^2_{\F}([0,T], \R^d) \times L^2_{\F}([0, T],\tilde{N})$ to the quadratic BSDEJ \eqref{eq:BSDEJ_general} for some \(p\geq1\).
Define the process $\Gamma:=\exp(Y)$, then, it admits the following dynamics:
\begin{equation}\label{eq:GammaDynamicsGeneral}
\left\{
\begin{aligned}
\frac{d\Gamma_t}{\Gamma_{t-}}
={}&
\Big(
-f\big(t,X_t,\log\Gamma_t,Z_t,U_t\big)
+\frac12|Z_t|^2
+\int_E h\big(U_t(e)\big)\pi(X_t,de)
\Big)dt
\\
&
\qquad + Z_t^\top dW_t
+\int_E
\big(e^{U_t(e)}-1\big)\tilde N(dt,de),
\\
\Gamma_T&=e^{Y_T},\quad Y_T:=F\big(\int_0^T X_s ds\big).
\end{aligned}
\right.
\end{equation}
Moreover, assume that the local martingale defined by
    \begin{align}\label{eq:localMart_M} 
		M_t &= \; \mathcal E\Big( \int_0^t Z_s^\top dW_s + \int_0^t\int_{E}\big(e^{U_s(e)}-1\big)\tilde{N}(ds,de) \Big). 
	\end{align} is a true $\P$- martingale, then \(\Gamma\) admits the following representaton
as a Laplace transform for every $ t\in [0,T]$
	\begin{equation}\label{eq:ito_gamma} 
		\Gamma_t = \E \Big[ \exp\Big(F\big(\int_0^T X_s ds\big)+\int_t^T \big(f\big(s,X_s,\log\Gamma_s,Z_s,U_s\big) ds
-\frac12|Z_s|^2
-\int_E h\big(U_s(e)\big)\pi(X_s,de)\big)\big) ds\Big) \mid \mathcal F_t\Big].
	\end{equation}
\end{Proposition}
\noindent{\bf Proof:} An application of It\^o-L\'evy formula to the exponential transformation $\Gamma = \exp(Y)$ together with the BSDE dynamics for $Y$ in~\eqref{eq:BSDEJ_general} yields~\eqref{eq:GammaDynamicsGeneral}.
	Now, let's define the process:
	\begin{align}
		M_t &= \;  \Gamma_t \exp\Big(\int_0^t\big( f\big(s,X_s,\log\Gamma_s,Z_s,U_s\big)
-\frac12|Z_s|^2
-\int_E h\big(U_s(e)\big)\pi(X_s,de)\big) ds\Big), \quad t \leq T.
	\end{align}
	An application of It\^o's formula combined with the dynamics \eqref{eq:GammaDef2} shows that $dM_t = M_t  \Big(\Lambda_t^\top dW_t+ \int_{E}\big(e^{U_t(e)}-1\big)\tilde{N}(dt,de)\Big)$, and so $M$ is a local martingale of the form~\eqref{eq:localMart_M}.
	\noindent Therefore, by assumption, the stochastic exponential \(M\) is a true $\P$- martingale. Now, as $\Gamma_T=\exp\Big(F\big(\int_0^T X_s ds\big)\Big)$, writing $\E[M_T|\mathcal F_t]=M_t$, we obtain  the representaton~\eqref{eq:ito_gamma}
    and the proof is completed. \hfill $\Box$

\smallskip \noindent {\bf Remark:} Note that, Equation~\eqref{eq:GammaDynamicsGeneral} is still a nonlinear BSDE with jumps and if moreover \(c_y\equiv0\), then
this Equation can be reffered to as a Ricatti backward stochastic differential equation with jumps (Riccati BSDEJ). (See analogy in the litterature with e.g.~\cite[Theorem 3.1]{AbiJaberMillerPham2021} or ~\cite[Theorem 3.1]{ChiuWong2014}, upon setting \(\Tilde{Z}_t= \Gamma_t Z_t\) ).

\section{Explicit solutions to a class of quadratic BSDEs with Jumps}\label{sect:Main_solBDSEJ}
This section is devoted to the explicit solution methodology of a class of quadratic BSDEs with jumps.  The key observation is that, under a suitable analytic specification of the generator $f$ (see for example~\cite{richter2014explicit}), the coupled system~\eqref{eq:BSDEJ_general} can be transformed into a system of Riccati-Volterra ordinary differential equations. This stands in contrast to the classical existence theory (see, e.g., ~\cite{Becherer2006,Morlais2010,Kobylanski2000}), which is primarily based on Lipschitz continuity and quadratic growth assumptions on the generator $f$. 

\medskip
\noindent
For a certain class of generators and terminal conditions \(F\) we can give
the solution to the above BSDEJ~\eqref{eq:BSDEJ_general} in terms of Riccati-Volterra ODEs. Suppose the
terminal condition
$F:\R^d_+\rightarrow\mathbb{R}$ is affine, more precisely,
\begin{equation}\label{eq:TermCond}
\exists \, (\Theta,\Xi)\in \R^d\times\R, \quad \text{such that}\quad F(x)= \Theta x+\Xi,
\quad
x\in \R^d_+.
\end{equation}
 Let the set $ {\cal S}$ be defined as the space which contains all
functions $ u : E \to \R $.
The class of generators $f$ is more involved, more precisely the generator
$f:[0,T]\times \mathbb{R}^d_+\times\mathbb{R}\times \mathbb{R}^d\times\mathbb{R}
\longrightarrow\mathbb{R}$ is allowed to have the following form 
\begin{equation}\label{eq:Driver_General}
\begin{aligned}
    f(t,x,y,z,u) ={}&  z^\top \diag(c_{zz}(t)) z
+ z^\top \diag(c_{z\sqrt{x}}(t)) \sqrt{x}
+ c_x(t)^\top x
+ c_y(t)y
+ c(t) + x^\top \int_E g_\nu (t, u(e)) \nu(de)
\\
&\;+ \int_E \Big(z^\top \diag\big(g_{z\sqrt{x}}(t, u(e))\big)\sqrt{x}
+ g_x(t, u(e))^\top x + g_y(t, u(e))y + g(t, u(e))\Big) \nu_0(de).
\end{aligned}
\end{equation}
where
\[
c_{zz},\,c_{z\sqrt{x}}, \,
c_x:[0,T]\to\mathbb{R}^d,\qquad c,c_y:[0,T]\to\mathbb{R},\qquad g_\nu : [0, T ] \times {\cal S} \to \R, 
\]
are continuous functions, 
$c_{zz},\,c_{z\sqrt{x}}$ assumed symetrix-valued, $\sqrt{x}:=\bigl(\sqrt{x_1},\ldots,\sqrt{x_d}\bigr)^\top$
is understood componentwise (componentwise square root), and $ g_\nu $ is an $\nu(de )$-integrable function, which is continuous in time.
 Finally, still in the above equation~\eqref{eq:Driver_General},
\[
g_{z\sqrt{x}},\,
g_x:[0,T]\times {\cal S}\to\mathbb{R}^d,\qquad g,g_y:[0,T]\times {\cal S}\to\mathbb{R},
\]
are $\nu_0(de)$-integrable functions 
which are also assumed to be
continuous in time. With the form the the driver~\eqref{eq:Driver_General}, the quadratic terms are scalar products of coordinatewise products.

\smallskip
\noindent The reader is invited to refer to~\cite[Propositions 3.2 and 3.5]{Gnabeyeu2026b,dro2026optimal} for motivating examples demonstrating that the proposed structure emerges naturally in relevant applications.

\medskip
\noindent 
Now, we are bound to determine an expression for the vector \(Z\) and the scalar \(U\). We begin by investigating toy examples using a distortion transformation technique. In each case, we relate the solution \(\Gamma:=\exp(Y)\) of the non-linear Riccati BSDEJ~\eqref{eq:GammaDynamicsGeneral} with an conditional expectation or a representation as a Laplace transform via a proper transformation to be specified in the sequel. 

\subsection{Intuition from the degenerate multivariate structure without jumps}\label{subsect:Intuition_Brownian}
In this section, we consider a case with no jumps \(\eta \equiv0\) (pure Brownian case).
For our first toy example, we consider a quadratic time-dependent BSDE without jumps and with a terminal condition depending linearly on the affine process $X$. Note that this framework encompasses, among others, the fake stationary Volterra Heston model introduced in \cite{EGnabeyeuR2025,EGnabeyeuPR2025,Gnabeyeu2026b}. In such case, Equation~\eqref{eq:BSDEJ_general} with $c_y\equiv0$ reads, for all $t\leq T$
	\begin{equation}\label{eq:GammaDef_}
		\left\{
		\begin{array}{ccl}
			dY_t &=&  \;  -\big(Z_t^\top \diag\big(c_{zz}(t)\big) Z_t+Z_t^\top \diag\big(c_{z\sqrt{x}}(t)\big)\sqrt{X_t}+c_x(t)^\top X_t+c(t)\big)dt+Z_t^\top dW_t, \\
			Y_T &=&  F\big(\int_0^T X_s ds\big). 
		\end{array}  
		\right.
	\end{equation}
\begin{Proposition}
	\label{prop:innerDegenerate} Let the assumptions of Proposition~\ref{prop:Transform} be in force. Assume moreover that the map \(c_{zz}\)   in~\eqref{eq:Driver_General} is of the form $c_{zz}(t) = c_{zz}\,1_{d}=(c_{zz},\ldots,c_{zz})^\top$ for every $t\in[0,T]$ and  $c_{zz}\neq 0$. 
Let $\psi \in C([0,T],\R^d)$ be the solution (if any) to the inhomogeneous Riccati-Volterra equation \eqref{eq:RiccatiMarkowitzpsi1}-\eqref{eq:RiccatiMarkowitzpsi2} below.
	\begin{align} 
		\psi^i(t)&= \int_0^t K_i(t-s)\big(\Theta_i + c_x^i(T-s)  + F_i(T-s,\psi(s))\big) ds,  \label{eq:RiccatiMarkowitzpsi1} \\
		F_i(s,\psi) &=   c^i_{z\sqrt{x}}(s) \sigma^i(s) \psi_i + (D^\top \psi)_i + \frac {1} 2  \times 2 c_{zz} \cdot (\sigma^i(s)\psi^i)^2, \quad i=1,\ldots,d, \label{eq:RiccatiMarkowitzpsi2}
	\end{align} 
    Then, the pair $ (Y, Z)$ defined by
		\begin{equation}\label{eq:Solbsde}
			\left\{
			\begin{array}{ccl}
				Y_t &=&   F\big(\int_0^t X_s ds\big) + \int_t^T c(s) ds +  \sum_{i=1}^d\int_t^T  \big(\Theta_i + c_x^i(s)  + F_i(s,\psi(T-s))\big) g^i_t(s) ds, \\
				Z_t^i &=&  \sigma^i(t) \psi^i(T-t) \sqrt{X^i_t}, \quad  i=1,\ldots,d, \quad 0 \leq t \leq T, 
			\end{array}  
			\right.
		\end{equation}
        where the $\R^d$-valued process \((g_t(s))_{t\leq s}\),  denotes the adjusted forward process given by~\eqref{eq:processg}, is a 
        $\mathbb{R}\times\R^d$-valued solution to the riccati BSDE~\eqref{eq:GammaDef_}.
\end{Proposition}
\medskip
\noindent{\bf Proof:}
\noindent We rather proceed with \(\Gamma:=\exp\big(Y\big)\). For the solvability of the non-linear Riccati BSDE~\eqref{eq:GammaDef_}, inspired by the martingale distortion transformation, let \(\delta \in \R\), applying It\^o's lemma to \(\Gamma^\delta=\exp\big(\delta Y\big)\) yields \(d\Gamma_t^\delta=\Gamma_t^\delta \big(\delta Y_t + \frac{\delta^2}{2}\,d\langle Y\rangle_t\big)\).
Since the map $c_{zz}$ in ~\eqref{eq:GammaDef_} is constant of the form $(c_{zz},\dots, c_{zz})^\top$ for $c_{zz}\neq 0$, we obtain
\begin{align*}
d\Gamma_t^\delta
={}&
\delta \Gamma_t^\delta
\Big[
-\Big( \big(
c_{zz}-\frac{\delta}{2}
\big)|Z_t|^2
+
Z_t^\top \diag\big(c_{z\sqrt{x}}(t)\big)\sqrt{X_t}
+c_x(t)^\top X_t+c(t)
\Big)dt + Z_t^\top dW_t \Big],
\\
={}&
\delta \Gamma_t^\delta
\Big[
-\Big( \big(
c_{zz}-\frac{\delta}{2}
\big)|Z_t|^2
+
c_x(t)^\top X_t+c(t)
\Big)dt+
Z_t^\top
\Big(
dW_t-\diag\big(c_{z\sqrt{x}}(t)\big)\sqrt{X_t}\,dt
\Big)\Big].
\end{align*}
From now on, we take \(\delta=2c_{zz}\) and we introduce the new probability measure  $\tilde{\mathbb{P}}$,  together with the associated Brownian motion $\widetilde{W}_t$, by setting 
\begin{equation}\label{eq:RN_density}
\frac{d\tilde{\mathbb{P}}}{d\mathbb{P}}|_{\mathcal{F}_t}
=
\mathcal{E}\Big(
\int_0^t
\big(\diag(c_{z\sqrt{x}}(s))\sqrt{X_s}\big)^{\top}\,dW_s
\Big),\quad \widetilde{W}_t
=
W_t-\int_0^t
\diag\big(c_{z\sqrt{x}}(s)\big)\sqrt{X_s}\,ds.
\end{equation}
By the continuity of \(c_{z\sqrt{x}}\) together with the finiteness of the Laplace transform of $X$ (see e.g.~\cite[Theorem A.1.]{dro2026optimal}), Novikov's condition is satisfied. Therefore, the stochastic exponential above, which defines the Radon-Nikodym density  \(\mathcal{F}_t\), is a true martingale.
Hence, $\tilde{\mathbb{P}}$
 is well-defined, and by Girsanov's theorem, $\widetilde{W}$
 is a standard Brownian motion under $\tilde{\mathbb{P}}$.
Moreover, with this choice of \(\delta\), the quadratic terms \(Z_t^\top Z_t\) in the above SDE satisfy by \(\Gamma_t^\delta\) cancel out. As a result, the dynamics of the process \(\Gamma_t^{2c_{zz}}\) under $\tilde{\mathbb{P}}$ can be written as follows:
\begin{align*}
	d\Gamma_t^{2c_{zz}}
	&= 2c_{zz} \Gamma_t^{2c_{zz}}\Big[-\big(c(t) + c_x(t)^\top X_t \big)dt + Z_t^\top d\widetilde{W}_t\Big]
\end{align*}
Define the process \(M_t:=\Gamma_t^{2c_{zz}}\exp\Big(2c_{zz}\int_0^t\big( c(s) + c_x(s)^\top X_s \big) ds\Big) =\mathcal{E}\Big(2c_{zz}\int_{0}^{t} Z_s^\top d\widetilde{W}_s\Big), \; t \leq T\) by It\^o's Lemma. Under our assumption, \(M\) is a true $\tilde{\mathbb{P}}$- martingale. Now, as $\Gamma_T=\exp\big(\int_0^T \Theta^\top X_s ds + \Xi \big)$, writing $\E[M_T|\mathcal F_t]=M_t$, we obtain the representation below for the auxiliary process 
\begin{equation}\label{eq:Markowitz_gamma}
\Gamma_t:=\Big(\E^{\tilde{\mathbb{P}}}\Big[\exp\Big(2c_{zz}\Big(\int_0^T \Theta^\top X_s ds + \Xi+ \int_t^T  \big(c(s) + c_x(s)^\top X_s \big) ds\Big) \Big)|\mathcal{F}_t\Big]\Big)^{\frac{1}{2c_{zz}}}.
\end{equation}
Note that, this approach is somewhat parallel in spirit, to the idea underlying the Feynman-Kac representation theorem for linear PDEs under a probability measure.
We write for \( \; 0\leq t\leq T\):
\begin{align*}
	\Gamma_t^{2c_{zz}} &=\E^{\tilde{\mathbb{P}}} \Big[ \exp\Big(2c_{zz}\Big(\int_0^T \Theta^\top X_s ds + \Xi+\int_t^T \big(c(s) + c_x(s)^\top X_s \big) ds\Big)\Big) \mid \mathcal F_t\Big]\\
	&= \exp\Big(2c_{zz}\Big(\int_0^t \Theta^\top X_s ds + \Xi+\int_t^T c(s)ds\Big)\Big)\E^{\tilde{\mathbb{P}}} \Big[ \exp\Big(2c_{zz}\int_t^T \big(\Theta+c_x(s)\big)^\top X_s ds\Big) \mid \mathcal F_t\Big],
\end{align*}
which ensures that $\Gamma_t>0$ $\P-a.s.$, since $X_t$ is non-negative ($X\in \mathbb R^d_+$) and $c(t)$ is deterministic. An application of the exponential-affine transform formula in ~\cite[Theorem A.1.]{dro2026optimal} with \(\mathcal{M} \ni m(\dd s) := 2c_{zz}\big(\Theta+c_x(s)\big)\,ds \) yields:  
\begin{equation}
	\E^{\tilde{\mathbb{P}}} \Big[ \exp\Big(\int_t^T 2c_{zz} \sum_{i=1}^d \big(\Theta_i+c_x^i(s)\big) X^i_s ds\Big) \mid \mathcal F_t\Big] = \exp\Big( \sum_{i=1}^d\int_t^T  \big(2c_{zz} \big(\Theta_i+c_x^i(s)\big)  + \tilde{F}_i(s,\tilde{\psi}(T-s))\big) g^i_t(s) ds \Big)
\end{equation}
where $g=(g^1,\ldots,g^d)^\T$, given as in~\eqref{eq:processg}  denotes the adjusted forward process and 
\begin{equation}\label{eq:tildepsi1}\tilde{F}_i(s,\tilde{\psi}) = c^i_{z\sqrt{x}}(s) \sigma^i(s) \tilde{\psi} + (D^\top \tilde{\psi})_i + \frac {1} 2 (\sigma^i(s)\tilde{\psi}^i)^2, \quad i=1,\ldots,d,
\end{equation}  where 
$\tilde{\psi}$ assumed in $C([0,T],\R^d \setminus{\{0_d\}})$ solves the inhomogeneous Ricatti-Volterra equation 
\begin{equation}\label{eq:tildepsi2}
	\tilde{\psi}^i(t)= \int_0^t K_i(t-s)\big(2c_{zz}\big(\Theta_i+c_x^i(T-s)\big)  + \tilde{F}_i(T-s,\tilde{\psi}(s))\big) ds, \quad i=1,\ldots,d.
\end{equation} 
Setting \(\tilde{\psi}= 2c_{zz}\psi \) implies that \(\tilde{F}_i(s,\tilde{\psi}) = 2c_{zz} F_i(s,\psi) \quad i=1,\ldots,d\), where \(F_i\) is given by~\eqref{eq:RiccatiMarkowitzpsi2},
so that 
$\psi\in C([0,T],\R^d\setminus{\{0_d\}})$ solves the inhomogeneous Ricatti-Volterra equation~\eqref{eq:RiccatiMarkowitzpsi1}.
Therefore, it holds that for all \(t\in[0,T]\),
\begin{equation*}
	\E^{\tilde{\mathbb{P}}} \Big[ \exp\Big(\int_t^T 2c_{zz}\sum_{i=1}^d \big(\Theta_i+c_x^i(s)\big) X^i_s ds\Big) \mid \mathcal F_t\Big] = \exp\Big( 2c_{zz} \sum_{i=1}^d\int_t^T  \big(\Theta_i+ c_x^i(T-s)  + F_i(s,\psi(T-s))\big) \tilde{g}^i_t(s) ds \Big)
\end{equation*}
Consequently, ~\eqref{eq:Markowitz_gamma} can be computed in semi-closed form and becomes 
\begin{equation}\label{eq:Markowitz_gamma_all}
	\Gamma_t = \exp\Big( \int_0^t \Theta^\top X_s ds + \Xi +\int_t^T c(s) ds +  \sum_{i=1}^d\int_t^T  \big(c_x^i(s)  + F_i(s,\psi(T-s))\big) g^i_t(s) ds \Big),
\end{equation}
where  $\psi\in C([0,T],\R^d\setminus{\{0_d\}})$ solves the inhomogeneous Ricatti-Volterra equation ~\eqref{eq:RiccatiMarkowitzpsi1}-~\eqref{eq:RiccatiMarkowitzpsi2}.
This yields the first equation in~\eqref{eq:Solbsde} and by standard computation the result that the dynamics of $\Gamma$ is  given by 
\begin{align*}
	d\Gamma_t &=  \Gamma_t \Big[ \Big(\Theta^\top X_t -c(t)  -  \sum_{i=1}^d \big( c_x^i(t) + c_{zz}\cdot(\sigma^i(t)\psi^i(T-t))^2 \big) X^i_t \Big)dt +\sum_{i=1}^d \psi^i(T-t)\sigma^i(t) \sqrt{X^i_t}d\widetilde{W}^i_t \Big]. 
\end{align*}
so that by identification, we may take \(Z^i := \sigma^i(t) \psi^i(T-t) \sqrt{X^i_t}\; \) for every \(i=1,\ldots,d\) with \( 0 \leq t \leq T\) and the proof is completed. \hfill $\Box$.

\medskip
\noindent {\bf Remark:}
 As \(X\in\mathbb{R}_+^d\), one easily checks that if \(\Theta + c_x(t) \in\mathbb{R}_-^d \), then the pair $\left(\Gamma,Z\right)$  is a ${ \mathbb{S}^{\infty}_{\F}([0,T], \R_+^*)}\times L^{2, loc}_{\F}([0,T], \R^d) $-valued solution to the quadratic BSDE~\eqref{eq:GammaDynamicsGeneral} in which \(c_y\equiv g_{z\sqrt{x}}\equiv g_y\equiv g_x\equiv g\equiv g_\nu\equiv\pi\equiv0\), provided solvability in \(C([0,T],\R^d\setminus{\{0\}})\) of the inhomogeneous Ricatti-Volterra equation ~\eqref{eq:RiccatiMarkowitzpsi1}-~\eqref{eq:RiccatiMarkowitzpsi2} (see Theorem~\ref{Thm:Riccatilocalexistence} in Appendix~\ref{app:SolRiccati}). 
In fact, if \(\Theta + c_x(t) \in \mathbb{R}_-^d \) for every \(t\in[0,T]\), then by~\eqref{eq:Markowitz_gamma} \(0 < \Gamma_t
\leq
\exp\!\big(
\int_t^T c(s)\,ds + \Xi
\big),
\; \mathbb{P}\text{-a.s.},\)
since \(X\in\mathbb{R}_+^d\) by assumption. As for \(Z\), see the proof of Theorem~\ref{Thm_BDSEJ_Gen} further on.

\subsection{Intuition from the pure Volterra jumps case}\label{subsect:Intuition_jumps}
In this section we consider the case of a pure-jump affine Volterra processes i.e $\sigma \equiv 0$ thus $Z\equiv0$.
As a second toy example, we consider a time-dependent BSDE without a Brownian component, but with the same terminal condition. This framework includes, in particular, the Volterra extension of the so-called Barndorff--Nielsen--Shephard model introduced in \cite{BarndorffNielsenShephard2001}. Such models have found important applications in optimal portfolio selection; see, for instance, \cite{BenthKarlsenReikvam2003,richter2014explicit} and the references therein.
To give an intuition for the jump part solution $U$ we take
equation~\eqref{eq:BSDEJ_general} with $c_y\equiv g_y\equiv0$, $g_{z\sqrt{x}}\equiv g_x\equiv0$ and $g\equiv g_\nu\equiv h$ in equation~\eqref{eq:ExpCom_functs} for all $t\leq T$. This simplifies the BSDE with jumps into:
\begin{equation}\label{eq:bsde_mv_jumps}
\left\{
\begin{array}{ccl}
dY_t
&=&
-\Big(c(t) + c_x(t)^\top X_t 
 +\int_E h\big(U_t(e)\big)\pi(X_t,de) 
\Big)dt + \int_E
U_t(e)\tilde N(dt,de),
\\
Y_T&=&F\big(\int_0^T X_s ds\big).
\end{array}
\right.
\end{equation}
  We have the following proposition
	\begin{Proposition} 
		\label{prop:MV_purejump_riccati} Let the assumptions of Proposition~\ref{prop:Transform} be in force. Assume moreover that there exists a solution $\psi \in C([0,T],\R^d)$ to the nonlinear Volterra integral equation with L\'evy jump term (compensated exponential jump term) for $i = 1,\ldots,d$:
		\begin{align}
			\psi^i(t) &= \int_0^t K_i(t-s)
			\big(\Theta_i + c_x^i(T-s) + F_i(\psi(s))\big)\,ds,
			\quad
			\label{eq:Riccatipsi_pureJ1}\\[6pt]
			F_i(\psi) &= (D^\top\psi)_i
			+ \int_{E} h\!\big(\psi^\top\eta(e)\big)\nu_i(de).
			\label{eq:Riccati_pureJ2}
		\end{align}
		Then, the couple $\left(Y,U\right)$ 
        defined as 
		\begin{equation}\label{eq:Solbsde2} 
\left\{
\begin{aligned}
    Y_t
    &=
    F\big(\int_0^t X_s\,\mathrm{d}s\big)
    + \int_t^T c(s)\,\mathrm{d}s \\
    &\quad
    + \int_t^T
    \Big(
        \int_E h\big(\psi(T-s)^\top\eta(e)\big)\,\nu_0(\mathrm{d}e)
        + \sum_{i=1}^d
        \big(
            \Theta_i + c_x^i(s)
            + F_i\bigl(\psi(T-s)\bigr)
        \big)
        g_t^i(s)
    \Big)\,\mathrm{d}s, \\
    U_t(e)
    &=
    \psi(T-t)^\top\eta(e),
    \qquad 0\leq t\leq T.
    \end{aligned}
    \right.
\end{equation}
		where $g$ $=$ $(g^1,\ldots,g^d)^\T$ is the $\R^d$-valued process \((g_t(s))_{t\leq s}\) given by \eqref{eq:processg}, is a  
        $\mathbb{R}\times\mathbb{R}$-valued solution to the non-linear BSDEJ~\eqref{eq:bsde_mv_jumps}.
    \end{Proposition}
\begin{proof}
Similarly to the pure diffusion case, we apply the distortion technique, which consists of linearizing the driver w.r.t. $U$. Let $\delta \in \mathbb{R}$, by applying It\^o's formula to $\Gamma^\delta_t := \exp(\delta Y_t)$ we have
\begin{align*}
d\Gamma^\delta_t
&= \Gamma^\delta_{t-}
\left[
\delta\, dY^c_t
+ \int_E \bigl(e^{\delta U_t(e)}-1\bigr)\,N(dt,de)
\right] 
\\&= \Gamma^{\delta}_{t_-} \Big[-\delta \Big(c_x^\top X_t + c(t) + \int_E \big(h(U_t(e)) - h_{\delta}(U_t(e))\big)\pi(X_t, de)\Big) dt + \int_E \bigl(e^{\delta U_t(e)}-1\bigr)\,\tilde{N}(dt,de)\Big].
\end{align*}
Taking $\delta=1$ and noticing that, by assumption, the local martingale defined for every $ t\in[0,T]$ by \(M_t
:=
\mathcal{E}\!\left(
\int_0^t \int_E
\bigl(e^{ U_s(e)}-1\bigr)
\,\tilde{N}(ds,de)
\right)\) is a true $\mathbb{P}$-martingale, we deduce from equation~\eqref{eq:ito_gamma} of Proposition~\ref{prop:Transform} the following  representation
\begin{align*}
\Gamma_t
&=
\mathbb{E}
\!\Big[
\exp\!\Big(F\big(\int_0^T X_s ds\big)+
\int_t^T \bigl(c(t) + c_x(t)^\top X_t \bigr)\,ds
\Big)
\Bigm| \mathcal{F}_t
\Big] \\
&=
\exp\!\Big(\int_0^t \Theta^\top X_s ds + \Xi+\int_t^T c(s)ds
\Big)
\mathbb{E}
\!\Big[
\exp\!\Big(
\int_t^T \sum_{i=1}^d \big(\Theta_i+c^i_x(s)\big) X_s^i\,ds
\Big)
\Bigm| \mathcal{F}_t
\Big].
\end{align*}
Calling upon the exponential affine representaion~\cite[Theorem A.1.]{dro2026optimal} and following the same computations as in the purely diffusion case (see Section~\ref{subsect:Intuition_Brownian}), we obtain the desired expression for $Y$. Furthermore, by differentiating this expression and identifying the $\tilde{N}$-martingale part, we deduce by uniqueness that
\[
U_t(e)
=
\psi^\top(T-t)\,\eta(e),
\qquad t\in[0,T].
\]

This completes the proof.
\end{proof}
\subsection{General case and main result: A verification argument}\label{subsect:Gen_Case}
\noindent
We are now in a position state the main theorem, which provides an explicit representation of the solution processes
$(Y,Z,U)$ in terms of the solution to a system of generalized inhomogeneous Riccati-Volterra equations.
This contrasts with \cite[Theorem 3.5]{richter2014explicit}, where the solution processes are characterized via a system of standard Riccati ODEs.
It follows from Sections~\eqref{subsect:Intuition_Brownian} and~\eqref{subsect:Intuition_jumps} that \(Z\) and \(U\) must be proportional to the corresponding expressions given in equations~\eqref{eq:Solbsde} and~\eqref{eq:Solbsde}, respectively. The proof will then be completed by verification.

\begin{Theorem}[Explicit solution to the BSDEJ]\label{Thm_BDSEJ_Gen}
Let $X$ be an affine Volterra process $X$
associated to the admissible
parameter set $(b,a,\pi)$ in~\eqref{eq:affineProcess}. 
Define the functions:
\begin{equation}
\label{eq:Funct_General}
\begin{aligned}
    \bar{c}_y(s,\psi)
	&= c_y(s) + \int_{E} g_y(s, \psi^\top\eta(e)) \nu_0(de)
    \\[1mm]
F_i(s,\psi)
	&=  c^i_{z\sqrt{x}}(s) \sigma^i(s) \psi_i + (D^\top \psi)_i + e^{-\int_0^s \bar{c}_y(u,\psi)du} c^i_{zz} (s) \cdot \big(\sigma^i(s)\psi^i\big)^2 \\[1mm]
      &+ 
	  \int_E \Big( e^{\int_0^s \bar{c}_y(u,\psi)du}\,g^i_x\big(s,e^{-\int_0^s \bar{c}_y(u,\psi)du}\psi^\top\eta(e)\big) + g^i_{z\sqrt{x}}\big(s,e^{-\int_0^s \bar{c}_y(u,\psi)du}\psi^\top\eta(e)\big)\sigma^i(s) \psi_i \Big)\nu_0(de) \\[1mm]
	&\hspace{3cm}+ 
	  e^{\int_0^s \bar{c}_y(u,\psi)du}\int_E g_\nu\big(s,e^{-\int_0^s \bar{c}_y(u,\psi)du}\psi^\top\eta(e)\big)\; \nu_i(de), \quad i=1,\ldots,d,
	\\[1mm]
	G(s,\psi)
	&= e^{\int_0^T \bar{c}_y(u,\psi)du}\Theta
	+ e^{\int_0^s \bar{c}_y(u,\psi)du}c_x(s)
	+ F(s,\psi),\quad F:=(F_1, \ldots,F_d)^\top, \quad (s,\psi)\in[0,T]\times\mathbb{R}^d.
\end{aligned}
\end{equation}
Assume that there exists  a unique solution $\psi \in C([0,T],\R^d)$ to the generalized inhomogeneous integral
Riccati-Volterra equation with L\'evy jump compensator
\begin{equation}\label{eq:RiccatiGeneral} 
    \psi(t) = \int_0^t K(t-s) G(T-s,\psi(s)) ds, 
\end{equation}
Furthermore let $\phi$ the $\R$-valued function given by
\begin{equation}\label{eq:phi_ODE}
\begin{aligned}
\phi(t)
&= \Xi +  \int_t^T \Big( \bar{c}_y(s, \psi(T-s))\phi(s) +c(s) + \int_{E} g(s,e^{-\int_0^s \bar{c}_y(u,\psi(T-u))du}\psi(T-s)^\top\eta(e)) \nu_0(de)\Big)\,ds,
\end{aligned}
\end{equation}
for every $t\in[0,T]$.
Then the above BSDEJ~\eqref{eq:BSDEJ_general} has a unique $\mathbb{D}([0,T], \R) \times L^2_{\F}([0,T], \R^d)\times L^2_{\F}([0,T], \tilde{N})$-valued solution triplet $\left(Y,\Lambda, U\right)$
, reading for every $0 \leq t \leq T$
\begin{equation}\label{eq:BSDEJ_solGen}
\left\{
\begin{array}{ccl}
Y_t
&=& e^{\int_t^T \bar{c}_y(u,\psi(T-u)) du}\int_0^t \Theta^\top X_s ds +
 e^{-\int_0^t \bar{c}_y(u,\psi(T-u)) du} \int_t^T G(s,\psi(T-s))^\top  g_t(s) ds
+\phi(t),
\\[1mm]
Z_t
&=& e^{-\int_0^t \bar{c}_y(u,\psi(T-u)) du}\sqrt{\diag(X_t)}\,\sigma^\top(t)\,\psi(T-t),
\\[1mm]
U_t(e)
&=& e^{-\int_0^t \bar{c}_y(u,\psi(T-u)) du} \psi^\top(T-t) \eta(e), \qquad
\forall\,e\in E.
\end{array}
\right.
\end{equation}
where the $\R^d$-valued process \((g_t(s))_{t\leq s}\), $g$ $=$ $(g^1,\ldots,g^d)^\T$ denotes the
adjusted conditional $\mathbb{P}$-expected forward process defined by~\eqref{eq:processg}--\eqref{eq:Condprocessg}.
\end{Theorem}
 
\noindent{\bf Remark:} 1. It is worth noting that the function $\phi$ in~\eqref{eq:phi_ODE} admit the closed-form variation-of-constants solution:
\begin{equation*}
\begin{aligned}
\phi(t)
=&  \int_t^T e^{\int_t^s \bar{c}_y\big(s,\psi(T-u)\big) du} \Big( c(s) + \int_{E} g(s,e^{-\int_0^s \bar{c}_y(u,\psi(T-u)) du}\psi(T-s)^\top\eta(e)) \nu_0(de)\Big)\,ds\\[1mm]
&\hspace{2cm}+ \Xi \, e^{\int_t^T \bar{c}_y(u,\psi(T-u)) du}.
\end{aligned}
\end{equation*}
\noindent 2. Equation~\eqref{eq:BSDEJ_solGen} can be reformulated by setting \(\widetilde{\psi}(T-t):=e^{-\int_0^t \bar{c}_y(u,\psi(T-u)) du}\psi(T-t)\) for every \(t\in[0,T]\), where $\widetilde{\psi}$ is the unique solution to a well-defined riccati Volterra equation with non-convolution kernel in general ($K(t,s):=e^{-\int_s^t c_y(u)\,du}
K(t-s)$).
In the Markovian case $K\equiv 1_{d\times d}$, the corresponding equation reduces to an ordinary Riccati equation, and the usual variation-of-constants transformation allows the $(\bar{c}_y)$-term to enter linearly in the Riccati equation~\eqref{eq:Funct_General}--\eqref{eq:RiccatiGeneral}, as in the standard ODE setting (see e.g.~\cite[Theorem 3.5.]{richter2014explicit}).

\noindent 3.  If \(g_{z\sqrt{x}}\equiv g_{x}\equiv g_{\nu}\equiv g_{y} \equiv0\) and \(\nu_i\equiv 0\) for every \(i=1,\ldots,d\) in Equation~\eqref{eq:Funct_General}, then the existence of a unique  non-continuable  solution $\psi \in C([0,T_{max}),\R^d)$ to the deterministic Volterra equation~\eqref{eq:RiccatiGeneral} under suitable assumption on the Kernel \(K\) is given by Theorem~\ref{Thm:Riccatilocalexistence}.

\medskip
\noindent{\bf Proof of Theorem \ref{Thm_BDSEJ_Gen}:}
Let $\psi$ and $\phi$ be the unique solutions of~\eqref{eq:RiccatiGeneral} and ~\eqref{eq:phi_ODE}. We apply It\^o-L\'evy formula for
semimartingales to the map \(t\to Y_t\) using the representation~\eqref{VolSqrt2}. 
We recall that $Y_t = A_t + H_t + \phi(t),$ for every \(0\leq t\leq T\) with
\begin{equation*}
     A_t:= e^{\int_t^T \bar{c}_y(u,\psi(T-u)) du}\int_0^t \Theta^\top X_s ds\quad\text{and}\quad H_t:= e^{-\int_0^t \bar{c}_y(u,\psi(T-u)) du} \int_t^T G(s,\psi(T-s))^\top  g_t(s) ds.
\end{equation*}
One checks that:
\begin{equation*}
\left\{
\begin{array}{ccl}
d A_t
&=& - \bar{c}_y(t,\psi(T-t)) A_t dt + e^{\int_t^T \bar{c}_y(u,\psi(T-u)) du} \Theta^\top X_t dt,
\\[1mm]
d H_t
&=&  - \bar{c}_y(t,\psi(T-t)) H_t dt +
 e^{-\int_0^t \bar{c}_y(u,\psi(T-u)) du} \big(-G(t,\psi(T-t))^\top  g_t(t) \\[1mm]
 &+& \int_t^T G(s,\psi(T-s))^\top  dg_t(s) ds\big).
\end{array}
\right.
\end{equation*}
Therefore, the dynamics of \( Y \) can readily be obtained by recalling \( g_t(s) \) from ~\eqref{eq:processg} and by observing that for fixed \( s \), the dynamics of \( t \to g_t(s) \) are given for every \(t \leq s\) by \(	dg_t(s) = K(s-t) \, dZ_t,\) where \(dZ_t=\big(D X_t dt + \sigma(t) \sqrt{\diag(X_t)}dW_t + \int_{E}\eta(e)\tilde{N}(dt,de) \big)\).
Since \( g_t(t) = X_t \), invoking the stochastic Fubini theorem in \cite[Theorem 2.2]{Veraar2012},
shows that the dynamic or the semimartingale decomposition of $Y$ reads as:
\begin{align}
    &\, dY_t = - \bar{c}_y(t,\psi(T-t)) \big(A_t + H_t + \phi(t)\big) dt + \Big(e^{\int_t^T \bar{c}_y(u,\psi(T-u)) du} \Theta^\top X_t-c(t)\nonumber\\[1mm]
    &\hspace{.9cm}-e^{-\int_0^t \bar{c}_y(u,\psi(T-u)) du} G(t,\psi(T-t))^\top  X_t - \int_{E} g(t,e^{-\int_0^t \bar{c}_y(u,\psi(T-u))du}\psi(T-t)^\top\eta(e))\nu_0(de)\Big) dt\nonumber\\[1mm]
    &\hspace{.9cm} 
    + e^{-\int_0^t \bar{c}_y(u,\psi(T-u)) du} \int_t^T G(s,\psi(T-s))^\top  K(s-t) ds\, dZ_t \nonumber\\[1mm]
    &\hspace{.6cm}= - \Big( c(t)  +   \bar{c}_y(t,\psi(T-t)) Y_t +  c_x(t)^\top X_t +  X_t^\top \int_E g_\nu\big(t,e^{-\int_0^t \bar{c}_y(u,\psi(T-u))du}\psi(T-t)^\top\eta(e)\big)\,\nu(de)  \nonumber\\[1mm]
    &\hspace{.4cm}+ \int_E g_x\big(t,e^{-\int_0^t \bar{c}_y(u,\psi(T-u))du}\psi(T-t)^\top\eta(e)\big)^\top X_t + g(t,e^{-\int_0^t \bar{c}_y(u,\psi(T-u))du}\psi(T-t)^\top\eta(e)) \nu_0(de)  \nonumber\\[1mm]
    &\hspace{.4cm} + e^{-\int_0^t \bar{c}_y(u,\psi(T-u))du}\Big(\big(D^\top \psi(T-t)\big)^\top X_t +\sum_{i=1}^d X^i_t \big( e^{-\int_0^t \bar{c}_y(u,\psi(T-u))du} c^i_{zz} (t) \cdot \big(\sigma^i(t)\psi^i(T-t)\big)^2 \nonumber\\[1mm]
    &\hspace{.4cm}+ c^i_{z\sqrt{x}}(t) \sigma^i(t) \psi_i(T-t) + \int_E g^i_{z\sqrt{x}}\big(t,e^{-\int_0^t \bar{c}_y(u,\psi(T-u))du}\psi(T-t)^\top\eta(e)\big)\sigma^i(t) \psi_i(T-t) \nu_0(de) \big)  \Big)  \Big)dt \nonumber\\[1mm]
    &\hspace{.4cm}+ e^{-\int_0^t \bar{c}_y(u,\psi(T-u)) du}\psi(T-t)^\top\, \big(D X_t dt + \sigma(t) \sqrt{\diag(X_t)}dW_t + \int_{E}\eta(e)\tilde{N}(dt,de) \big) \label{eq:DynmY_}
    \end{align}
    Now, noting that:
   \begin{align*}
    &\, \psi(T-t)^\top \sigma(t)\sqrt{\diag(X_t)}
	  e^{-\int_0^t \bar{c}_y(u,\psi(T-u))du} \diag\big(c_{zz}(t)\big)\sqrt{\diag(X_t)}\sigma^\top(t)\psi(T-t)\nonumber\\[1mm]
    &\hspace{.9cm}=\sum_{i=1}^d e^{-\int_0^t \bar{c}_y(u,\psi(T-u))du} c^i_{zz} (t) \cdot \big(\sigma^i(t)\psi^i(T-t)\big)^2\,  X^i_t
    \end{align*}
    which coincides with the first term in the sum appearing in the above equality~\eqref{eq:DynmY_}. Analogously, the remaining two terms are identified with the corresponding matrix-product form. 
    Consequently, the dynamics of $Y$ reduce to
   \begin{align}
    &\, dY_t =- \Big( 
    c(t) +  \psi(T-t)^\top \sigma(t)\sqrt{\diag(X_t)}
	  e^{-2\int_0^t \bar{c}_y(u,\psi(T-u))du} \diag\big(c_{zz}(t)\big)\sqrt{\diag(X_t)}\sigma^\top(t)\psi(T-t) \nonumber\\[1mm]
    &\hspace{.9cm}+  c_x(t)^\top X_t + e^{-\int_0^t \bar{c}_y(u,\psi(T-u)) du} \psi(T-t)^\top \sigma(t)\sqrt{\diag(X_t)} \diag\big(c_{z\sqrt{x}}(t)\big)\sqrt{X_t} \nonumber\\[1mm]
    &\hspace{.9cm}+  X_t^\top \int_E g_\nu\big(t,e^{-\int_0^t \bar{c}_y(u,\psi(T-u))du}\psi(T-t)^\top\eta(e)\big)\,\nu(de) + \int_{E} g_y(t, \psi(T-t)^\top\eta(e))Y_t \,\nu_0(de) \nonumber\\[1mm]
       &\hspace{.9cm} + \,c_y(t) Y_t + 
	  \int_E \Big(\psi(T-t)^\top\sigma(t)\sqrt{\diag(X_t)}\diag\big(g_{z\sqrt{x}}\big(t,e^{-\int_0^t \bar{c}_y(u,\psi(T-u))du}\psi(T-t)^\top\eta(e)\big) \big)\sqrt{X_t}\nonumber\\[1mm]
	&\hspace{.9cm}+ g_x\big(t,e^{-\int_0^t \bar{c}_y(u,\psi(T-u))du}\psi(T-t)^\top\eta(e)\big)^\top X_t + g(t,e^{-\int_0^t \bar{c}_y(u,\psi(T-u))du}\psi(T-t)^\top\eta(e)) \Big)\nu_0(de) \nonumber\\[1mm]
    &\hspace{.2cm} + \underbrace{e^{-\int_0^t \bar{c}_y(u,\psi(T-u)) du} \psi(T-t)^\top \sigma(t)\sqrt{\diag(X_t)}}_{=Z_t^\top}dW_t+  \int_{E} \underbrace{e^{-\int_0^t \bar{c}_y(u,\psi(T-u)) du} \psi(T-t)^\top\eta(e) }_{=U_t(e)}\tilde{N}(dt,de),\label{eq:DynmY}
\end{align}
where we recognize the expressions for \(Z\) and \(U\) given by formulas~\eqref{eq:BSDEJ_solGen}.
Hence the BSDEJ~\eqref{eq:BSDEJ_general}--\eqref{eq:BSDEJ_general} is solved by~\eqref{eq:BSDEJ_solGen} provided the finite variation parts of~\eqref{eq:DynmY} and the
BSDEJ coincide. Now, owing to the special form of the generator \(f\) in~\eqref{eq:BSDEJ_general}, together with the expressions of \(Z\) and \(U\) in~\eqref{eq:BSDEJ_solGen}, it boils down that
\begin{equation}
dY_t=- f(t,X_t, Y_t, Z_t, U_t) + Z_s^{\top}\,dW_s + \int_{E}U_s(e)\tilde{N}(ds,de)
\end{equation}
As for the terminal condition, it is straightforward that $Y_T
=
\int_0^T \Theta^\top X_s\,ds+\phi(T) = \Xi+\int_0^T\Theta^\top X_s\,ds= F\big(\int_0^T X_s\,ds\big)$ owing to~\eqref{eq:TermCond}.

\smallskip\noindent
It remains to show that, under our assumption $\left(Z, U \right)\in L^2_{\F}([0,T], \R^d) \times L^2_{\F}([0, T],\tilde{N})$. For \(Z\), it is clear that its belongs to \(L^2_{\F}([0,T], \R^d)\) since 
\(c_y\), \(\sigma\) and \(\psi\) are countinuous, thus bounded on \([0,T]\) and \(\mathbb{E}\!\Big[
\int_0^T
|X_s|\,ds
\Big]
<\infty\)
by the inequality~\eqref{eq:SupMoment_X} .
Likewise, $U \in L^2_{\F}([0, T],\tilde{N})$ due to the boundary property of $\psi$ and integrability condition in \eqref{assump:kernelJumpVolterra}~$(ii)$.\\
Now, from~\eqref{eq:BSDEJ_general}, $t\mapsto Y_t$ has c\`adl\`ag paths (up to indistinguishability or modification) on the compact interval \([0,T]\).
Moreover, since $U$ is deterministic, the L\'evy-integrated $g$-coefficients
evaluated at $U$ are deterministic functions of time.
 Under the continuity assumptions on the coefficients, together with Young's inequality, there exists a constant $C>0$ such that
 \begin{equation}
   \forall\, t\in [0,T],\qquad |f(t,X_t,Y_t,Z_t,U_t)|
\leq C\Big(
1+|X_t|+|Y_t|+|Z_t|^2
\Big).
\end{equation}
Hence, since $Z\in L^2_{\mathbb F}([0,T],\mathbb{R}^d)$ and \(\mathbb{E}\!\left[\int_0^T |X_s|\,ds\right]<\infty,\) owing to~\eqref{eq:SupMoment_X} i.e.  $X\in L^1([0,T],\mathbb{R}^d)$ pathwise  
while $Y$ is c\`adl\`ag and hence bounded on $[0,T]$ almost surely, we obtain \(\mathbb{E}\!\Big[
\int_0^T
\big(
1+|X_s|+|Y_s|+|Z_s|^2
\big)\,ds
\Big]<\infty\) 
and therefore $\int_0^T
|f(s,X_s,Y_s,Z_s,U_s)|\,ds<\infty
\;\text{a.s.}$
Consequently,
the BSDEJ~\eqref{eq:BSDEJ_general}--\eqref{eq:BSDEJ_general} is solved by~\eqref{eq:BSDEJ_solGen} in the sense of Definition~\ref{def:Solut_bsde}. 
This completed the Proof. \hfill $\Box$

\section{Applications to continuous-time Mean-variance optimization problem under multivariate affine Volterra models with jumps}\label{sect-appl}
In this section, we first provide an overview of the probabilistic market model. We then introduce the Markowitz portfolio optimization problem and derive a candidate optimal portfolio strategy in section~\ref{Sec:Markowitz}. 
Finally, in Section~\ref{Verif_result}, we provide an explicit solution to the problem under a more general multidimensional correlation structure between asset returns and their respective volatilities, as in \cite[Theorem~4.2]{HanWong2020a}, where the analysis is carried out in the one-dimensional case ($d=1$). The optimal value will be characterized by the solution of an associated Riccati BSDEJ whose generator is quadratic in the Brownian stochastic integrand (denoted by \(\Lambda\)) and with an exponential term involving the Poisson stochastic integrand variable (denoted by \(U\)). 

\subsection{Formulation of the stochastic Market model}\label{subsect:Market_model}
We consider the model introduce in~\cite{dro2026optimal}, where the volatility incorporates roughness and memory effects as well as both Brownian and Poisson sources of randomness, as considered in~\cite{BondiLivieriPulido2024,AlfonsiSzulda2025}, while the underlying asset prices remain continuous, in the spirit of~\cite{HanWong2020a,AbiJaberMillerPham2021, Gnabeyeu2026b}.
Fix $T > 0$, $d\in \N$. Let a frictionless financial market on $[0,T]$  
with \(d+1\) securities, consisting of a  bond and \(d\) stocks. The non--risky asset  $S^0$ satisfies the (stochastic) ordinary differential equation: 
\begin{align*}
	dS^0_t = S^0_t r(t) dt,
\end{align*}
with a time-dependent deterministic  short risk-free rate $r:\R_+ \to \R$, and  $d$ risky assets (stock or index) whose return vector process $(S_t)_{t \ge 0} = (S_{t}^1, \ldots, S_{t}^d)_{t \ge 0}$ is defined via the dynamics given by the vector-stochastic differential equation (SDE):
\begin{align}
	\label{eq:stocks}
	dS_t = \diag(S_t) \big[ \big( r(t) {\bold{1}_d} + \sqrt{\diag(V_t)} \lambda_t  \big)dt + \sqrt{\diag(V_t)}dB_t \big],
\end{align}
driven by a $d$-dimensional Brownian motion $B$, with a $\R^d$-valued continuous stochastic process $\lambda$, 
called {\it market price of risk} and a $d\times d$-matrix valued  continuous stochastic volatility process $\sqrt{\diag(V)}$ whose dynamics is driven by $V=(V^1,\ldots, V^d)^\top$,  a $\R^d_+$--valued scaled Volterra square--root process with jumps, itself driven by an $d$-dimensional process $W=(W^1,\ldots,W^d)^\T$ and the compensated random measure $\tilde{N}$:
\begin{equation} 
	\label{VolSqrt_}
	\begin{aligned}
		V_t = \varphi(t) V_0 + \int_0^t K(t-s) \Big( \big(\mu(s) + D V_s\big) ds  + \sigma^v \varsigma(s)\sqrt{\diag(V_s)}dW_s + \int_{E}\eta(e)\tilde{N}(ds,de) \Big), \; V_0\perp\!\!\!\perp W \perp\!\!\!\perp \tilde{N}.
	\end{aligned}
\end{equation}
Here $K=\diag(K_1,\ldots,K_d)$ is a diagonal with scalar kernels $K_i\in L^{2}([0,T],\R)$ on the diagonal, $\varphi=\diag(\varphi^1, \ldots,\varphi^d)$, $\mu : [0,T] \to \mathbb{R}^d_{+}$ a positive function, $\sigma^v=\diag(\sigma^v_1,\ldots,\sigma^v_d)$, $\varsigma=\diag(\varsigma^1,\ldots,\varsigma^d)$ with \(\varsigma^i\) a (locally) positive bounded Borel function, $\eta = (\eta^1,\ldots,\eta^d)^\top$ a measurable map $\eta : E \to \mathbb{R}^d$ and $D$ a $\R^{d\times d}$-valued matrix satisfying
\[
D \in \mathbb{R}^{d \times d}, 
\qquad 
D_{ij} \ge 0 \ \text{for } i \neq j.
\]
The correlation structure of $W$ with $B$ is given by
{\begin{align}\label{eq:correstructureheston}
		W^i = \rho_i B^i + \sqrt{1-\rho_i^2} B^{\perp,i} =  \Sigma_i^\top  B_t + \sqrt{1-\Sigma_i^\T \Sigma_i} B^{\perp,i}_t, \quad i=1,\ldots,d,
	\end{align}
	for some $(\rho_1,\ldots,\rho_d)\in[-1,1]^d$},
where $(0,\ldots,\rho_i,\ldots,0)^\T:=\Sigma_i \in \R^{d}$ is such that $\Sigma_i^\T \Sigma_i\leq1$,  and $B^{\perp}$ $=$ $(B^{\perp,1},\ldots,B^{\perp,d})^\T$ is an $d$--dimensional Brownian motion independent of $B$. The correlation $\rho_i $ between stock price \(S^i\) and variance \(V^i\) is assumed constant.
Note that $d\langle W^i \rangle_t = dt$  but $W^i$ and $W^j$ can be correlated, hence $W$ is not necessarily a Brownian motion.

\medskip
\noindent Observe that 
processes $\lambda$ and $\sigma$ are $\F$-adapted, possibly unbounded,  but not necessarily adapted to the filtration generated by $W$.  We point out that $\F$ may be strictly larger than the augmented filtration generated by $B$ and $B^{\perp}$ as we deal with weak solutions to stochastic Volterra equations. 

\medskip
\noindent
 We will be chiefly interested in the case where \(\lambda_t\) is linear
in \(\sigma_t\). More specifically, the market price of risk (risk premium) is assumed to be in the form
{$\lambda$ $=$ $\big(\theta_1\sqrt{V^1},\ldots,\theta_d \sqrt{V^d}\big)^\top$}, for some constant {$\theta_i \geq 0$},  so that the dynamics  for the stock prices \eqref{eq:stocks} reads following \cite{Kraft2005,abi2019affine}
\begin{align}
	\label{eq:hestonS} 
	dS^i_t = S^i_t \left( r(t)  + \theta_i V^i_t   \right) dt + S^i_t \sqrt{V^i_t} dB^i_t, \quad i=1,\ldots, d.
\end{align}
 Note that our model (\ref{eq:hestonS})-(\ref{eq:correstructureheston})-(\ref{VolSqrt_}) features correlation between the stocks and between a stock and its volatility. 
Since \(S\) is fully determined by \(V\), the existence of $S$ readily follows from that of $V$. In particular, weak existence of c\`adl\`ag solution $V$ of~\eqref{VolSqrt_} satisfying~\eqref{eq:SupMoment_X} is established  under suitable assumptions~\ref{assump:kernelJumpVolterra} on the kernel $K$ and the jump components $(\eta,\pi)$ together with the specifications of the input curve $g_0$. 
In view of Assumption~\ref{assump:kernelJumpVolterra} and Remark~\ref{rm:Kernels}, we may therefore assume that the stochastic Volterra equation (\ref{eq:hestonS})-(\ref{VolSqrt_})  admits a unique in law 
c\`adl\`ag $\R^{d}_+ \times \R^{d}_+$-valued weak solution $(S,V)$ for any initial condition $(S_0, V_0) \in \R^{d}_+ \times L^1_{\R^{d}_+}(\P)$ defined on some filtered probability space $(\Omega,\mathcal F, (\mathcal F)_{t\geq 0}, \mathbb P)$ such that $V$ satisfied~\eqref{eq:SupMoment_X}.

\medskip
\noindent {$\rhd$ {\em Problem formulation}:}  Since we consider weak solutions to the stochastic Volterra equations with jumps \eqref{VolSqrt_}--\eqref{eq:hestonS}, the Brownian motions and the Poisson random measure are part of the solution itself. Moreover, 
the expected terminal wealth and its variance depends only on the law of the wealth process, we may, without loss of generality, fix a weak solution $(S,V,B,B^\top,N)$ provided above, defined on the filtered probability space $(\Omega,\mathcal F,\mathbb F,\mathbb P)$, where $\mathbb F=(\mathcal F_t)_{t\ge0}$ satisfies the usual conditions.
\subsection{Markowitz portfolio selection: Mean-variance optimization problem}\label{Sec:Markowitz}

\medskip
Let $\pi_t= (\pi_{t,1}, \ldots, \pi_{t,d})^\top$ denote the vector of the amounts of money invested in the risky assets $S$ at time $t$ in a self--financing strategy where $\pi_{t,k}$ represents the amount of money invested in asset $k$ at time $t$, and the remaining amount $X^\pi_t-\pi_t^\top {\bold{1}_d}$ in a bond of price $S_t^0$
with interest rate $r(t)$. The notation $X^\pi_t$ emphasizes the dependence of the wealth on the strategy $\pi = (\pi_t)_{t \ge 0}$.
 We assume that the process \((\pi_t)_{t\geq0}\) are progressively measurable. Then, the dynamics of the wealth $X^{\pi}$ of the portfolio is given by 
\begin{align*}
	\mathrm d X^\pi_t
	&=  \bigl(  \pi_t^\top \big( r(t) {\bold{1}_d} + \sigma(V_t) \lambda_t  \big) + (X^\pi_t - \pi_t^\top {\bold{1}_d})r(t) \bigr)\,\mathrm dt + \pi_t^\top\sigma(V_t)  \,\mathrm d B_t \\
	&=  \bigl(r(t) X^\pi_t + \pi_t^\top  \sigma(V_t) \lambda_t\bigr)\,\mathrm dt + \pi_t^\top   \sigma(V_t) \,\mathrm d B_t
\end{align*}
Let $\alpha_t := \sigma^\top(V_t)\pi_t $ be the investment strategy, with $\alpha_t = (\alpha_{t,1}, \ldots, \alpha_{t,d})^\top$, an $\mathbb{R}^d$-valued,  $\mathbb{F}$-progressively measurable process.
By \(\mathcal{A}\) we denote the set of admissible portfolio or investment strategies i.e. the set of all $\mathbb{F}$-progressively measurable processes $ (\alpha_t)_{t \in [0,T]}$ valued in the Polish space  \(\R^d\). Under a fixed portfolio strategy \(\alpha\), the dynamics of the corresponding wealth process $(X_t^\alpha)_{t \ge 0}$ of the portfolio we seek to optimize is given by 
\begin{align}
	\label{eq:wealth}
	dX^{\alpha}_t &= \big( r(t) X^{\alpha}_t  + \alpha_t^\T \lambda_t \big) dt + \alpha_t^\T dB_t, \quad t \geq 0, \quad X_0^\alpha = x_0 \in \R. 
\end{align} 
\noindent 
The goal is to determine the control \( \alpha(\cdot) \) that maximizes the expected value of a certain cost functional to be specified latter, which accounts for the terminal cost.
\noindent By a solution to~\eqref{eq:wealth}, we mean an $\F$-adapted continuous process $X^{\alpha}$ satisfying \eqref{eq:wealth} on $[0,T]$ $\P$-a.s. (that is the wealth process~\eqref{eq:wealth} has a unique solution, on $[0,T]$ with $\P$-a.s. continuous sample paths ) and such that 
\begin{align}
	\label{eq:estimateX}
	\E\big[\sup_{t\leq T} |X^{\alpha}_t|^2 \big] &< \;  \infty.
\end{align}
By a standard calculation, the wealth process is then given by
\begin{equation}\label{eq:wealthProcess} 
	X_t = e^{ \int^t_0 r(s) ds} \left( x_0 + \int_0^t e^{ -\int^s_0 r(u) du}  \left(\alpha_s^\T dB_s + \alpha_s^\T \lambda_s \mathrm{d}s \right) \right), 
\end{equation}
Note that it is sufficient to assume that \(\int_0^t (\left|\lambda_s\right|^2 + \left|\alpha_s\right|^2 )\, \mathrm{d}s < +\infty \) almost surely for all \( t \ge 0\)
in order to construct the stochastic integrals in Equation~\eqref{eq:wealthProcess}. This boundedness condition holds owing to the inequality \( a^2 \le 2 e^{a}, \, \forall a \geq 0 \) (take \(a=\int_0^t \left|\lambda_s\right|^2\, \mathrm{d}s\) ), together with the condition introduced later~\eqref{eq:assumption_novikov}, and the following admissibility assumption, which is consistent with \cite{ChiuWong2014,Shen2015,AbiJaberMillerPham2021}. 
\begin{Definition}\label{Def:adm3}
	In the setting described above, we say that an investment strategy $\alpha(\cdot)$ is admissible if
	\begin{enumerate}
		\item[$(a)$] The SDE~\eqref{eq:wealth} for the wealth process $(X_t^{\alpha})$ has a unique solution in terms of $(S,V,B)$ satisfying \eqref{eq:estimateX} with $\p$-$\as$ continuous paths.
		\item[$(b)$] $\alpha(\cdot)$ is progressively measurable and $\int^T_0 \left|\alpha_s \right|^2ds < \infty$, $\p$-$\as$;
	\end{enumerate}
	The set of all admissible controls is denoted as $\cA$ and is naturally defined as the collection of processes $\c$ below:
	$$\mathcal A  = \{ \alpha \in {L^{2,loc}_{\F}([0,T], \R^d)} \mbox{ such that \eqref{eq:wealth} has a  solution satisfying } \eqref{eq:estimateX} \}.$$
\end{Definition}
\noindent By Definition~\ref{Def:adm3}, the wealth process corresponding to \(\alpha \in \mathcal A\) satisfies~\eqref{eq:estimateX} so that both the expectation and the variance of the terminal wealth (\(X_T\)) are well defined.
\noindent The agent's objective is to find a portfolio \(\alpha\) such that \((X^\alpha(\cdot), \alpha(\cdot))\) satisfy ~\eqref{eq:wealth} and the expected terminal wealth satisfies \(\E[X_T] = m\) for some \(m\in \R\), while the risk measured by the variance of the terminal wealth $ \V (X_T)=\E\big[\big( X_T-\E[ X_T ]\big)^2\big]$  is minimized. The constant \(m\in \R\) is the target wealth level at the terminal time \(T\).

\noindent The risk-free investment being possible, as the interest rate process \(r\) is deterministic, the agent can expect a terminal wealth at least \(m_0:=x_0e^{\int_{0}^{T}r(s)ds}\) and hence it is reasonable to restrict \(m\geq m_0\), which was initially introduced by~\cite{LimZhou2002}.

\noindent The Markowitz portfolio selection problem in continuous-time thus consists in solving the following stochastic optimization problem with linear equality constraints parameterized by \(m\)
\begin{align} \label{optimization_problem}
	V(m) &  := \; \inf_{\substack{\c \in \Acal}}\big\{ \V (X_T):   \text{s.t. } \E[X_T] = m \big\}. 
\end{align}
i.e. given some expected return value $m$ $\in$ $\R$.  
Here an optimal portfolio of the problem is called an efficient portfolio, the corresponding \((\V (X_T),m)\) is called an efficient point, and the set of all efficient points is called an efficient frontier when \(m\) goes over \([m_0,+\infty)\).

\medskip
\noindent The MV problem is said to be feasible for
\(m\geq m_0\) if there exists a \(\alpha\in \mathcal A\) that satisfies \(\E[X_T] = m\). The feasibility of our problem is guaranteed for any \(m\geq m_0\) by a slight modification to the proof in~\cite[26, Propsition 6.1]{Lim2004}.

\noindent As the mean-variance problem~\eqref{optimization_problem} is feasible and has a linear constraint and a convex cost functional which is bounded below, it follows from the Lagrangian duality theorem~\cite{Luenberger1968} 
that the constrained Markowitz problem \eqref{optimization_problem} is equivalent to the following max-min problem: 

\begin{equation}
	\label{outer_inner_optimization_pb}
	V(m) \; = \; \max_{\zeta \in \R} \min_{\substack{\c \in \Acal}} \Big\{  \E\Big[\big| X^{\c}_T - (m-\zeta) \big|^2\Big] - \zeta^2 \Big\}. 
\end{equation}
This Lagrangian approach essentially moves the expectation constraint to the objective function of the optimization problem with the price to solve the additional outermost maximization problem.
Thus, solving problem \eqref{optimization_problem} involves two steps. First, the internal minimization problem in term of the Lagrange multiplier $\zeta$ has to be solved. Second, the optimal value of $\zeta$ for the external maximization problem has to be determined. Let us then introduce the inner optimization problem i.e., the following quadratic-loss minimization problem:
\begin{equation}
	\label{pb:P(c)}
	\tilde V(\xi) \;  := \;  \min_{\alpha \in \mathcal A}\E\Big[ \big| X^{\alpha}_T-\xi\big|^2\Big], \quad \xi \in \R. 
\end{equation}
where $\xi=m-\zeta$, for $\zeta$ $\in$ $\R$. Now, considering the inner Problem~\eqref{pb:P(c)} with an arbitrary \(\xi \in\R\) and defining $\Tilde{X}^{\alpha}_t = X_t^{\alpha} - \xi e^{-\int_t^T r(s) ds}$, for any $\alpha \in \mathcal{A}$, then, applying It\^o's lemma yields
\begin{equation}\label{eq:Tilwealth}
	d \Tilde{X}^{\alpha}_t =  \big( r(t) \Tilde{X}^{\alpha}_t  + \alpha_t^\T \lambda_t \big) dt + \alpha_t^\T dB_t, \quad 0 \leq t \leq T, \; 
	\Tilde{X}^{\alpha}_0  \; = \;   x_0 - \xi e^{-\int_0^T r(s) ds}. 
\end{equation}
As a result, $\Tilde{X}^{\c}$ and $X^{\c}$ have the same dynamics and $\Tilde{X}^{\c}_T=X^{\c}_T-\xi$ so  that problem \eqref{pb:P(c)} can be alternatively written as
\begin{equation}\label{pb:AlterP(c)}
	\tilde V(\xi) \;  := \;  \min_{\alpha \in \mathcal A}\E\Big[ \big| X^{\alpha}_T-\xi\big|^2\Big] = \min_{\alpha \in \mathcal A}\E\Big[ \big|\Tilde{X}^{\alpha}_T\big|^2\Big], \quad \Tilde{X}^{\alpha}_t = X_t^{\alpha} - \xi e^{-\int_t^T r(s) ds}, \quad \xi \in \R. 
\end{equation}

\medskip
\noindent To make a completion of squares inspired by \citet[Proposition 3.1]{LimZhou2002}, \citet[Proposition 3.3]{Lim2004}, \citet[Theorem 3.1]{ChiuWong2014} and \citet{Shen2015}, we need the auxiliary
process \(\Gamma\) as an additional stochastic factor in a place consistent with previous studies of Mean-Variance portfolios under semimartingales.
Heuristically speaking, the non-Markovian and non-semimartingale
characteristics of the inhomogeneous affine Volterra model with jumps are overcome by considering \(\Gamma\) whose construction is guided by the observations established in this section.
We consider the family of processes $(J^\alpha_t)_{t\in[0,T]}$, $\alpha\in\mathcal{A}$
    \begin{equation}\label{eq:ansatz}
		J^\alpha_t = \Gamma_t
		\left|\tilde{X}^\alpha_t\right|^2,\quad \text{where}\quad \Gamma_t = \exp(Y_t)\quad\text{for}\quad t\in[0,T]
	\end{equation}
	 and the triplet $(Y,\Lambda,U)$, solution under \(\P\), to the following backward stochastic differential equation with jumps (BSDEJ) and a driver \(f: [0,T] \times \R \times \R^d \times \R \to \R\):
\begin{equation}\label{eq:bsde_Y_Def}
	\left\{
	\begin{array}{ccl}
        dY_t &=&  \; -f(t,Y_t ,\Lambda_t, U_t)dt + \Lambda_t^\top dW_t + \int_E U_t(e) \tilde{N}(dt, de), \\
		Y_T &=&  0. 
	\end{array}  
	\right.
\end{equation}
It is worth noting that, BSDE theory has been developped extensively and enjoys profound applications in many areas, especially in finance (see e.g., 
\cite{oksendal2019applied} for the latest accounts of the theory and its applications). 

\begin{Proposition}\label{prop:preambule} Let \(p>1\) and assume that there exists a solution triplet $\left(Y, \Lambda, U\right) \in \mathbb{S}^{p}_{\F}([0,T],\R) \times L^2_{\F}([0,T], \R^d) \times L^2_{\F}([0, T],\tilde{N})$ to the Riccati BSDEJ
		\begin{equation}\label{eq:bsde_mv}
			\left\{
			\begin{array}{ccl}
				dY_t &=&   (-2r(t) +  \left| \lambda_t + \Sigma \Lambda_t \right|^2 - \frac{1}{2} |\Lambda_t|^2 - \int_E h(U_t(e) \pi(V_t,de))dt  + \Lambda_t^\top dW_t + \int_{E}U_t(e)\tilde{N}(dt,de) . \\
				Y_T &=&  0.
			\end{array}  
			\right.
		\end{equation}
        such that \(\exp(Y):=\Gamma \in \mathbb{S}^{\infty}_{\F}([0,T], \R^*_+)  \). Then, the family of processes $(J^\alpha_t)_{t\in[0,T]}$, $\alpha\in\mathcal{A}$ defined in~\eqref{eq:ansatz} is a submartingale.
        Furthermore, fix $\xi$ $\in$ $\R$, then, a candidate for the optimal feedback control $\c^*(\xi)$  of the inner minimization problem~\eqref{pb:AlterP(c)} is given for every \(t\in [0,T]\) by 
	\begin{equation}\label{eq:optimal_controlDeg_}
		\c^*_t(\xi) = -  \left(\lambda_t + \Sigma \Lambda_t\right) \big(X^{\c^*(\xi)}_t  - \xi e^{-\int_t^T r(s) ds}\big).
	\end{equation}
    Finally, under the $\P$-martingality assumption of Proposition~\ref{prop:Transform}, \(\Gamma\) admits the following representaton 
	\begin{equation}\label{eq:ito_gamma}
		\Gamma_t = \E \Big[ \exp\Big(\int_t^T \big(2r(s) - |\lambda_s + \Sigma \Lambda_s|^2\big) ds\Big) \mid \mathcal F_t\Big], \quad t\leq T,
	\end{equation}
\end{Proposition}
\noindent {\bf Proof:}
In order to show that $(J^\alpha_t)_{t\in[0,T]}$, $\alpha\in\mathcal{A}$ defined in~\eqref{eq:ansatz} is a submartingale, we apply It\^o's product rule of Levy process on $J_t^{\alpha}$ combined with the property of \(Y\) in~\eqref{eq:bsde_Y_Def}, which yields
	\begin{align}
		dJ^\alpha_t =& \Gamma_t \Big[ \alpha^\T_t \alpha_t + 2 \tilde{X}_t\alpha^\T_t(\lambda_t + \Sigma \Lambda_t) + (\tilde{X}^\alpha_t)^2(2r(t) - f(t, \Lambda_t, U_t) + \int_Eh(U_t(e))\pi(V_t,de) + \frac{1}{2} |\Lambda_t|^2)\Big] dt \nonumber\\&+ 2 \Gamma_t \tilde{X}^\alpha_t \alpha^\T_tdB_t + \Gamma_t (\tilde{X}^\alpha_t)^2dW_t  +\Gamma_{t-}(\tilde{X}^\alpha_{t-})^2 \int_E (\exp(U_t(e))-1)\tilde{N}(dt,de).\label{eq:dynamics_J}
	\end{align}
    where $f$  is the driver of the BSDEJ \((Y, \Lambda, U)\). 
In these terms we are bound to choose a function \(f\) for which $(J^\alpha_t)_{t\in[0,T]}$ is a submartingale for all \(\alpha\in \mathcal A\), while \eqref{eq:optimal_controlDeg_} defines the corresponding candidate optimal feedback control.
	\noindent To ease notations, we set $h_t := \lambda_t + \Sigma \Lambda_t $. Choosing \(f\) as in Equation~\eqref{eq:bsde_mv}, for any admissible strategy $\alpha \in \mathcal A$, Equation~\eqref{eq:dynamics_J} and a completion of squares in $\alpha$ yield:
	\begin{align*}
		&\,dJ^\alpha_t \, =  \big| \Tilde{X}_t^{\c} \big|^2 \Gamma_t \Big[ \big(-2r(t)  +\left| \lambda_t + \Sigma \Lambda_t \right|^2 \big)dt +\Lambda_t^\top dW_t \Big] + \Gamma_t \Big(2 \Tilde{X}_t^{\c} \big(r(t) \Tilde{X}_t^{\c} +  \alpha_t^\T \lambda_t\big) + \c_t^\T \c_t  \Big) dt  \\
		&\qquad \qquad\;+ 2\Gamma_t \Tilde{X}_t^{\c} \c_t^\T dB_t + 2 \c^\T_t \left( \Sigma \Lambda_t\right) \Tilde{X}_t^{\c} dt +\Gamma_{t-}(\tilde{X}^\alpha_{t-})^2 \int_E (e^{U_t(e)}-1)\tilde{N}(dt,de)\\
		& \;= \big( \c_t + h_t \Tilde{X}_t^{\c} \big)^\T \Gamma_t \big( \c_t + h_t\Tilde{X}_t^{\c} \big) dt + 2\Gamma_t \Tilde{X}_t^{\c} \c_t^\T dB_t + \Gamma_t \big| \Tilde{X}_t^{\c} \big|^2 \Lambda_t^\T dW_t +\Gamma_{t-}(\tilde{X}^\alpha_{t-})^2 \int_E (e^{U_t(e)}-1)\tilde{N}(dt,de).
	\end{align*}
	As a consequence,  {using $\Gamma_T=1$}, we get 
	\begin{align}
		\big| \Tilde{X}_T^{\c} \big|^2 =& \Gamma_0 \big| \Tilde{X}_0^{\c} \big|^2 + \int_0^T \big( \c_s + h_s \Tilde{X}_s^{\c} \big)^\T  \Gamma_s \big( \c_s + h_s \Tilde{X}_s^{\c} \big) ds + 2\int_0^T \Gamma_s \Tilde{X}_s^{\c} \c_s^\T dB_s +  2 \int_0^T \Gamma_s \big| \Tilde{X}_s^{\c} \big|^2 \Lambda_s^\T dW_s\nonumber\\
        &\qquad\qquad+\int_0^\cdot\int_E \Gamma_{s-}(\tilde{X}^\alpha_{s-})^2
		\bigl(\exp(U_s(e))-1\bigr)\,\tilde{N}(ds,de).\label{eq:ansatz_int}
	\end{align}
    Notice that, by assumption,  $\Tilde{X}^{\alpha}$ has $\p$-$\as$ continuous paths and is bounded ($X^{\c}$ satisfies \eqref{eq:estimateX}), 
    \begin{equation*}
		(\alpha,\Lambda,U)\in
		L^{2,\mathrm{loc}}_{\mathbb{F}}([0,T];\mathbb{R}^d)
		\times
		L^2_{\mathbb{F}}([0,T];\mathbb{R}^d)
		\times
		L^2_{\mathbb{F}}([0,T];\tilde{N}),
	\end{equation*}
	and $\Gamma$ is bounded on $[0,T]$ ($\Gamma \in \mathbb{S}^{\infty}_{\F}([0,T], \R)$). 
    Consequently, the integrand of the stochastic integral with respect to the Brownian motion is locally square-integrable under \(\P\) and the stochastic integrals
	\begin{equation*}
		\int_0^\cdot 2\Gamma_s\tilde{X}^\alpha_s\,\alpha_s^\top dB_s,
		\qquad
		\int_0^\cdot \Gamma_s(\tilde{X}^\alpha_s)^2\,dW_s,
		\qquad
		\int_0^\cdot\int_E \Gamma_{s-}(\tilde{X}^\alpha_{s-})^2
		\bigl(\exp(U_s(e))-1\bigr)\,\tilde{N}(ds,de)
	\end{equation*}
	are well-defined.
    Moreover, these integrals are  $(\mathbb{F},\mathbb{P})$-local martingales. By the definition of $f$ in~\eqref{eq:bsde_mv}, the drift of
	$J^\alpha$ is non-negative, and by a localization argument combined with
	the dominated convergence theorem, using the admissibility of $\alpha$,
	$J^\alpha$ is a submartingale. 
	
	\medskip
	\noindent Let $\{\tau_k\}_{k\geq 1}$ be a common localizing increasing sequence of stopping times  such that $\tau_k \uparrow T$ when $ k \rightarrow \infty$. The local martingales stopped by $\{ \tau_k \}_{k \geq1}$ are true $(\mathbb{F},\mathbb{P})$-martingales. Consequently, Equation~\eqref{eq:ansatz_int} integrated from \(0\) to \(T \wedge \tau_k\) and then taking expectations on both sides, yields
	\begin{equation}
		\E \Big[ \big| \Tilde{X}_{T \wedge \tau_k}^{\c} \big|^2 \Big] \;= \;  \Gamma_0 \big| \Tilde{X}_0^{\c} \big|^2 
		+ \E \Big[ \int_0^{T \wedge \tau_k} \big( \c_s + h_s \Tilde{X}_s^{\c} \big)^\T  \Gamma_s \big( \c_s + h_s \Tilde{X}_s^{\c} \big) ds \Big]. 
	\end{equation}
	Since $\alpha$ is admissible ($\alpha \in \mathcal A$), $X^{\c}$ satisfies \eqref{eq:estimateX}, and therefore $\E\left[\sup_{t\leq T}  |\Tilde X^{\c}_t|^2\right]<\infty$. In particular, for every \(k\), $| \Tilde{X}_{T \wedge \tau_k}^{\c}|^2$ is dominated by a non-negative integrable random variable. Sending $k$ to infinity, an application of the dominated convergence theorem on the left term combined with the monotone convergence theorem on the right term, recall that $\Gamma$ is strictly positive (\(\Gamma \in \mathbb{S}^{\infty}_{\F}([0,T], \R^*_+)\)), yields, as $k \to \infty$,
	\begin{equation}\label{Eq:Square}
		\E \Big[ \big| \Tilde{X}_{T}^{\c} \big|^2 \Big] \; = \; \Gamma_0 \big| \Tilde{X}_0^{\c} \big|^2 
		+ \E \Big[ \int_0^{T} \big( \c_s +  h_s \Tilde{X}_s^{\c} \big)^\T  \Gamma_s  \big( \c_s + h_s \Tilde{X}_s^{\c} \big) ds \Big].
	\end{equation}
	Therefore, since $\Gamma_s$ is positive definite for any $s \leq T$ ($\Gamma_s>0$), the cost functional~\eqref{pb:AlterP(c)} is minimized when \(\c_s =- h_s \Tilde{X}_s^{\c}\) for every \(s\in[0,T]\). Consequently, we obtain that a candidate for the optimal control $\c^*(\xi)$ of the inner minimization problem~\eqref{pb:AlterP(c)} is given by 
\begin{align}
	\c^*_t(\xi) &= -h_t \Tilde{X}_t^{\c^*(\xi)} =-  \left(\lambda_t + \Sigma \Lambda_t\right) \left(X^ {\c^*(\xi)}_t  - \xi e^{-\int_t^T r(s) ds}\right), \quad 0 \leq t \leq T.
\end{align}
and the triplet \((Y,\Lambda, U)\) satisfies the BSDEJ~\eqref{eq:bsde_mv}. This yields~\eqref{eq:optimal_controlDeg_}.
	Furthermore, an application of Proposition~\ref{prop:Transform} ( It\^o formula to $\Gamma = \exp(Y)$) and noting that $Y_T=0$ yields
	\begin{equation}\label{eq:GammaDef2}
		\left\{
		\begin{array}{ccl}
			d\Gamma_t &=&  \;  \Gamma_t \Big[ \big(-2r(t)  +\left| \lambda_t + \Sigma \Lambda_t \right|^2 \big)dt +\Lambda_t^\top dW_t + \int_{E}\big(e^{U_t(e)}-1\big)\widetilde{N}(dt,de)\Big], \\
			\Gamma_T &=&  1. 
		\end{array}  
		\right.
	\end{equation}
    together with the Laplace transform representation~\eqref{eq:ito_gamma} owing to the $\P$-martingality assumption of the stochastic exponential \(M\) defined for every $0 \leq t \leq T$ by :
	\begin{align}\label{eq:localMart_M2} 
		M_t &:= \; \mathcal E\Big( \int_0^t \Lambda_s^\top dW_s + \int_0^t\int_{E}\big(e^{U_s(e)}-1\big)\tilde{N}(ds,de) \Big) = \;  \Gamma_t \exp\Big(\int_0^t\big( 2r(s) -\left| \lambda_s + \Sigma \Lambda_s \right|^2\big) ds\Big).
	\end{align}
	This complete the proof. \hfill $\Box$

\medskip
\noindent {\bf Remark:} The inner optimization problem~\eqref{pb:AlterP(c)} (and thus the mean-variance problem~\eqref{outer_inner_optimization_pb}) then boils down to proving the existence and uniqueness of solutions to Equation~\eqref{eq:GammaDef2}, or equivalently Equation~\eqref{eq:bsde_mv}. This problem corresponds to a specific setting of the theory developed in Section~\ref{sect:Main_solBDSEJ}, which consequently yields explicit expressions for the vector \(\Lambda\) and the scalar \(U\).

\subsection{Optimal strategy in the general multivariate correlation structure: A verification argument}\label{Verif_result}

\noindent In order to avoid restrictions on the correlation structure arising from the martingale distortion approach used earlier in Section~\ref{subsect:Intuition_Brownian}, which may be specific to the one-dimensional case $(d = 1)$ (see e.g.,~\cite{HanWong2020a}) and to certain degenerate
multivariate settings such as $(\Sigma = \diag(\rho,\ldots,\rho))$, we instead employ a verification argument in the spirit of~\cite[Theorem 3.1]{Gnabeyeu2026b,dro2026optimal} to solve the Markowitz optimization problem with the martingale optimality principle. 

\medskip
\noindent
Let $\Lambda$ and $U$ be defined as 
\begin{equation}\label{eq:LambdaU}
	\Lambda_t :=  \sqrt{\diag(V_t)}\,\varsigma^\top(t)\, \sigma^v \,\psi(T-t) 
    \quad \text{and}\quad
    U_t(e) :=\psi^\top(T-t) \eta(e), 
    \quad e\in E, \quad 0 \leq t \leq T.
\end{equation}
We will work under the following assumption,
\begin{assumption}\label{assm:gen}
	Assume that there exists a solution
	$\psi \in C([0,T],\mathbb{R}^d)$ to the above-mentioned inhomogeneous Riccati--Volterra equation satisfying the below appropriate
	boundedness condition i.e. such that 
	\begin{equation}\label{eq:condtheta}
		\max_{1 \leq i \leq d} \Big[ \sup_{t \in [0,T]} \left( \theta_i^2 + (\sigma^v_i)^2 \varsigma^i(t)^2 \psi^i(T-t)^2 + \int_E|U_t(e)|\,\nu_i(de)\right)  \Big]\leq \frac{a}{a(p)},
	\end{equation}
	holds for a sufficient large $p > 1$, where
     the constant $a(p)$ is given by 
	{  \begin{equation}a(p)=\max \Big[p \left(2 + |\Sigma| \right),   {2 (8p^2 {- 2p}) \left( 1  + |{\Sigma}|^2  \right)}, { p \left( 1  + |{\Sigma}|^2  \right)}, \,{ 4 p } \Big]. \label{eq:constap}
	\end{equation}}
	and the constant $a>0$ is such that $\E\left[\exp\big(a\int_0^T \sum_{i=1}^d V^i_s ds\big)\right] < \infty$.
\end{assumption}
	
	\noindent {\bf Remark on Assumption~\ref{assm:gen}:} 
    \noindent 1. Since $\psi$ is continuous on the compact interval $[0,T]$, it is bounded. Hence, for \(U\) in~\eqref{eq:LambdaU}
we have $|U_t(e)|\leq \|\psi\|_{\infty,T}\|\eta(e)\|.$
Consequently, for every $i=0,\ldots,d$ as $\nu_i(E)<\infty$ by the finiteness of $\nu_i$ and $\int_E\|\eta(e)\|^2\,\nu_i(de)<\infty$ owing to assumption~\ref{assump:kernelJumpVolterra}~$(ii)$, then $\sup_{t\in[0,T]}\int_E|U_t(e)|^2\,\nu_i(de)<\infty$
and by the Cauchy--Schwarz inequality $\beta_i:=\sup_{t\in[0,T]}\int_E|U_t(e)|\,\nu_i(de)<\infty.$
    That is the constants \(\beta_i\), $i = 1, \ldots, d$ are well defined and finite.
    
  \noindent  2.  
Note that if Assumption~\ref{assm:gen} hold, then 
\begin{equation}\label{eq:assumption_novikov}
		\E \Big[ \exp\Big( a(p)\int_0^T \big(  |\lambda_s|^2 + \left|\Lambda_s\right|^2 + V_s^\top \int_E|U_s(e)|\,\nu(de)\big)ds \Big) \Big] \;< \;  \infty,,
	\end{equation}
	holds for some $p > 1$ and a constant $a(p)$ given by~\eqref{eq:constap}.
	\noindent In fact, under Assumption~\ref{assm:gen}, we will have 
	\begin{align}\label{assum_for_some_p}
		&a(p)\left( |\lambda_s|^2 + \left|\Lambda_s\right|^2 + V_s^\top \int_E|U_s(e)|\,\nu(de) \right) \; \nonumber \\&\hspace{3cm}= \; a(p) \sum_{i=1}^d  \left( \theta_i^2 + (\sigma^v_i)^2 \varsigma^i(s)^2\psi^i(T-s)^2 + \int_E|U_s(e)|\,\nu_i(de) \right)\,V_s^i \; \leq \;  a \sum_{i=1}^d V_s^i,
	\end{align} 
	which implies that $\E \left[ \exp\Big(  a(p) \int_0^T \Big( |\lambda_s|^2 + \left|\Lambda_s\right|^2 + V_s^\top \int_E|U_s(e)|\,\nu(de) \Big)ds\Big) \right]< \infty$, thanks to the exponential-affine representation in~\cite[Theorem A.1.]{dro2026optimal} ( with \(\mathcal{M} \ni m(\dd s) := a{\bold{1}_d}\,\delta_0(\dd s) \)).

	\smallskip
	\noindent 3. Note that condition~\eqref{eq:condtheta}--\eqref{assum_for_some_p} on $\Lambda$ and $U$ is particularly useful for establishing the martingale property of~\eqref{eq:localMart_M} in order to have Laplace functional for the affine Volterra processes~\eqref{eq:volterra2}--~\eqref{VolSqrt_} (see also proof of Lemma~5.1 in \cite{dro2026optimal}).

    \smallskip
\noindent 4. We will show that the solution $\psi$ to the Riccati--Volterra equation associated with the solution of the BSDEJ~\eqref{eq:bsde_mv} satisfies $\psi\in\mathbb{R}^d_-$, so that $U$ is non-positive on
$[0,T]\times E$ (see~\eqref{eq:LambdaU}). Consequently, in contrast to the exponential integrability condition imposed on $U$ in~\cite[Assumption 3.1]{dro2026optimal}, no additional integrability condition on  $U$ is required in the present setting to ensure that the function $h$ in~\eqref{eq:bsde_mv} is $\pi(x,de)$-integrable for every $x\in\mathbb{R}^d_+$.
Indeed, this follows from the bound $0\leq e^u-u-1\leq\frac{u^2}{2},\; u\leq0,$ together with the fact that $U_t\in L^2_{\nu_k}(E,\mathbb{R})$ for every $t \in [0,T]$ and every $k=0,\ldots,d$ owing to condition~\eqref{eq:IntJumpSize}.

 
	\medskip
	\noindent{\bf Remark:} 
	Condition \eqref{eq:condtheta} concerns the risk premium constants  $(\theta_1,\ldots, \theta_d)$. For a large enough constant $a>0$, from ~\cite[Theorem A.1.]{dro2026optimal} ( with \(\mathcal{M} \ni m(\dd s) := a{\bold{1}_d}\,\delta_0(\dd s) \)),   a sufficient condition ensuring that
	$\E\big[\exp\big(a\int_0^T \sum_{i=1}^d V^i_s ds\big)\big]<\infty$ is the existence of a continuous solution $\tilde{\psi}$  on $[0,T]$ to the inhomogeneous Riccati--Volterra equation \(\forall t \in [0,T]\) and $i=1,\ldots,d$,
	\begin{align} \label{eq:Riccati}
		\tilde{\psi}^i(t) &= \; \int_0^t K_i(t-s) \Big(a+ \big(D^\T\tilde{\psi}(s)\big)_i + \frac{(\sigma^v_i)^2}{2}(\varsigma^i(T-s)\tilde{\psi}^i(s)) ^2\Big) ds  + \int_{E}h\big(\psi(T-s)^\top\eta(e)\big)\nu_i(de). 
	\end{align} 
	
	\medskip
	\noindent We start by the below proposition establishing that, for some \(p>1\), the triplet $\left(Y,\Lambda, U\right)$  is a $\mathbb{S}^{p}_{\F}([0,T], \R) \times L^2_{\F}([0,T], \R^d)\times L^2_{\F}([0,T], \tilde{N})$-valued solution to a Riccati BSDE with jumps provided solvability in \(C([0,T],(\R^d)^*)\) of an inhomogeneous Ricatti-Volterra equation with L\'evy exponential jump compensator.
	\begin{Proposition} 
		\label{prop:MV_riccati} 
		Assume that there exists  a solution $\psi \in C([0,T],\R^d)$ to the inhomogeneous Riccati-Volterra equation with L\'evy jumps (exponential jump compensator) for $i = 1,\ldots,d,$:
		\begin{align}
			\psi^i(t) &= \int_0^t K_i(t-s)
			\Bigl(-\theta_i^2 + F_i(T-s,\psi(s))\Bigr)\,ds,
			\quad 
			\label{eq:Riccatipsi1}\\[6pt]
			F_i(s,\psi) &= -2\theta_i\rho_i\sigma^v_i\,\varsigma^i(s)\psi^i
			+ (D^\top\psi)_i
			+ \frac{(\sigma^v_i)^2}{2}(1-2\rho_i^2)\,(\varsigma^i(s)\psi^i)^2
			+ \int_{E} h\!\left(\psi^\top\eta(e)\right)\nu_i(de).
			\label{eq:Riccatipsi2}
		\end{align}
		Let $\left(Y, \Lambda, U\right)$ be defined as 
		\begin{equation}\label{eq:Solbsde_mv}
			\left\{
			\begin{array}{ccl}
				Y_t &=&  \ \int_t^T \Big( 2r(s) + \int_{E}h\big(\psi(T-s)^\top\eta(e))\nu_0(de)+\sum_{i=1}^d \big( -\theta_i^2  +  F_i(s,\psi(T-s))\big)  g^i_t(s) \Big)ds , \\
				\Lambda_t^i &=&  \sigma^v_i\varsigma^i(t) \psi^i(T-t) \sqrt{V^i_t}, \quad U_t(e) = \psi^\top(T-t) \eta(e),\quad  i=1,\ldots,d, \quad 0 \leq t \leq T, 
			\end{array}  
			\right.
		\end{equation}
		where $g$ $=$ $(g^1,\ldots,g^d)^\T$ is given by \eqref{eq:processg} i.e. the $\R^d$-valued process \((g_t(s))_{t\leq s}\)  is defined in~\eqref{eq:processg}.  Then for some $p >1 $,  $\left(Y, \Lambda, U \right)$  is a  $\mathbb{S}^{p}_{\F}([0,T], \R) \times L^2_{\F}([0,T], \R^d) \times L^2_{\F}([0, T],\tilde{N})$-valued solution to the Riccati BSDEJ~\eqref{eq:bsde_mv}.
		Furthermore the process $\Gamma = \exp(Y)$ satisty, $0<\Gamma_t\leq e^{2 \int_t^T r(s) ds}$ $\P-a.s.$ for all $t\leq T$ so that  $\Gamma \in \mathbb{S}^{\infty}_{\F}([0,T], \R) $
		and if moreover $g^i_0(0)>0$ for some $i\leq d$, then  {$0<\Gamma_0 < e^{2 \int_0^T r(s) ds}$} with 
		\begin{equation}\label{eq:Gamma0}
        \begin{aligned}
			\Gamma_0 &=\exp\Big( 2 \int_0^T r(s) ds + \int_0^T \big( \int_{E}h\big(\psi(T-s)^\top\eta(e)\big)\nu_0(de)\\
            &\qquad+ \sum_{i=1}^d \big[V_0^i  \big(-\theta_i^2  + F_i(s,\psi(T-s))\big) \varphi^i(s) ds +  \psi^i(T-s) \mu^i(s)\big] \big) ds \Big).
        \end{aligned}
		\end{equation}
	\end{Proposition}
	\noindent The following remark makes precise the existence of a continuous solution to the Riccati-Volterra equation with L\'evy jumps \eqref{eq:Riccatipsi1}-\eqref{eq:Riccatipsi2}.
    
	\medskip
	\noindent{\bf Remarks:} 1. In a one-dimensional example $d=1$ and $\nu_0\equiv0$ together with $\varsigma\equiv I$,~\cite[section~$5$]{BondiLivieriPulido2024} establishes that, there exists a continuous global solution $\psi \in C(\R_+,\R_-)$ to  the Ricatti--Volterra Equation ~\eqref{eq:Riccatipsi1}--~\eqref{eq:Riccatipsi2}.
	
	\smallskip
	\noindent 
	2. In the case, where we take $\nu_i\equiv0$ for every $i=1,\cdots,d$ and $K$ satisfies the Assumption~\ref{assump:kernelJumpVolterra}~$(i)$, Theorem~\ref{Thm:Riccatilocalexistence} provides the existence of a unique global continuous solution $\psi$ to the inhomogeneous Ricatti--Volterra Equation ~\eqref{eq:Riccatipsi1}--~\eqref{eq:Riccatipsi2}. More specifically,
	\begin{itemize}
		\item  If {\textit{\(1-2\rho_i^2 < 0\)}}, then 
		Theorem~\ref{Thm:Riccatilocalexistence}~$(b)$ guarantees that there exists a unique non-continuable continuous solution \((\psi, T_{\max})\) to Equation~\eqref{eq:Riccatipsi1}--~\eqref{eq:Riccatipsi2} with  \(\psi_i \le 0\) for \(i=1,\ldots,d\), in the sense that $\psi$ satisfies~\eqref{eq:Riccatipsi1}--~\eqref{eq:Riccatipsi2} on $[0,T_{max})$ with $T_{max} \in (0,T]$ and $\sup_{t<T_{max}}|\psi^i( t)| = +\infty$, if $T_{max}<T$. 
		
		\item If {\textit{\(1-2 \rho_i^2 \geq 0\)}}, then Theorem~\ref{Thm:Riccatilocalexistence} establishes the existence of a unique global continuous solution $\psi \in C([0,T],\R^d)$ to ~\eqref{eq:Riccatipsi1}--~\eqref{eq:Riccatipsi2} such that \(\psi^i(t)< 0\) for every \(t>0\) and \(i=1,\cdots,d.\)
	\end{itemize}

	\smallskip
	\noindent 3. Furthermore, for our numerical case studies, in Section~\ref{Sec:Num}, note as in~\cite{Gnabeyeu2026b} that, if the matrix $D$ in the drift of the volatility process is a diagonal matrix, i.e. $D=-\diag{(\lambda_1,\dots, \lambda_d)}$, for \(i=1,\ldots,d,\) by Theorem~\ref{Thm:Riccatilocalexistence}~$(c)$, since $-\theta_i^2<0$, $\psi^i \in C([0,T],\R_-)$ is unique global solution to the following Volterra equation  (here, for short $\bar{D}_i(s):=\lambda_i+2\theta_i \rho_i \sigma^v_i \varsigma^i(T-s)$). 
	\begin{equation}\label{eq:comp}
		\chi(t) = \int_0^t K_i(t-s)  \Big( -\theta_i^2 - \big(\lambda_i + 2 \theta_i \rho_i \sigma^v_i \varsigma^i(T-s)\big) \chi(s)  + \frac {(\sigma^v_i)^2} 2  ( 1- 2\rho_i^2)\varsigma^i(T-s)^2\chi(s) ^2\Big)ds, \; t\leq T.
	\end{equation}
	Combining the component-wise solutions, we finally obtain the unique global solution $\psi$ of the inhomogeneous Ricatti--Volterra Equation~\eqref{eq:Riccatipsi1}--~\eqref{eq:Riccatipsi2}. 
	
	\smallskip
	\noindent Moreover, it follows in this case that the condition \eqref{eq:condtheta} can be made more explicit by bounding $\psi$ with respect to the vector $\theta$.
	Indeed setting for \(i=1,\ldots,d,\)  \(\bar{\lambda}_i:=\inf_{t\in [0,T]} \big(\lambda_i+2\sigma^v_i\rho_i\theta_i\varsigma^i(t)\big)= \lambda_i+2\sigma^v_i\rho_i\theta_i \|\varsigma^i\|_{\infty} {\bold{1}_{\rho_i \leq 0}}, \) (assuming boundedness for the function \(\varsigma\)) and assuming that \(\bar{\lambda}_i \neq 0\), by Corollary~\ref{Corol:Riccatilocalexistence} we have:
	\begin{equation}
		\sup_{t \in [0,T]} |\psi^i(t)| \leq \frac{|\theta_i |^2}{\bar{\lambda}_i}\int_0^T f_{\bar{\lambda}_i}(s)ds = \frac{|\theta_i |^2}{\bar{\lambda}_i}(1- R_{\bar{\lambda}_i}(T)), \;i=1,\ldots,d.
	\end{equation}
	where $R_{\bar{\lambda}_i}$ is the \textit{ $\bar{\lambda}_i$-resolvent} associated to the real-valued kernel $K_i$ and $f_{\bar{\lambda}_i}$ its antiderivative (see e.g.,~\cite[Section 2.2]{EGnabeyeu2025}).
	Consequently, combining those component-wise estimates, we finally obtain that, a sufficient condition on $\theta$ to ensure \eqref{eq:condtheta} would be for all $i=1,\ldots,d$ 
    \begin{equation}
		\theta^2_i\left( 1+(\sigma^v_i\|\varsigma^i\|_\infty\frac{\theta_i}{\bar{\lambda}_i})^2\big(1- R_{\bar{\lambda}_i}(T)\big)^2 + \frac{1}{\bar{\lambda}_i}(1- R_{\bar{\lambda}_i}(T)) \sqrt{\nu_i(E)}
\Big(\int_E \|\eta(e)\|^2\,\nu_i(de)\Big)^{\frac12}\right) \leq \frac{a}{a(p)}.
	\end{equation}

\medskip
\noindent First, we provide a verification result for the inner optimization problem \eqref{pb:P(c)} via the standard completion of squares technique, see  for instance \citet[Proposition 3.1]{LimZhou2002}, \citet[Proposition 3.3]{Lim2004} and \citet[Theorem 3.1]{ChiuWong2014}. 

\smallskip
\noindent
 In the following theorem, we propose a candidate optimal control $\c^*(\xi)$, we prove its optimality and establish its admissibility and the integrability of the corresponding state process $X^{\c^*(\xi)}$ for any $\xi$ $\in$ $\R$.

\begin{Theorem}
	\label{thm:inner}
    Assume there exists a solution triplet $(Y, \Lambda, U) \in \mathbb{S}^{p}_{\F}([0,T], \R_) \times L^2_{\F}([0,T], \R^d) \times L^2_{\F}([0, T],\tilde{N})$ to the equation \eqref{eq:bsde_mv} such that $\exp(Y_t)>0$, for all $t\leq T$ and \eqref{assum_for_some_p} holds for some $p \geq 1$ and a constant $a(p)$ given by \eqref{eq:constap}.
	Fix $\xi$ $\in$ $\R$, 	Then, the  inner minimization problem \eqref{pb:P(c)} admits an optimal feedback control  $\c^*(\xi)$ satisfying
	\begin{align}
		\label{eq:optimal_control}
		\c^*_t(\xi) &= -  \left(\lambda_t + \Sigma \Lambda_t\right) \left(X^{\c^*(\xi)}_t  - \xi e^{-\int_t^T r(s) ds}\right), \quad 0 \leq t \leq T \\
		&= \;  \Big(- \big(\theta_i + \rho_i\sigma^v_i\varsigma^i(t) \psi^i(T-t) \big) \sqrt{V_t^i} \big(X^{\c^*(\xi)}_t - \xi e^{-\int_t^T r(s) ds}\big)   \Big)_{1 \leq i \leq d}, \quad 0 \leq t \leq T.
	\end{align}
	Moreover, the control $\c^*(\xi)$ is  unique under a given solution \((S, V, B, B^\top)\) to ~\eqref{VolSqrt_}-~\eqref{eq:hestonS} and is admissible with
	\begin{equation}\label{eq:boundX}
		\E\Big[ \sup_{ t \in [0, T]} |X^{\c^*(\xi)}_t|^p \Big] < \infty,\quad \text{for some sufficiently large}\quad p\geq1.
	\end{equation}
	 The optimal value for the minimization problem~\eqref{pb:P(c)} or the associated optimal cost is 
	\begin{equation}
		\label{eq:optimal_vamue_inner}
		\tilde V(\xi) \; = \;    \Gamma_0 \left|x_0 - \xi e^{-\int_0^T r(s) ds} \right|^2.
	\end{equation}
\end{Theorem}
\noindent For the sake of brevity, yet without compromising self-containment, we present a sketch of the proof, deferring the detailed proof to Section~\ref{sect:proofMresult}.

\smallskip
\noindent \emph{Sketch of Proof:} In contrast to~\cite{AbiJaberMillerPham2021}, we use a verification method of the martingale optimality principle
	(see for e.g.\ \cite{dro2026optimal}, \cite{kharroubi2013mean}, or its first
	use in \cite{HuImkellerMueller2005}). 
In order to prove that $\alpha^*$ in~\eqref{eq:optimal_control} is indeed the optimal portfolio strategy for the inner problem, we show that the
	family of processes $(J^\alpha_t)_{t\in[0,T]}$, $\alpha\in\mathcal{A}$, defined in~\eqref{eq:ansatz}
	 where the triplet $(Y,\Lambda,U)$ is solution to the BSDEJ~\eqref{eq:bsde_mv}, satisfies the following conditions:
	\begin{enumerate}
		\item $J_T^\alpha = |X_T^\alpha - \xi|^2$
		for all $\alpha \in \mathcal{A}$;
		
		\item $J_0^\alpha$ is constant and independent of $\alpha \in \mathcal{A}$;
		
		\item $J^\alpha$ is a submartingale for all $\alpha \in \mathcal{A}$, 
		and for $\alpha^* \in \mathcal{A}$ in~\eqref{eq:optimal_control},
		$J^{\alpha^*}$ is a martingale.
	\end{enumerate}
    
\medskip
\noindent Finally, using Theorem~\ref{thm:inner}, we can now provide the explicit solution
for the optimal investment strategy 
of the Markowitz problem~\eqref{optimization_problem} in the multivariate fake stationary Volterra Heston model \eqref{VolSqrt_}-\eqref{eq:hestonS}. More specifically, combining the above Theorem~\ref{thm:inner}, we deduce the solution for the outer optimization problem \eqref{optimization_problem} under a non-degeneracy condition on the solution $\Gamma$ to the exponential change of variable of the quadratic BSDEJ~\eqref{eq:bsde_mv},  yielding Theorem~\ref{Thm:OuterMarkowitz} below.

\smallskip
\noindent
In the following theorem, we show that $\c^*$ in~\eqref{eq:optimal_control} is optimal for the outer optimization problem~\eqref{optimization_problem} for some optimal $\xi^* \in \mathbb{R}$, and we derive the corresponding efficient frontier.

\begin{Theorem}\label{Thm:OuterMarkowitz}
	Assume that there exists  a solution $\psi \in C([0,T],\R^d)$ to the Riccati-Volterra equation with L\'evy jumps \eqref{eq:Riccatipsi1}-\eqref{eq:Riccatipsi2} such that Assumption~\ref{assm:gen} is in force. 
	Assume that $g^i_0(0)>0$ for some $i\leq d$. Then, the optimal investment 
	strategy for the Markowitz problem or  maximization problem \eqref{optimization_problem} in the   multivariate fake stationary Volterra Heston model \eqref{VolSqrt_}-\eqref{eq:hestonS} is given by the admissible control
	\begin{align}
		\label{eq:optimal_control_Heston}
		\c^*_t &= (\c^{*i}_t)_{1\leq i\leq d} = -  \left(\lambda_t + \Sigma \Lambda_t\right) \left(X^{\c^*}_t - \xi^* e^{-\int_t^T r(s) ds}\right), \quad 0 \leq t \leq T \\
		&= \;  \Big(- \big(\theta_i + \rho_i\sigma^v_i\varsigma^i(t) \psi^i(T-t) \big) \sqrt{V_t^i} \big(X^{\c^*}_t - \xi^* e^{-\int_t^T r(s) ds}\big)   \Big)_{1 \leq i \leq d}, \quad 0 \leq t \leq T.\label{eq:optimal_control_Heston2}
	\end{align}
	where $\xi^*:=m-\zeta^*$ with \(m\) the expected terminal wealth is defined as:
	\begin{equation}
		\label{eq:eta_star}
		\xi^* \; = \;  \frac{ m - \Gamma_0 e^{-\int_0^T r(s) ds} x_0}{1- \Gamma_0 e^{-2 \int_0^T r(s) ds}}\quad \text{with}\quad 
		\zeta^* \; = \;  \frac{\Gamma_0 e^{-\int_0^T r_s ds} \Big( x_0 - m e^{-\int_0^T r(s) ds } \Big)}{1-\Gamma_0 e^{-2\int_0^T r(s) ds} }.
	\end{equation}
	and the wealth process $X^*$ $=$ $X^{\alpha^*}$ by \eqref{eq:wealth}--\eqref{eq:wealthProcess} with $\lambda=\big(\theta_1 \sqrt{V^1},$ $\ldots, \theta_d \sqrt{V^d} \big)^\top$ and satisfies
	\begin{equation}\label{eq:boundX^*}
		\E\Big[ \sup_{ t \in [0, T]} |X^{\c^*}_t|^p \Big] < \infty,\quad \text{for some sufficiently large}\quad p\geq1.
	\end{equation}
	Moreover, ~\eqref{eq:optimal_control_Heston} is unique under a given solution \((S, V, B, B^\top)\) to ~\eqref{VolSqrt_}-~\eqref{eq:hestonS}.
	The optimal value of \eqref{optimization_problem} for the optimal wealth process $X^*=X^{\alpha^*}$ or the variance of \(X^{\alpha^*}_T\) is given below with $\Gamma_0$ as in equations~\eqref{eq:Gamma0}.  
	\begin{equation}
		\label{eq:value_final}
		V(m) \; = \;    \V (X_T^{\alpha^*}) \; = \;  \Gamma_0 \frac{\big(x_0 - m e^{-\int_0^T r(s) ds} \big)^2}{1 - \Gamma_0 e^{-2\int_0^T r(s) ds}}.
	\end{equation}   
\end{Theorem}
\noindent Note that Theorem~\ref{Thm:OuterMarkowitz} above extends \cite[Theorem 4.2]{HanWong2020a}, \cite[Theorem 4.4]{AbiJaberMillerPham2021} to the multivariate, time-dependent diffusion coefficient case. 

\medskip
\noindent {\bf Proof of Theorem~\ref{Thm:OuterMarkowitz}:}
By assumption and owing to Proposition~\ref{prop:MV_riccati}, there exists a solution triplet $ \big(\Gamma:=\exp(Y), \Lambda, U\big)\in { \mathbb{S}^{\infty}_{\F}([0,T], \R) \times L^2_{\F}([0,T], \R^d) \times L^2_{\F}([0, T],\tilde{N})}$ solution to the  Equation~\eqref{eq:GammaDef2}  under the specification \eqref{eq:Solbsde_mv}. 
Consequently, under Assumption~\ref{assm:gen}, Theorem~\ref{thm:inner} gives that the candidate for the optimal feedback control for the inner problem~\eqref{pb:P(c)} is defined in  \eqref{eq:optimal_control} so that the max-min problem~\eqref{outer_inner_optimization_pb} (which is equivalent to the Markowitz problem \eqref{optimization_problem}) is equivalent to (thanks to~\eqref{Eq:Square})
\begin{equation}
	\label{outer_inner_optimization_pb_}
	\max_{\zeta \in \R}  J(\zeta), \quad \mbox{ with } \; J(\zeta) \; = \;  \Gamma_0 \big|x_0 - (m-\zeta) e^{-\int_0^T r(s) ds} \big|^2  - \zeta^2. \\
\end{equation}
The first and second order derivatives are
\begin{equation*}
	\frac{\partial J}{\partial \zeta} =  2\Gamma_0 \big[x_0 - (m - \zeta) e^{- \int^T_0 r(s) ds} \big] e^{- \int^T_0 r(s) ds} - 2 \zeta \quad\text{and}\quad
	\frac{\partial^2 J}{\partial \zeta^2} =  \Gamma_0 e^{- 2 \int^T_0 r(s) ds} - 2 < 0,
\end{equation*}
where we have used the strict inequality $\Gamma_0 <  e^{2 \int_0^T r(s) ds}$, by the last claim of Proposition~\ref{prop:MV_riccati}. This ensures that the quadratic function $J$ is
strictly concave. This yields that  the maximum is achieved 
from the first-order condition $\frac{\partial J}{\partial \zeta} (\zeta^*)$ $=$ $0$, which gives 
\begin{align*}
	\zeta^* &= \;  \frac{\Gamma_0 e^{-\int_0^T r(s) ds} \Big( x_0 - m e^{-\int_0^T r(s) ds } \Big)}{1 - \Gamma_0 e^{-2\int_0^T r(s) ds}},
\end{align*}
and thus $\xi^*$ $=$ $m-\zeta^*$ is given by  \eqref{eq:eta_star}.  We conclude that the optimal control is equal to  $\alpha^*=\alpha^*(\xi^*)$ as in \eqref{eq:optimal_control_Heston}, and by 
\eqref{outer_inner_optimization_pb}, the optimal value or variance of terminal wealth \(\V (X_T^*)\) of \eqref{optimization_problem} is obtained by direct simplification with $\zeta^*$ thanks to~\eqref{Eq:Square}--~\eqref{eq:optimal_vamue_inner} and is equal to:
\begin{align*}
	V(m)=\tilde V(\xi^*) - (\zeta^*)^2 =   \Gamma_0 \left|x_0 - \xi^* e^{-\int_0^T r(s) ds} \right|^2-(\zeta^*)^2 = \Gamma_0 \frac{\big|x_0 - m e^{-\int_0^T r(s) ds} \big|^2}{1 - \Gamma_0 e^{-2\int_0^T r(s) ds}}
\end{align*}
where we used~\eqref{eq:eta_star} and simple calculation  to obtain the last equality. This give \eqref{eq:value_final} and complete the proof. \hfill $\Box$

\medskip\noindent {\bf Remark:}
\smallskip
\noindent { 1.} \(\Gamma_0\) is the initial value of the process \(t \mapsto \Gamma_t\) solution to a BSDEJ~\eqref{eq:GammaDef2} of Riccati type which appears to have an explicit closed form formula $\Gamma:=\exp(Y)$ with $Y$ in~\eqref{eq:Solbsde_mv} or equivalently~\eqref{eq:Gamma0}.

 \smallskip
 \noindent { 2.}  The efficient frontier of the mean-variance portfolio selection problem~\eqref{optimization_problem} is given by Equation~\eqref{eq:value_final}.
\noindent In fact, Equation \eqref{eq:value_final} gives the relationship between the expected terminal wealth \(m\) and variance of the terminal wealth \(\sigma^2:=V(m)\) of efficient portfolios. The efficient frontier is thus the curve of \eqref{eq:value_final} depicted on the plane \((\sigma,m)\) of mean and standard deviation. Taking the square root to both sides of \eqref{eq:value_final}, and accounting for the feasibility of the MV problem, the efficient frontier reads.
\begin{equation}\label{eq:effFront}
	m=x_0e^{\int_0^T r(s) ds} + a\, \sigma, \quad \text{with}\quad a:=\sqrt{\frac{e^{2\int_0^T r(s) ds}}{\Gamma_0}-1}
\end{equation}
It is still a straight line (also termed the \textit{capital market line}); see also \citet{ZhouLi2000,LimZhou2002,ChiuWong2014}. Its slope, called the \textit{price of risk} writes owing to equations~\eqref{eq:Gamma0},
\begin{align*} 
a^2 &\,= \exp\Big(  -\int_0^T \int_{E}h\big(\psi(T-s)^\top\eta(e)\big)\nu_0(de)\, ds + \sum_{i=1}^d\int_0^T  \big(\theta_i^2 - F_i(s,\psi(T-s))\big) g^i_0(s) ds \Big)-1\\
	&= \exp\Big( \int_0^T \big( -\int_{E}h\big(\psi(T-s)^\top\eta(e)\big)\nu_0(de)+ \sum_{i=1}^d \big[V_0^i  \big(\theta_i^2  - F_i(s,\psi(T-s))\big) \varphi^i(s) ds\\
    &\hspace{2cm}- \psi^i(T-s) \mu^i(s)\big] \big) ds \Big) - 1
\end{align*}

\section{Numerical experiments: A $2D$ fake stationary rough Heston model}\label{Sec:Num} 
	\noindent We consider a financial market consisting of one risk-free asset and \(d = 2\) risky assets, with an investment horizon of \(T = 1\) year. In this case, there is no jumps part, $\eta\equiv0$ and we set $\nu_i\equiv0$ for every $i\in\{0,1,2\}$.  
To model the roughness of the asset price dynamics, we employ an appropriate integration kernel. We choose a fractional kernel of Remark~\ref{rm:Kernels} of the form:
\[
K(t) = \begin{pmatrix}
	\frac{t^{\alpha_1-1}}{\Gamma(\alpha_1)} & 0 \\
	0 & \frac{t^{\alpha_2-1}}{\Gamma(\alpha_2)}
\end{pmatrix}, \quad 0.1 +\frac12=\alpha_1,\, 0.4 +\frac12=\alpha_2 \in \big( \frac{1}{2}, 1 \big).
\]
Here, \( \Gamma(\alpha) \) is the Gamma function, and the parameter \( \alpha \) controls the degree of roughness in the model. 
Note that, the model is sufficiently rich to capture several well-known stylized facts of financial markets:

\begin{itemize}[label=\(\bullet\)]
	\item Each asset \(S^i\), \(i=1,2\) exhibits stochastic rough volatility driven by the process \(V^i\), with different Hurst indices \(\alpha_i\).
	We can even assume a corellation between \(B^1\) and \(B^2\) through \(\rho \in (-1,1)\).
	\item Each stock \(S^i\) is correlated with its own volatility process through the parameter \(\rho_i\) to take into
	account the leverage effect.
\end{itemize}
For the two-dimensional fake stationary rough Heston volatility model described in~\cite[Section 2.1]{Gnabeyeu2026b}, we consider the setting in which the simplified specification \(\varphi(t) = I_{2\times2}, \; t \ge 0, \)
is adopted in Equation~\eqref{VolSqrt_}.  Consequently, the \(\R^2_-\)-valued mean-reverting function \(\mu\) is constant in time, that is, \(\mu(t) = \mu_0 \in \R^2_+, \; \forall\,t \ge 0,\) (see, e.g.,~\cite{EGnabeyeu2025}).
Let \(D:=\diag(\lambda_1,\lambda_2)\) together with
    $\varsigma^i (t) :=\varsigma^i_{\alpha_i,\lambda_i,c_i}(t)
		\ge 0$ given as the unique solution to the functional convolution equation (see~\cite{EGnabeyeu2025}) for every $i\in\{1,2\}$:
		\begin{equation}\label{eq:VolterraStabilizer2}
			\textit{($E_{\alpha_i, \lambda_i, c_i}$)}: \;\forall\, t\ge 0, \; c_i \lambda^2_i\Big(1-\big(1 -\int_0^t f_{\alpha_i,\lambda_i}(s)\,ds\big)^2 \Big) =  \int_0^t f_{\alpha_i,\lambda_i}^2(t-s) \varsigma_{\alpha_i,\lambda_i,c_i}^2(s)\,ds, \;  \textit{with} \; c_i >0.
		\end{equation}
        Here, \(f_{\alpha_i,\lambda_i}\) called the \textit{Mittag--Leffler  probability density} is the resolvent of the second kind of \(\lambda_i K_i\), defined on $(0,+\infty)$ by $f_{\alpha_i, \lambda_i}(t) = \alpha_i\lambda_i t^{\alpha_i-1} E'_{\alpha_i}(-\lambda_i t^{\alpha_i})$ where 
        $E_{\alpha}$ denotes the \textit{ standard   Mittag-Leffler function}.
	The function \(\varsigma:=\varsigma_{\alpha,D,c}\) is a stabilysing function design to ensure constant mean $m_0:=D^{-1}\mu_0$ and covariance $\Omega:=\diag((\sigma^v_1)^2c_1\frac{\mu_0^1}{\lambda_1},(\sigma^v_2)^2c_2\frac{\mu_0^2}{\lambda_2})$ for the Volatility process \(V\) in the spirit of~\cite{EGnabeyeu2025,EGnabeyeuPR2025,Gnabeyeu2026b} and admit a power series expansion whose coefficients are defined inductively.
    
\medskip
\noindent {\bf Remark:}
In order to numerically implement the optimal strategy \eqref{eq:optimal_control_Heston2}, one needs to simulate  the non-Markovian process $(V,S)$ in Equation~\eqref{VolSqrt_}--~\eqref{eq:stocks} and to discretize the Riccati-Volterra equation  for $\psi$ in~\eqref{eq:Riccatipsi1}--~\eqref{eq:Riccatipsi2}. 	

\subsection{About the numerical scheme for the Volterra and fractional Riccati equations}
\noindent To simulate the Volterra process~\eqref{VolSqrt_}--~\eqref{eq:hestonS}, we use the $K$-integrated discrete time Euler-Maruyama scheme defined below on the time grid $t_k =t^n_k =\frac{kT}{n}, k=0, \dots, n$ and in the fractional kernel case, namely for \(i=1,2\), $\;\overline V^{i,n}_{0}=V^{i,0}$ and for every $k=1,\ldots,n$,  
\begin{align}\label{eq:EulerXdisc}
	\overline V^{i,n}_{t_{k}} 
	&= V^{i,0} +  \sum_{\ell=1}^{k} \Big(\big(\mu_0^i -\lambda_i\overline{V}_{t_{\ell-1}}^{i,n} \big) \int_{t_{\ell-1}}^{t_{\ell}} K_{\alpha_i}(t_{k} -s)ds   +  
    \sigma^v_i\varsigma^i(t_\ell)\sqrt{ \overline{V}_{t_{\ell-1}}^{i,n}} \int_{t_{\ell-1}}^{t_{\ell}} K_{\alpha_i}(t_{k} -s) \, dW_s \Big).
\end{align}
\smallskip
\noindent  The discretization involves both deterministic and stochastic integrals. Accordingly,  for each \(i=1,2\), let $C^i=(C_{k\ell}^i)_{k,\ell=1:n}$ denote the $n\times n$ lower triangular matrix associated with the deterministic integrals, and $G^i= (G_{k\ell}^i)_{ k=1:n+1, \ell=1: n}$ the  $(n+1)\times n$ matrix with entries the random terms $\int_{t_{\ell-1}}^{t_{\ell}}K_{\alpha_i}(t_{k}-s)dW_s^i$.
\begin{align*}
&(C_{k\ell}^i)_{k,\ell=1:n} 
= \left(
\int_{t_{\ell-1}}^{t_\ell} K_{\alpha_i}(t_k - s)\,ds \,\mathbf{1}_{\{1 \le \ell \le k \le n\}}
\right), \,(G_{k\ell}^i)_{ k=1:n+1}^{\ell=1: n}= \left(
\begin{array}{l}
\int_{t_{\ell-1}}^{t_\ell} K_{\alpha_i}(t_k - s)\,dW_s^i \,\mathbf{1}_{\{1 \le \ell \le k \le n\}} \\[0.3em]
\Delta W_{t_\ell}^i, \quad k = n+1,\; \ell=1:n
\end{array}
\right),
\end{align*}
The last line in \(G^i\) line has been introduced in order to be able to perform a consistent joint simulation of $\bar V^{n}$ and the Euler-Maruyama scheme of the Markovian stock price process $S$ in~\eqref{eq:stocks}--\eqref{eq:hestonS} and  
wealth process $X$ in~\eqref{eq:Tilwealth} depending on $\bar V^{n}$ and $W$ (see e.g. \cite{GnabeyeuPages2026} for further details), given recursively by
\begin{align}\label{eq:xi_bar^N}
	\overline{X}^{\alpha,n}_{t_{k}}&=\overline{X}^{\alpha,n}_{t_{k-1}}+\big(r(t_{k})\overline{X}^{\alpha,n}_{t_{k-1}}+\alpha_{t_{k-1}}^\top \lambda_{t_{k-1}} \big)\frac{T}{n} +\alpha_{t_{k-1}}^\top \big(\Sigma^\top  \Delta W_{t_k} - \sqrt{I-\Sigma^\T \Sigma} \Delta W^{\perp}_{t_k}\big)\\  
	&=\overline{X}^{\alpha,n}_{t_{k-1}}+\big(r(t_{k})\overline{X}^{\alpha,n}_{t_{k-1}}+\sum_{i=1}^{d=2}\theta_i\pi^i_{t_{k-1}}V^i_{t_{k-1}} \big)\frac{T}{n} +\sum_{i=1}^{d=2}\big(\rho_i  \Delta W^i_{t_k} - \sqrt{1-\rho_i^2} \Delta W^{i,\perp}_{t_k}\big)\pi^i_{t_{k-1}}\sqrt{V^i_{t_{k-1}}}.\nonumber
\end{align}
Here, we deduce in particular that from the correlation structure~\eqref{eq:correstructureheston}, \(B_t
=
\Sigma^\top  W_t
-
\sqrt{I-\Sigma^\top \Sigma}\, W_t^{\perp}\) with \(W\perp\!\!\!\perp W^\top\) two  independent $d$-dimensional Brownian motions.

\smallskip
\noindent 
Then the following relation holds for \(i=1,2\): $\;\overline V^{i,n}_{0}=V^{i,0}$ and for every $k=1,\dots,n$
\begin{equation}\label{eq:VectSimul}
\big( \overline V^{i,n}_{t_{k}} \big)_{k=1:n}= V^{i,0}\mbox{\bf 1}_n +C_{k,\cdot}^i\big(\mu_0^i -\lambda_i\overline{V}_{t_{j-1}}^{i,n} \big)_{j=1:n} +G_{k,\cdot}^i\big( \sigma^v_i \varsigma^i(t_{j}) \sqrt{ \overline{V}_{t_{j-1}}^{i,n}})_{j=1:n}. 
\end{equation}
\noindent which is simulated on the discrete  grid \((t^n_k)_{0\leq k\leq n}\) by generating an independent sequence of gaussian vectors \( G^i_{k,\cdot}, k=1 \cdots n\) using an extended and stable version of Cholesky decomposition of a well-defined covariance matrix.
The reader is referred to \cite[Appendix A]{EGnabeyeu2025, GnabeyeuPages2026} for further details about the simulation of the Gaussian stochastic integrals terms in the semi-integrated Euler scheme introduced in this context for Equation~\eqref{VolSqrt2}-\eqref{VolSqrt_}. Theoretical guarantees for the convergence of this numerical scheme, as well as the convergence rate are established 
for more general Kernels and path-dependent coefficients in~\cite{GnabeyeuPages2026}.


\medskip
\noindent
To design an approximation scheme for solving the two-dimensional Riccati--Volterra system of equations 
~\eqref{eq:Riccatipsi1}--~\eqref{eq:Riccatipsi2} numerically,
we employ as in~\cite[section 4.1]{dro2026optimal} the generalized Adams--Bashforth--Moulton method, often referred to as the \textit{fractional Adams method} investigated in~\cite{DiethelmFordFreed2002,DiethelmFordFreed2004} as a useful numerical algorithm for solving a fractional ordinary differential equation (ODE) based on a predictor-corrector approach.



\medskip
\noindent {\bf Remark:} Denote the fractional integral 
 as \(I^{\alpha_i}\psi(t) =
K_{\alpha_i} \star \psi(t)\), then we show by integration by parts that, equation~\eqref{eq:Gamma0} reads:
\begin{equation}\label{eq:Gamma_0Const}
	\Gamma_0= \exp\Big( 2\int_0^T r(s)  
    ds +  \sum_{i=1}^d V_0^i\,I^{1-\alpha_i}\psi^i(T) +  \sum_{i=1}^d \mu^i_0 I^{1}\psi^i(T) \Big).
\end{equation}
Consequently, it is possible to make use of the open-source Python package \textit{differint} to compute the fractional integrals \(I^{1-\alpha_i}\) and \(I^{1}\) in equation~\eqref{eq:Gamma_0Const} for each \(i=1,2\).
\noindent Recall that the fractional integral of order \( r \in (0, 1] \) of a function \( f \) is \(I^r f(t) = \frac{1}{\Gamma(r)} \int_0^t (t - s)^{r - 1} f(s) \, ds,\)
	whenever the integrals exist.
\subsection{Numerical illustrations of the $2D$ fake stationary rough Heston volatility}
	\noindent 
    In this section, we provide graphical illustrations of the two-dimensional fake stationary rough Heston volatility described at the beginning of Section~\ref{Sec:Num}.
\noindent We consider as in~\cite{Gnabeyeu2026b} the following estimates for the model parameters: The initial variance \(V_0\) defined in~\eqref{VolSqrt_} is such that  \(V_0^i \sim \text{Gamma}(\frac{(m_0^i)^2}{\Omega_{ii}},\frac{\Omega_{ii}}{m_0^i})\) and 
\[
c = \begin{pmatrix}
	0.01 \\
	0.03
\end{pmatrix}, \;
\mu_0 = \begin{pmatrix}
	2.0 \\
	1.0
\end{pmatrix},\;
D = \begin{pmatrix}
	-0.20 & 0 \\
	0 & -0.20
\end{pmatrix}, \quad
\Sigma = \begin{pmatrix}
	-0.7 & 0 \\
	0 & -0.55
\end{pmatrix}, \;
\theta = \begin{pmatrix}
	0.1 \\
	0.12
\end{pmatrix}, \;
\sigma^v = \begin{pmatrix}
	0.40 \\
	0.32
\end{pmatrix}.
\]
Moreover, we set the risk-free rate $r = 0.02$, the initial
wealth $x_0 = 2$ and the expected terminal wealth $m= x_0 e^{(r+0.1)T} = 2.255$.
\begin{figure}[H]
	\centering
	\includegraphics[width=0.93\linewidth]{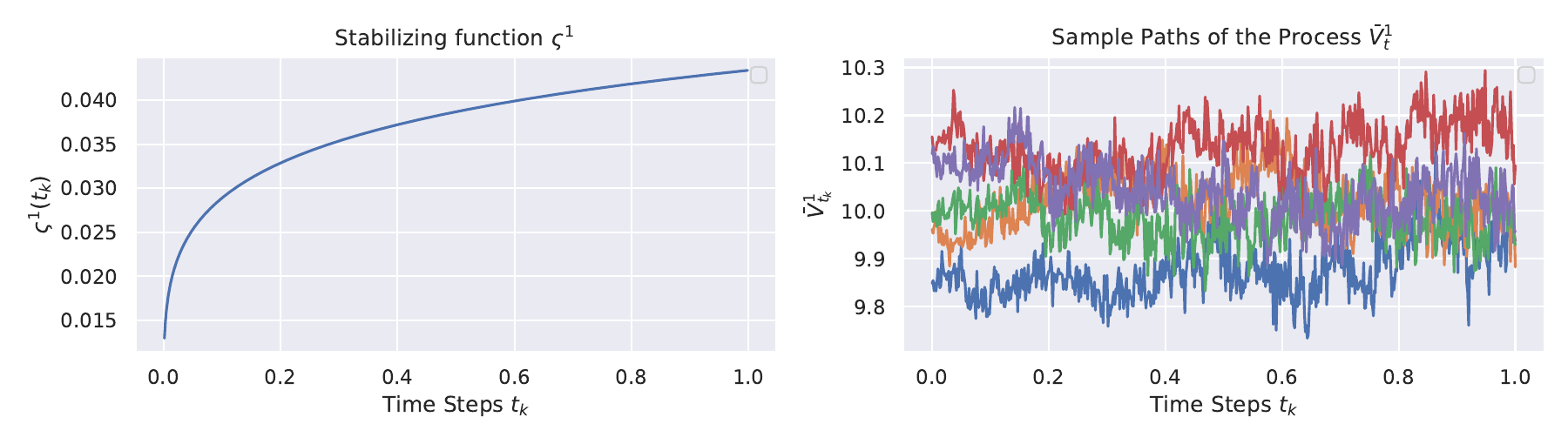
	}
	\caption{\textit{Graph of the stabilizer $ t \to \varsigma_{\alpha_1,\lambda_1,c_1}(t)$ (left) and \(5\) samples paths \( t_k \mapsto V^1_{t_k} \) (right) over the time interval \( [0, 1] \), for the Hurst esponent \( H = 0.1 \), \( c_1 = 0.01 \) and number of time steps \( n = 600 \).}}\label{fig:stabil_Mean1}
\end{figure}
\begin{figure}[H]
	\centering
	\includegraphics[width=0.93\linewidth]{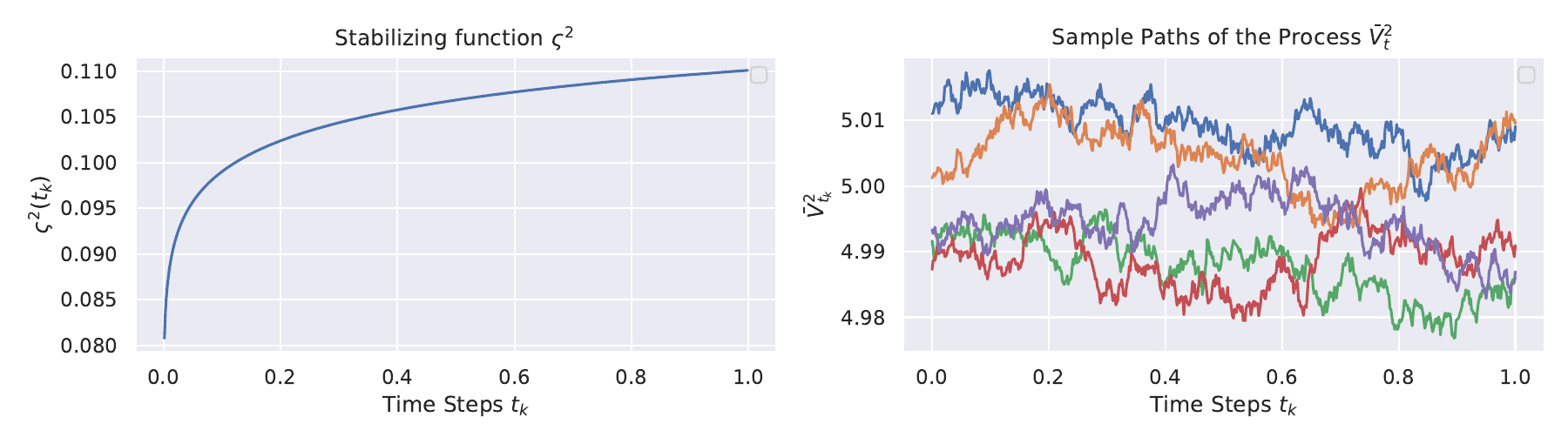
	}
	\caption{\textit{Graph of the stabilizer $ t \to \varsigma_{\alpha_2,\lambda_2,c_2}(t)$ (left) and \(5\) samples paths \( t_k \mapsto V^2_{t_k} \) (right) over the time interval \( [0, 1] \), for the Hurst esponent \( H = 0.4 \), \( c_2 = 0.03\) and number of time steps \( n = 600 \).}}\label{fig:stabil_Mean2}
\end{figure}
\noindent The stabilizing functions $\varsigma\equiv\varsigma_{\alpha,D,c}$ defined by $\textit{($E_{\alpha_i, \lambda_i, c_i}$)}$ in~\eqref{eq:VolterraStabilizer2}
			 is designed to ensure that the volatility processes $(V^1, V^2)$ exhibit constant marginal means (Figures~\ref{fig:stabil_Mean1}--\ref{fig:stabil_Mean2}) and variances (cf.~ Figures~\ref{fig:_variance1}--\ref{fig:_variance2} below) over time, thereby ensuring invariance of the first two moments under time shifts.
\begin{figure}[H]
	\centering
	\includegraphics[width=0.91\linewidth]{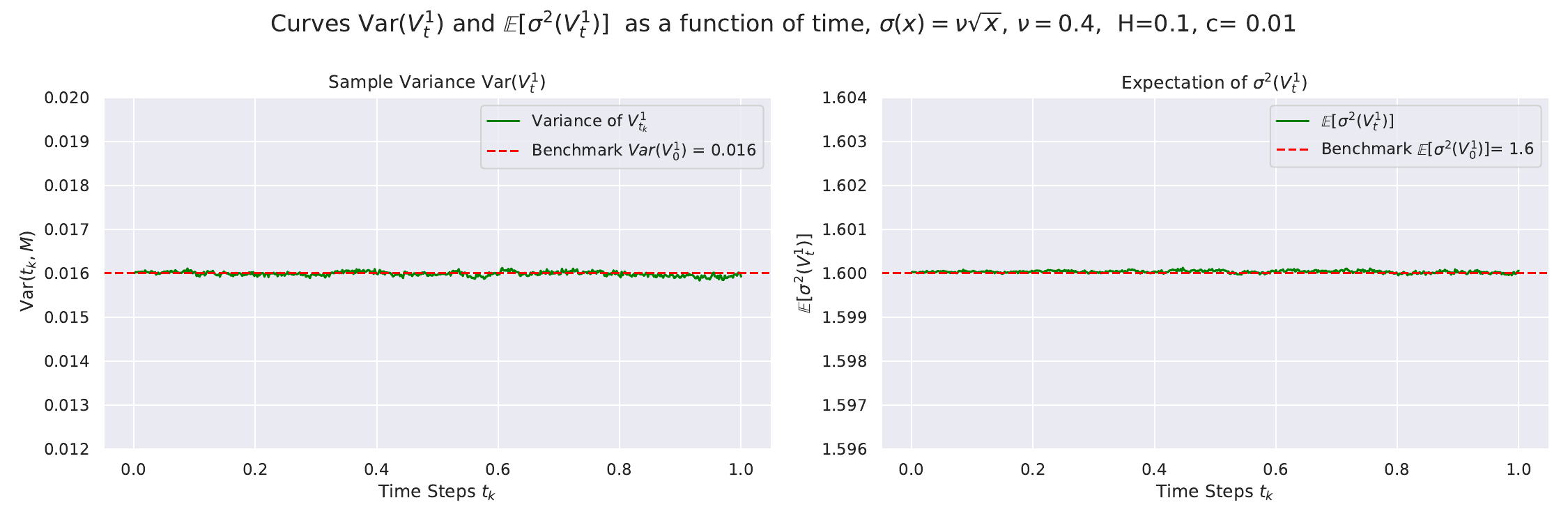}
	\caption{\textit{Graph of \( t_k \mapsto \text{Var}(V^1_{t_k}, M) \) and \( t_k \mapsto \mathbb{E}[\sigma^2(V^1_{t_k},M)] \) over \( [0, 1] \), \( c_1 = 0.01 \) and \( n = 600 \).}}\label{fig:_variance1}
    \vspace{-.3cm}
\end{figure}
\vspace{.2cm}
\begin{figure}[H]
	\centering
	\includegraphics[width=0.91\linewidth]{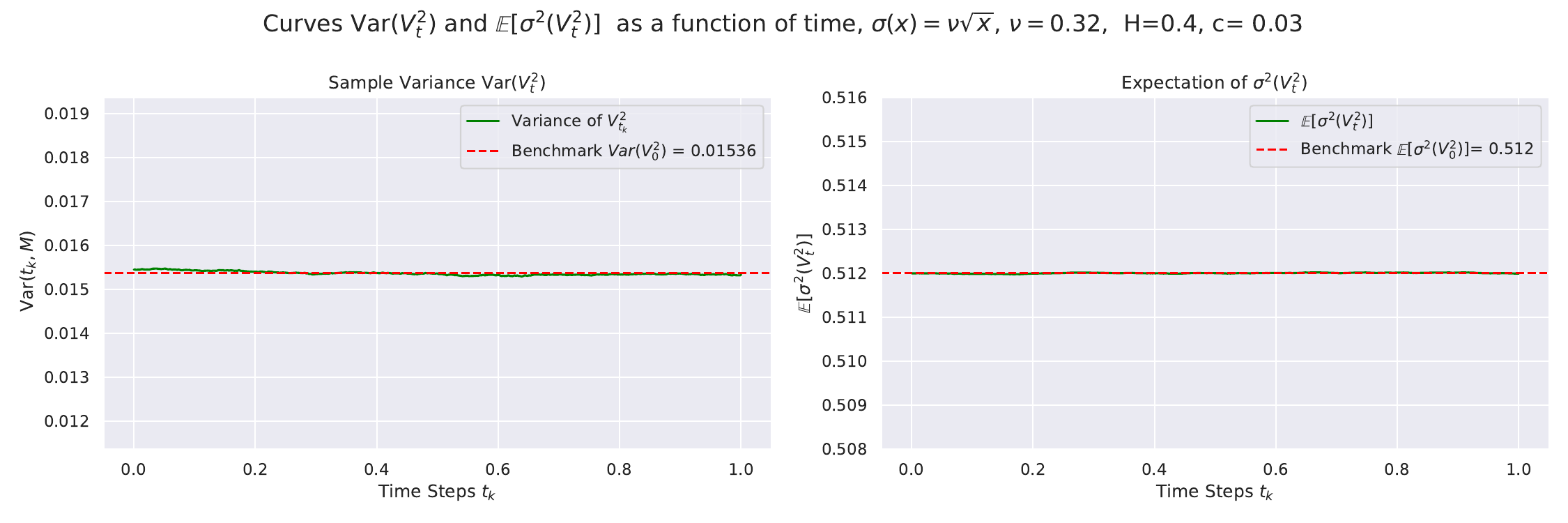}
	\caption{\textit{Graph of \( t_k \mapsto \text{Var}(V^2_{t_k}, M) \) and \( t_k \mapsto \mathbb{E}[\sigma^2(V^2_{t_k},M)] \) over \( [0, 1] \), \( c_2 = 0.03 \) and \( n = 600 \).}}\label{fig:_variance2}
\end{figure} 

\vspace{.2cm}
 \subsection{Graphical illustrations of the optimal strategy and the efficient frontier}
 \noindent In this section, we illustrate the results of Section~\ref{Sec:Markowitz} by numerically computing the optimal portfolio strategy for a special case of two-dimensional fake stationary rough heston model. 
\begin{figure}[H]
	\centering
	\begin{minipage}{0.49\textwidth}
		\centering
		\includegraphics[width=0.98\textwidth]{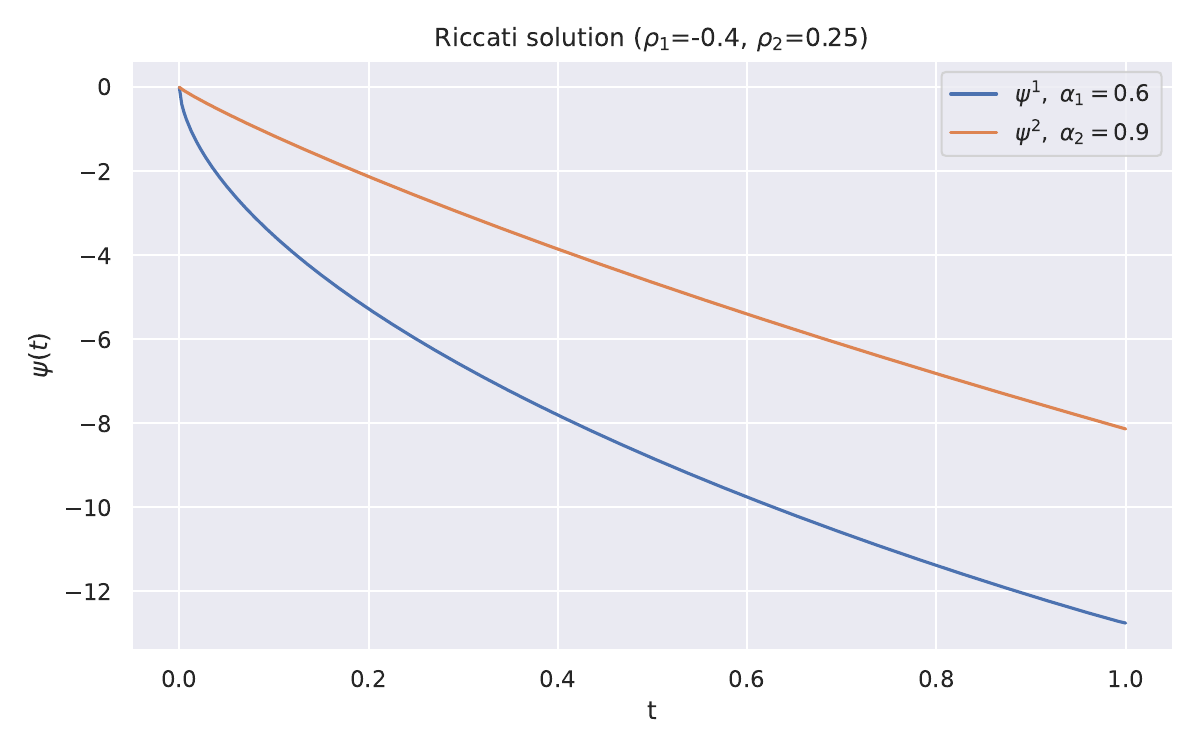}
		\subcaption{\textit{ Graph of \( t_k \mapsto \psi_{t_k} \) with risk premium parameter\\ $\theta = (3.6,3.0)$ and correlation $|\rho_i| \leq \frac12$.}}\label{Fig:RiccatiA}
	\end{minipage}%
	\begin{minipage}{0.49\textwidth}
		\centering
		\includegraphics[width=0.98\textwidth]{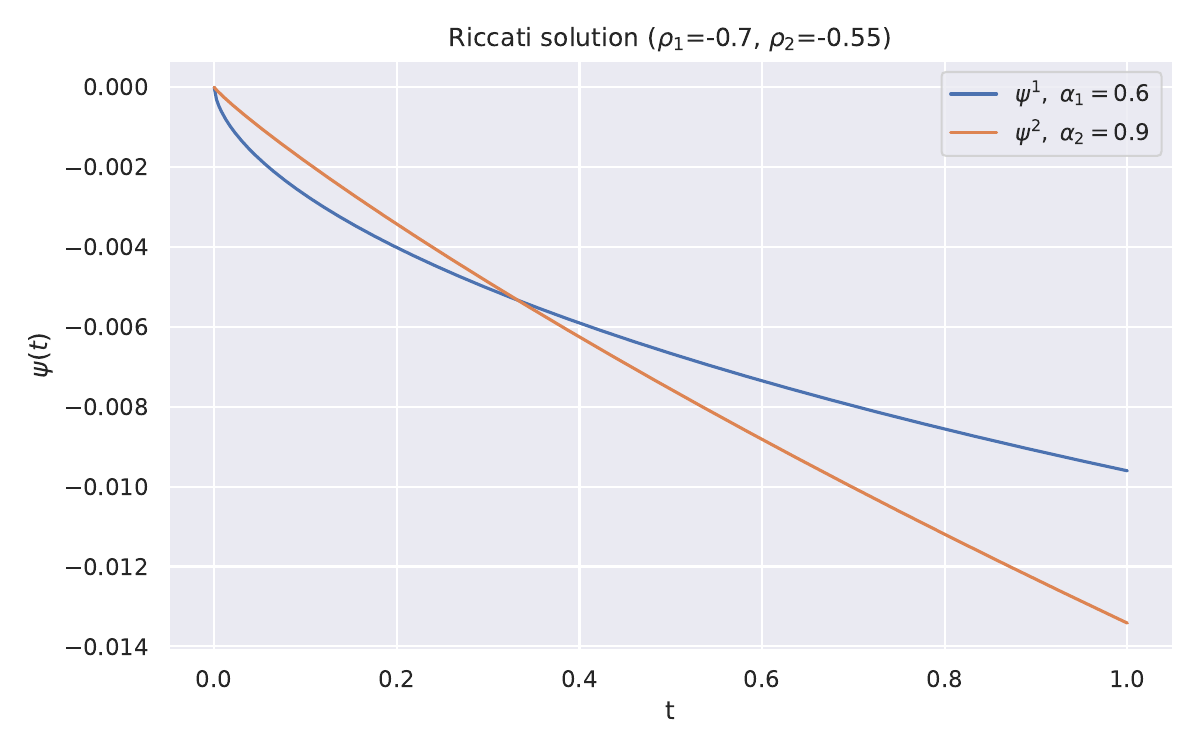}
		\subcaption{\textit{ Graph of \( t_k \mapsto \psi_{t_k} \) with risk premium parameter\\ $\theta = (0.1,0.12)$ and correlation $|\rho_i| > \frac12$.}}\label{Fig:RiccatiB}
	\end{minipage}
	\caption{\textit{ Graph of \( t_k \mapsto \psi^1_{t_k} \) and \( t_k \mapsto \psi^2_{t_k} \) over \( [0, 1] \) with the fractional Adams algorithm and the number of time steps \( n = 600 \) both for correlation verifying $|\rho_i| \leq \frac12$ (left) and $|\rho_i| > \frac12$ (right), \(i\in\{1,2\}\).}
}
\end{figure}
\noindent Figures~\ref{Fig:RiccatiA}--\ref{Fig:RiccatiB} also confirm the Remark on Proposition~\ref{prop:MV_riccati}, that is the claim that $\psi \leq 0$ whenever $1 - 2\rho_i^{2} \geq 0$ or $1 - 2\rho_i^{2} < 0$ for every \(i\in\{1,2\}\).

\bigskip
\noindent We consider in Figure~\ref{fig:Soptstrat}, the evolution of the optimal strategy of each stock, that is, the mapping \(t \mapsto \alpha_t^*,\) along a single initial variance \(\bar{V}_0=(\frac{\mu_0^1}{\lambda_1}, \frac{\mu_0^2}{\lambda_2})\) and the evolution of the associated wealth process given by the mapping \(t \mapsto X_t^{\alpha^*}\).
\begin{figure}[H]
	\centering
	\includegraphics[width=0.98\linewidth]{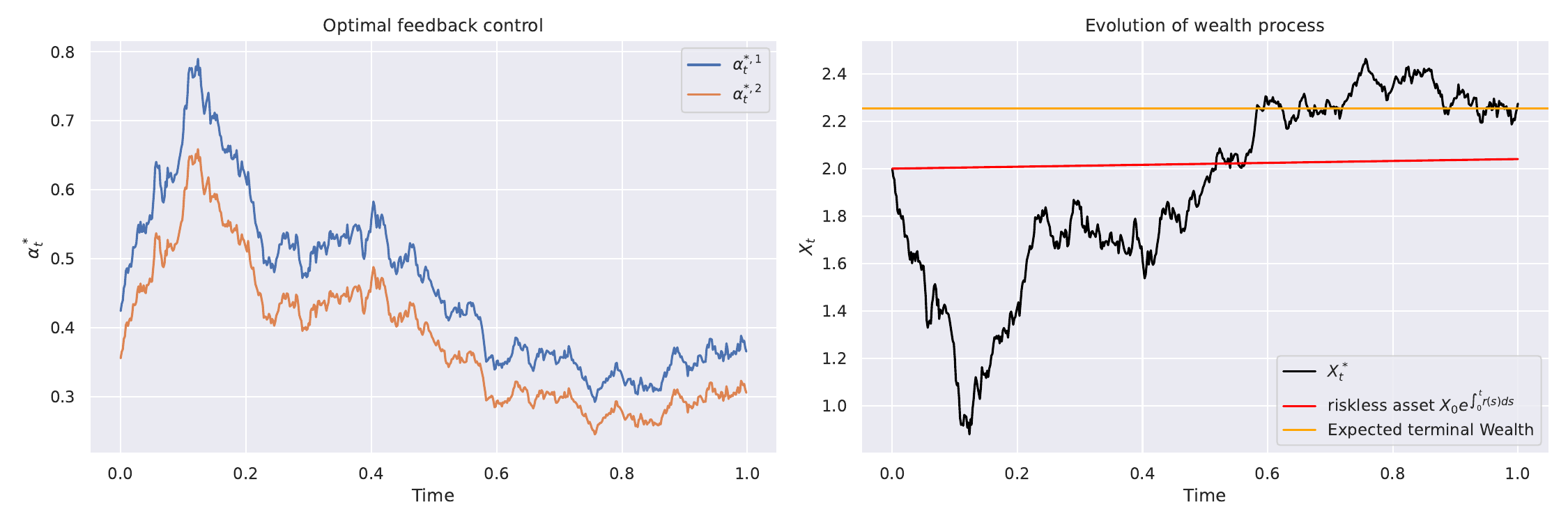}
	\caption[Optimal startegies and Wealth.]{\textit{Graph of \( t_k \mapsto \alpha_{t_k}^* \) (left) and \( t_k \mapsto X_{t_k}^{\alpha^*} \) (right) over the time interval \( [0, T] \), \( T = 1 \). Number of steps: \( n = 600 \).}}\label{fig:Soptstrat}
\end{figure}

\noindent Since the optimal strategies $(\alpha^*_t)_{t\in[0,T]}$ given by~\eqref{eq:optimal_control_Heston2} are stochastic processes as well as the corresponding wealth process~\eqref{eq:wealth}--\eqref{eq:wealthProcess}, we rather consider in what follows, the evolution of the associated deterministic mapping \(t \mapsto \mathbb{E}[\alpha_t^*]\) and \(t \mapsto \mathbb{E}[X_t^{\alpha^*}]\)
\begin{figure}[H]
	\centering
	\includegraphics[width=0.95\linewidth]{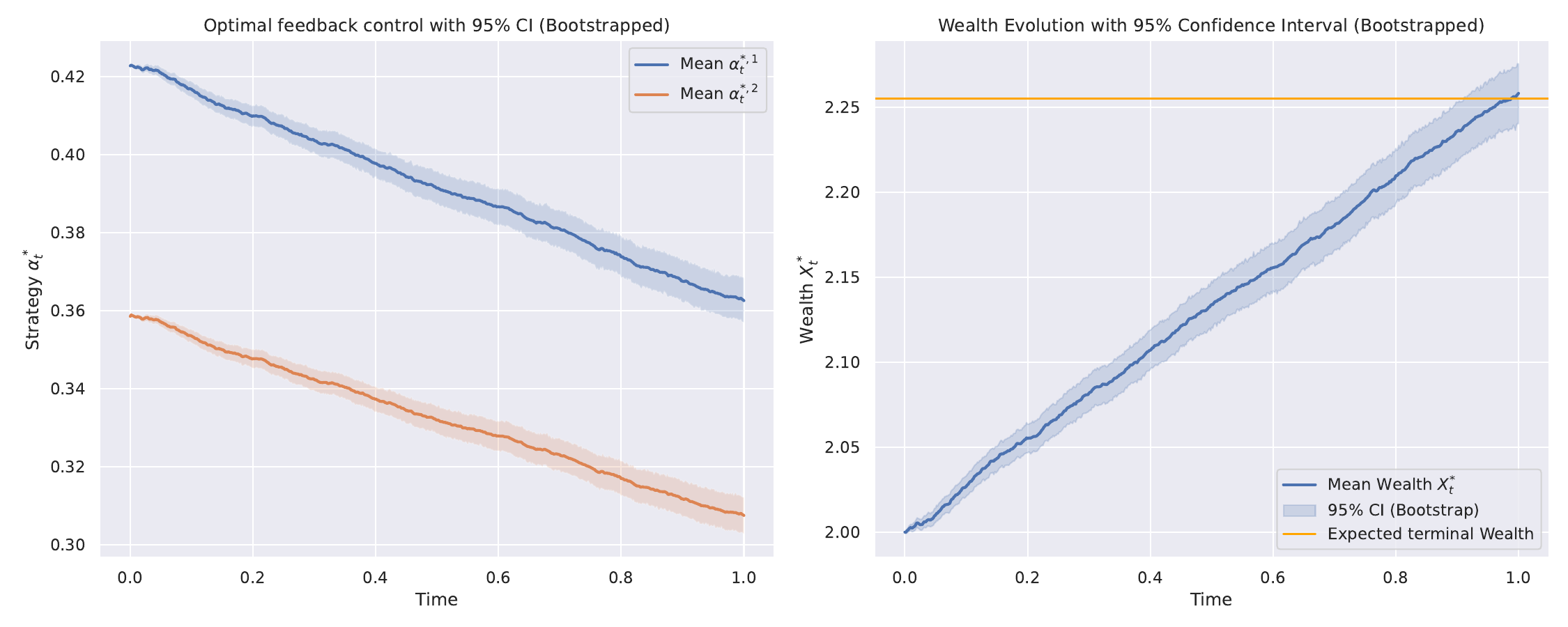}
	\caption{\textit{Statistics for strategies and wealth. Based on \(5000\) simulated paths, the solid line plots
		the mean and the shadow area is the \(95\%\) confidence interval estimated by bootstrapping. The terminal wealth is closer to the expected
		value \(m = 2.255\). The parameters are the same as in Figure~\ref{fig:Soptstrat}.}}\label{fig:Soptstrat_mc}
\end{figure}
\noindent As illustrated in Figure~\ref{fig:Soptstrat}--\ref{fig:Soptstrat_mc}, the fake stationary rough Heston strategy considered in these examples achieves an average terminal wealth closer to the target \(m := 2.255\).

\noindent The form of the
control~\eqref{eq:optimal_control_Heston}--~\eqref{eq:optimal_control_Heston2} is also dictated by the stabilizing function in~\eqref{eq:VolterraStabilizer2}, designed to ensures consistency across different time scales.


\medskip
\noindent Although rough volatility models excel in capturing high-frequency dynamics with remarkable accuracy, their inherent lack of stationarity renders them unsuitable for long-term investment analysis. Consequently, this framework ~\eqref{VolSqrt_}--\eqref{eq:VolterraStabilizer2}(\textit{stabilized models}) is pivotal for formulating a well-defined infinite-horizon investment problem, as well as for time-series analysis and long-term econometric estimation.
    \begin{figure}[H]
	\centering
	\subfloat[$T = 0.5$]{
		\label{efficient_frontier_small_T}
		\includegraphics[width=55mm]{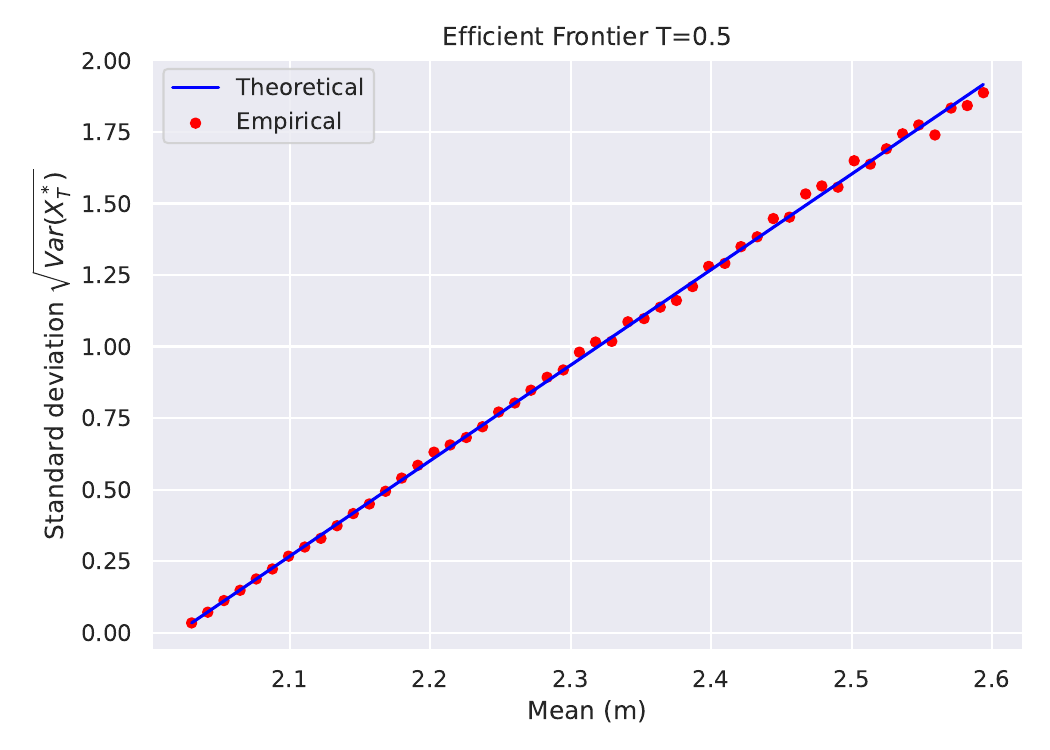}} 
	\subfloat[$T = 1.0$]{
		\label{efficient_frontier_medium_T}
		\includegraphics[width=55mm]{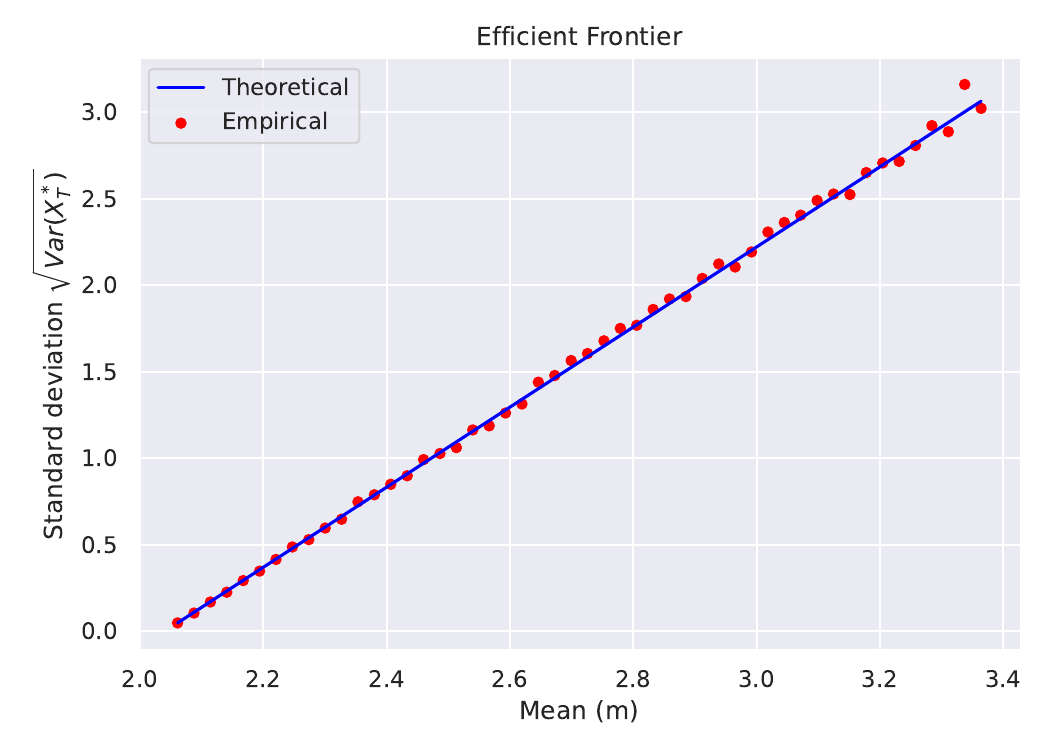} 
	}
	\subfloat[$T = 5.0$]{
		\label{efficient_frontier_big_T}
		\includegraphics[width=55mm]{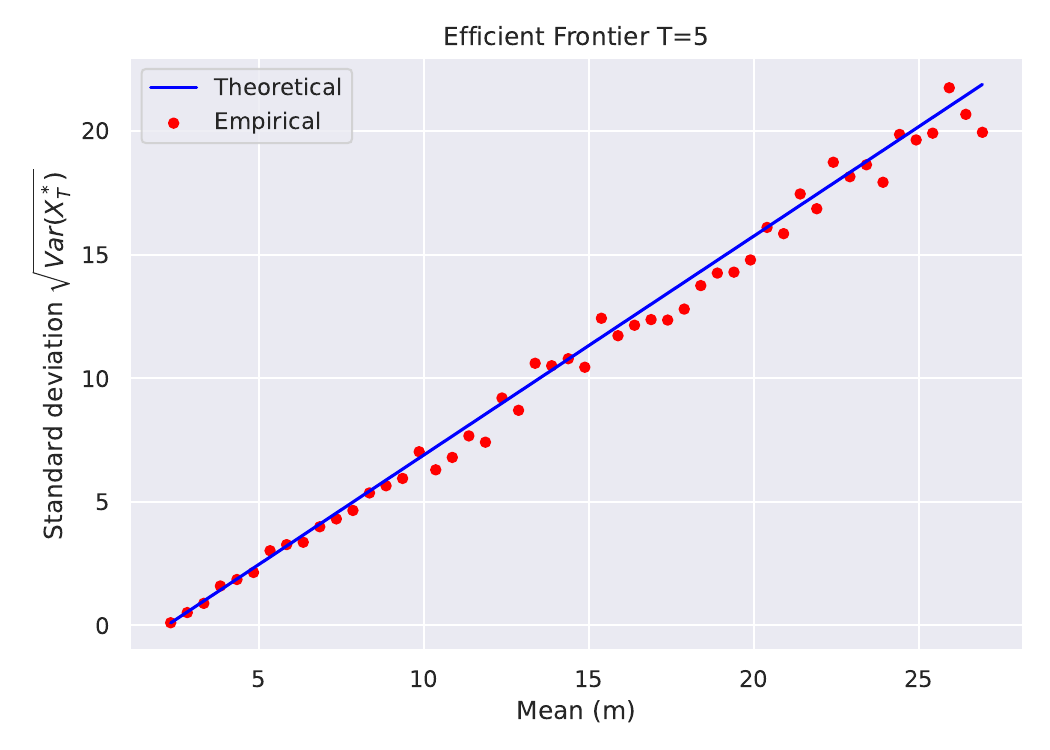}} 
	
	\caption[Efficient frontier.]{\textit{Plots of the efficient frontier and variance.
			We set the parameters $x_0 = 2$, $T \in \{0.5, 1.0, 5.0\}$, and $m \in \left[x_0 e^{(r+0.01)T},\, x_0 e^{(r+0.5)T}\right]$.
			The solid curve corresponds to the theoretical frontier in~\eqref{eq:effFront}, while the dotted curves represent the variance obtained from the Monte Carlo simulations of \(X_T^{\alpha^*}\).}}
	\label{fig:EfficientFront}
\end{figure}
\noindent
 We observe that, when the horizon $T$ is small,  the variance of the terminal wealth \(X_T\) given in~\eqref{eq:value_final} tends to be smaller.

    
\section{Additional technical Proofs of the results}\label{sect:proofMresult} 

\subsection{Proof of Proposition \ref{prop:MV_riccati}}

The BSDEJ~\eqref{eq:bsde_mv} is a simplified version of~\eqref{eq:BSDEJ_general}-~\eqref{eq:Driver_General} in which we take $c_y\equiv g_y\equiv0$, $g_{z\sqrt{x}}\equiv g_x\equiv0$ and $g\equiv g_\nu\equiv h$ in equation~\eqref{eq:ExpCom_functs} for all $t\leq T$. 
By Theorem~\ref{Thm_BDSEJ_Gen} together with \(Y_T=0\), we get that $\left(Y,\Lambda, U\right)$ as defined in ~\eqref{eq:Solbsde_mv} is the unique $\mathbb{D}([0,T], \R) \times L^2_{\F}([0,T], \R^d)\times L^2_{\F}([0,T], \tilde{N})$-valued solution triplet to the BSDEJ~\eqref{eq:bsde_mv} provided uniqueness of a continuous solution to the inhomogeneous Riccati--Volterra equation with jump compensator \eqref{eq:Riccatipsi1}-\eqref{eq:Riccatipsi2} for $\psi$

 {It remains to show that $Y \in \mathbb{S}^{p}_{\F}([0,T],\R)$ for some sufficiently large \(p>1\). 
	First, notes that we have the representation \(Y_t:=\E \Big[ - \int_t^T  f\big(Y_s, \Lambda_s, U_s\big) ds \mid \mathcal F_t\Big]\) where \(Y_T=0\) and \(f\) is the driver of \eqref{eq:bsde_mv}. 
	Consequently, we may write:
	\[\forall\, t\in[0,T], \quad Y_t\leq \E \Big[ -2r(t) + \int_0^T \left| \lambda_s + \Sigma \Lambda_s \right|^2 ds \mid \mathcal F_t\Big]:=M_t\]
	Using the elementary inequality \((a+b)^p \le 2^{(p-1)^+}\big(a^p + b^p\big)\) for \(a,b>0\), one gets
	\begin{equation}\label{eq:boundOpt}
		|\lambda_s +  \Sigma\Lambda_s|^2 \leq 2(|\lambda_s|^2 +  |\Sigma\Lambda_s|^2) \leq 2(1 + |\Sigma|^2)(|\lambda_s|^2 +  |\Lambda_s|^2) \leq 2(1 + |\Sigma|^2)(|\lambda_s|^2 +  |\Lambda_s|^2).
	\end{equation}
	Consequently, by simple calculation, using the inequality $|z|^q\leq c_{q}e^{|z|}$ for all $z \in \R$, together with  Cauchy-Schwarz inequality in the third line:
	\begin{align}
		&\,\E \Big[ -2r(t) +\int_0^T \left| \lambda_s+ \Sigma \Lambda_s \right|^2ds \Big] \leq c_{1} \E \big[ \exp\big(|M_T|\big)\big]\label{eq:boundExpoM_T}\\
		&\hspace{1.2cm}\leq c_{1}\E \Big[ \exp\Big(\int_0^T2|r(s)| ds+ (1 + |\Sigma|^2)\int_0^T \big(|\lambda_s|^2 +  |\Lambda_s|^2\big)ds\Big) \Big]\nonumber\\
		&\hspace{1.2cm}\leq c_{1} \Big(\E \Big[ \exp\Big(4\int_0^T |r(s)| ds\Big) \Big]\Big)^\frac12\Big(\E \Big[ \exp\Big(2(1 + |\Sigma|^2)\int_0^T \big(|\lambda_s|^2 +  |\Lambda_s|^2\big)ds\Big) \Big]\Big)^\frac12<\infty\nonumber
	\end{align}
	thanks to the Novikov-type condition~\eqref{eq:assumption_novikov} with constant \(2a = 2(1 + |\Sigma|^2)\) , together with the exponential affine representation in ~\cite[Theorem A.1.]{dro2026optimal}.  Therefore, $M_t$ is a martingale under $\P$.  
	Let $ p > 1 $.
	By {virtue of Doob's maximal inequality} together with the inequality $|z|^q\leq c_{q}e^{|z|}$  for all $z \in \R$: 
	\begin{align*}
		&\, \E \Big[ \sup_{ t \in [0, T]} \big| Y_t \big|^p \Big] \leq \E \Big[ \sup_{ t \in [0, T]} \big| M_t \big|^{p}\Big] \leq \Big(\frac{p}{p-1}\Big)^{p} \E \Big[ |M_T|^p\Big]\leq c_{p} \Big(\frac{p}{p-1}\Big)^{p} \E \big[ \exp\big(|M_T|\big)\big]< \infty
	\end{align*}
	which is finite owing to ~\eqref{eq:boundExpoM_T}.  Therefore, $\E\big[ \sup_{ t \in [0, T]} |Y_t|^p \big] < \infty $ holds. 

    \smallskip
	\noindent Furthermore, as a direct consequence of \cite[Lemma~5.1]{dro2026optimal} with $g_{2} = 1$ and $g_{1}(s) = 0$, the stochastic exponential \(M\) in~\eqref{eq:localMart_M} where \(Z=\Lambda\) is a true $\P$- martingale so that the representaton~\eqref{eq:ito_gamma}
	 ensures that $0<\Gamma_t\leq e^{2 \int_t^T r(s) ds}$, $\P-a.s.$, since $V\in \mathbb R^d_+$, so that \(\Gamma \in \mathbb{S}^{\infty}_{\F}([0,T],\R)\). Therefore, it follows that $\left(\Gamma,\Lambda, U\right)$  is a  $\mathbb{S}^{\infty}_{\F}([0,T], \R) \times L^2_{\F}([0,T], \R^d)\times L^2_{\F}([0,T], \tilde{N})$-valued solution to the Riccati backward stochastic differential equation with jumps (BSDEJ)~\eqref{eq:GammaDef2}.
	
	In addition,  \eqref{eq:ito_gamma} implies that $\Gamma_0 < e^{2\int_0^T r(s) ds}$ since $g^i_0(0)>0$ for some $i\leq d$ by assumption and $V^i$ is continuous and positive.\\
	\noindent
	To derive equation~\eqref{eq:Gamma0},  recalling~\eqref{VolSqrt2}--\eqref{eq:processg}, we have that the initial adjusted forward process curve \( s \mapsto g^i_0(s) \) verifies for \(i = 1, \ldots, d\), \(g^i_0(s) = V_0^i \varphi^i(s) + \int_0^s K(s-u) \mu(u) \, du\). Applying regular Fubini's theorem, together with a change of variables, and using equation~\eqref{eq:Riccatipsi1} in the last line, we obtain successively (here we set, \(\forall\, s\in[0,T], \; g_i (s,\psi):=-\theta_i^2  + F_i(s,\psi)\) for every \(i=1,\ldots,d\)):
	\begin{align*}
		&\,\int_0^T  g_i(s,\psi(T-s)) \int_{0}^{s}K(s-u)\mu(u)\,du ds = \int_0^T  \int_{u}^{T}K(s-u) \big(-\theta_i^2  + F_i(s,\psi(T-s))\big)  ds  \mu^i(u) \,du \\
		&\hspace{6cm}= \int_0^T  \int_{0}^{T-u}K(T-u-s) \big(-\theta_i^2  + F_i(T-s,\psi(s))\big)  ds  \mu^i(u) \,du \\
		&\hspace{6cm}= \int_0^T  \psi(T-u) \mu^i(u) \,du
	\end{align*}
	This complete the proof of the proposition. \hfill $\Box$

	\subsection{Proof of Theorem~\ref{thm:inner} }\label{subsect:proofMresultMarkowitz2}
	\noindent {\sc Step~1} (\textit{Solution of the inner Problem:})
	Let's us consider the inner Problem~\eqref{pb:P(c)} with an arbitrary 
	\(\xi \in\R\) and define $\Tilde{X}^{\alpha}_t = X_t^{\alpha} - \xi e^{-\int_t^T r(s) ds}$, for any $\alpha \in \mathcal{A}$ as in the preamble of Section~\ref{Sec:Markowitz}. 
    Then, by It\^o's lemma we have that $\Tilde{X}^{\c}$ satisfies~\eqref{eq:Tilwealth}.  In particular,
	\(\tilde{X}^{\c}\) and \(X^{\c}\) have the same dynamics, with
	\(\tilde{X}^{\c}_T = X^{\c}_T - \xi\). Consequently, problem~\eqref{pb:P(c)}
	reduces to~\eqref{pb:AlterP(c)}.

        We show that $J_t^{\alpha}$ defined in~\eqref{eq:ansatz} fulfills the martingale optimality principle in the sketch of proof. For the first condition, note that $Y_T=0$ and hence \(\Gamma_T=1\) and $J_T^{\alpha}=|\tilde{X}^{\alpha}_T|^2$. Since $\exp(Y_0)$ is a constant independent of $ \alpha \in \cA$, $J^\alpha_0 = \exp(Y_0) \Big|x_0-\xi e^{-\int_0^T r(s)ds}\Big|^2$ is a constant independent of $ \alpha \in \cA$ and consequently, the second condition is also satisfied. In order to show that the third condition is fulfilled, we first note that by Propositions~\ref{prop:MV_riccati} and~\ref{prop:preambule}, $J^\alpha$ is a submartingale for any admissible strategy $\alpha\in\mathcal{A}$.
    Moreover, we obtain that a candidate for the optimal control $\c^*(\xi)$ is given by~\eqref{eq:optimal_control}. In particular, $J^{\alpha^*}$ is a martingale by~\cite[Lemma 5.1]{dro2026optimal}, and all conditions of the
	martingale optimality principle are satisfied, except for the admissibility of the
	candidate optimal strategy $\alpha^*$ , which follows from Steps~2 and~3 below.	
    
	\smallskip
	\noindent {\sc Step~2} (\textit{Existence and uniqueness of $\c^*(\xi)$,  $\xi \in \R$ being fixed:})
	First note that, the uniqueness of  $\c^*(\xi)$  follows directly from equation~\eqref{Eq:Square} and  non-degeneracy of \(\Gamma\) ( $\Gamma_t>0$, $\p$-$\as$, $\forall \; t \in [0, T]$). 
	The existence of a control $\c^*(\xi)$ satisfying  ~\eqref{eq:optimal_control} is also guaranteed by the existence of the solution \(X^{\alpha^*}\). For this, we prove that  the corresponding wealth equation~\eqref{eq:wealth}
	admits a solution.  As in the proof of Step~1 above, it is enough to consider the modified equation 
	\begin{align*}
		d\Tilde X_t^* &= \big(r(t) \Tilde X_t^*  + \lambda_t^\top A_t \Tilde X_t^*  \big)dt + \big(A_t   \Tilde X_t^* \big)^\top dB_t, \quad
		\Tilde X_0^* \; = \; x_0 -\xi e^{-\int_0^T r(s) ds},
	\end{align*}
	where $A_t = - \left(\lambda_t + \Sigma \Lambda_t\right)$, and then set  $X_t^{*}= \Tilde{X}^{*}_t + \xi e^{-\int_t^T r(s) ds}$. By virtue of  It\^o's lemma the unique continuous solution is given by    
	\begin{equation}
		\Tilde{X}^{*}_t =\Tilde{X}^{*}_0 \exp\Big(\int_0^t \big(r(s) + \lambda_s^\T A_s - \frac{A_s^\T A_s}{2} \big)ds + \int_0^t A_s^\T dB_s\Big).
	\end{equation}
	Setting $\alpha^*_t(\xi)$ $:=$ $A_t\Tilde X^*_t$, we obtain that $\alpha^*(\xi)$ satisfies \eqref{eq:optimal_control} with the controlled wealth $X^{\alpha(\xi)^*}=X^*$. The crucial step is now to  
	obtain the admissibility condition \eqref{eq:estimateX}. For that purpose, observe by virtue of  \eqref{eq:assumption_novikov}, that the Dol\'eans-Dade exponential $\mathcal E\left( \int_0^{\cdot} A_s^\top dB_s\right)$ satisfies Novikov's condition, and is therefore a true martingale. 
	Whence, successive applications of the arithmetic mean-geometric mean (AM-GM) inequality $ab\leq (a^2+b^2)/2 $ together with Doob's maximal inequality yield, for some $p\geq1$,
	\begin{align*}
		\E \Big[ \sup_{t \in [0,T]} |\tilde{X}^*_t|^p \Big] &\leq  \frac{x_0^p}{2}  \left(\E \Big[ \sup_{t \in [0,T]}  \big|  e^{\int_0^t {\left(r(s) + \lambda_s^\T A_s \right)}ds} \big|^{2p}  \Big] 
		+ \E \Big[ \sup_{t \in [0,T]} \Big| e^{-\int_0^t \frac{A_s^\T  A_s}{2}  ds + \int_0^t A_s^\T dB_s} \Big|^{2p}  \Big]\right) \\
		&\leq \frac{x_0^p}{2} \left(e^{ 2p \int_0^Tr(s) ds} \E \Big[   e^{ 2p  \int_0^T | \lambda_s^\T A_s | ds}  \Big] + \Big(\frac{2p}{2p-1}\Big)^{2p} \E \Big[   e^{- p \int_0^T A_s^\T  A_s  ds + 2p\int_0^T A_s^\T dB_s}  \Big]\right)
	\end{align*}
	which is finite since on the first hand condition~\eqref{eq:assumption_novikov} ensures that the first term is finite i.e. 
	\begin{equation}
		\E \Big[   e^{ 2p  \int_0^T | \lambda_s^\T A_s | ds}  \Big] \leq  \E \left[ \exp\left({a(p) \int_0^T \left( |\lambda_s|^2 + |\Lambda_s|^2\right)  ds}\right)    \right]< \infty,
	\end{equation}
	with constant \(a(p) =p \left(2 + |\Sigma| \right)\) where 	we used the elementary inequality $|ab|\leq (|a|^2+|b|^2)/2 $, to bound
	\begin{equation*}
		|\lambda_s^{\T} A_s| = \big(|\lambda_s|^2 + |(\Sigma\Lambda_s)^\T\lambda_s|\big) \leq (|\lambda_s|^2 + |\Sigma| \frac{|\lambda_s|^2 + |\Lambda_s|^2}{2}) \leq \frac12(2 + |\Sigma|)(|\lambda_s|^2 +  |\Lambda_s|^2)
	\end{equation*}
	and, on the other hand, for the second term {by virtue of the Cauchy-Schwarz inequality},
	\begin{align*}
		\E \Big[   e^{- p \int_0^T A_s^\T  A_s  ds + 2p\int_0^T A_s^\T dB_s}  \Big] &\leq \left(   \E \left[ e^{(8p^2 - 2p) \int_0^T  A_s^\T  A_s   ds} \right]    \right)^{1/2} \left(\E \left[  e^{-8p^2\int_0^T A_s^\T  A_s ds + 4p \int_0^T A^\T_s dB_s}    \right] \right)^{1/2} \\
		&  \leq  \left(   \E \left[ e^{a(p) \int_0^T \left(|\lambda_s|^2 + |\Lambda_s|^2 \right)  ds} \right]    \right)^{1/2} \times 1  < \infty,
	\end{align*}
	with \(a(p) =2 (8p^2 {- 2p}) \left( 1  + |{\Sigma}|^2  \right)\) and where we used Jensen's inequality to bound 
	\begin{equation}
		A_s^\top A_s = |\lambda_s + \Sigma\Lambda_s|^2 \leq 2(|\lambda_s|^2 +  |\Sigma\Lambda_s|^2) \leq 2 (1 + |\Sigma|^2)(|\lambda_s|^2 +  |\Lambda_s|^2),
	\end{equation}
	together with condition~\eqref{eq:assumption_novikov} and Novikov's condition to the Dol\'eans-Dade exponential $\mathcal{E}(4p \int_0^{{\cdot}} A_s^\T dB_s)$. 
	
	\smallskip
	\noindent {\sc Step~3} (\textit{Proof of Admissibility and optimal value for the inner problem:}) We next address the admissibility of the optimal control $\c^*(\xi)$ in~\eqref{eq:optimal_control} for any \(\xi\in\R\). 
	Finally, to get that $\c^*(\xi)$ is admissible,  we are left to prove that $\alpha^*(\xi) \in L^2_{\F}([0,T], \R^d)$.   Let $1/p + 1/q=1$, for some $p,q>1$. By H\"older's inequality, we may write  
	\begin{align*}
		\E \left[ \int_0^{T} |\c_s^*(\xi)|^2 ds \right] & { =} \E \left[ \int_0^{T} | A_s \tilde{X}_s^*|^2 ds \right] \leq  \E \left[ \sup_{t \in [0,T]} |\tilde{X}_t^*|^2 \int_0^T |A_s|^2 ds \right] \\
		& \leq \Big(  \E \Big[ \sup_{t \in [0,T]} |\tilde{X}_t^*|^{2p}  \Big] \Big)^{\frac1p} \Big( \E\Big[ \Big( \int_0^T |A_s|^2 ds \Big)^{q}  \Big]  \Big)^{\frac1q} \\
		&   \leq \Big( \E \Big[ \sup_{t \in [0,T]} |\tilde{X}_t^*|^{2p}  \Big] \Big)^{\frac1p}   \Big( \E\Big[ \Big(  2 (1 + |\Sigma|^2) \int_0^T \Big(|\lambda_s|^2 +  |\Lambda_s|^2\Big) ds \Big)^{q}  \Big]  \Big)^{\frac1q} \\
		&   \leq  c_q^{\frac1q} \Big( \E \Big[ \sup_{t \in [0,T]} |\tilde{X}_t^*|^{2p}  \Big] \Big)^{\frac1p}   \Big( \E\Big[ \exp\Big(  2\times 6 (1 + |\Sigma|^2) \int_0^T \Big(|\lambda_s|^2 +  |\Lambda_s|^2\Big) ds \Big)  \Big]  \Big)^{\frac1q} < \infty
	\end{align*}
	where the last term is finite due to condition  \eqref{eq:assumption_novikov} and the inequality $|z|^q\leq c_{q}e^{|z|}$ for all $z \in \R$, together with the bound in Equation~\eqref{eq:boundX}. 
	
	\smallskip
	\noindent Finally, as for the optimal value, the cost functional is minimized when $\c^*(\xi)$ is given by \eqref{eq:optimal_control} and  the optimal value of \eqref{pb:P(c)} is equal to \begin{align*}
		\tilde V(\xi) \; = \; \Gamma_0 \big| \Tilde{X}_0^{\c^*(\xi)} \big|^2 \;= \;  \Gamma_0 \big| X_0 - \xi e^{-\int_0^T r(s) ds} \big|^2,
	\end{align*}
	thanks to~\eqref{Eq:Square}. This gives \eqref{eq:optimal_vamue_inner}. The proof is complete and we are done \hfill $\Box$

\medskip\noindent {\bf Remark and conclusion:} As a concluding remark, we highlight a natural direction for future research that arises from this work. In particular, it would be of
interest to investigate the existence of a unique continuous solution $\psi \in C([0,T],\R^d)$ to the generalized inhomogeneous integral
Riccati-Volterra equation with L\'evy jump compensator~\eqref{eq:RiccatiGeneral}-~\eqref{eq:Funct_General}. Such systems involve non-linear cross-component interactions in \(\psi\) and pose significant analytical challenges that are beyond the scope of the present paper. 

	\medskip
\noindent {\bf Acknowledgements:} 
	 The authors would like to thank I. Kharroubi and G.  Pag\`es
    for their invaluable insights and for many fruitful and inspiring discussions.
	
	\bibliographystyle{plainnat}
	\bibliography{references}
	
	\appendix
	
	\section{Supplementary material and Proofs. }\label{app:lemmata}
		
			\subsection{Existence of solutions for inhomogeneous Riccati-Volterra equations in the continuous case}\label{app:SolRiccati}
			\noindent We derive the existence and uniqueness of the solution to an inhomogeneous Riccati--Volterra equation. 
			
			\noindent {$\rhd$ {\em Preliminaries}. As a first preliminary, we recall the following result which deals with non-negativity of solutions to linear deterministic Volterra
			equations.
			\begin{Lemma}\label{lemma:LinearVoltPositiv}
				Fix \(d\geq1\), $T > 0$ and let \(K\) be diagonal with scalar kernels $K_i$ on the diagonal for \(i=1,\ldots,d\) that is completely monotone on $(0, \infty)$ satisfying Assumption~\ref{assump:kernelJumpVolterra}~$(i)$.
				Let
				$F \in C([0,T],\mathbb{R}^d)$ and $G \in C([0,T],\mathcal{M}_d(\mathbb{R}))$ 
				be such that $F_i \ge 0$, and $G_{ij} \ge 0$ for all 
				$i,j = 1,\dots,d$ with $i \ne j$. Then the linear Volterra equation
				\begin{equation}\label{eq:volterra_linear}
					\chi(t) = \int_0^t K(t-s)\big(F(s) + G(s)\chi(s)\big)\,ds
				\end{equation}
				has a unique solution $\chi \in C([0,T],\mathbb{R}^d)$ satisfying 
				$\chi_i(t) \ge 0$ for all $t \in [0,T]$ and $i=1,\dots,d$. 
			\end{Lemma}
			\noindent {\bf Proof:} This follows directly from \cite[Theorem C.2]{abi2019affine}. The continuity follows from the uniqueness of the global solution and ~\cite[Theorem 12.1.1]{gripenberg1990}. 
			\begin{Theorem}\label{Thm:Riccatilocalexistence} Fix a kernel $K$ satisfying Assumption~\ref{assump:kernelJumpVolterra}~$(i)$ for any $T>0$ along with functions $f: \R_+ \to \R^d$ and $F: \R_+ \times \R^d \to \R^d$ define by
				$$ F_i(t,x):=x^\top A_i(t)x + b_i(t)^\top x, \quad i=1,\ldots, d, \quad (t,x)\in \R_+\times \R^d,$$ where $A_i:\R_+\to \mathcal M_d(\R)$, $b_i:\R_+\to \R^d$, $f_i :\R_+\to \R$ are continuous functions. Then, it holds that: 
				
				\begin{itemize}
					\item[$(a)$] The deterministic Volterra equation 
					\begin{align}\label{eq:localexist}
						\psi(t) = \int_0^t K(t-s) \big(f(T-s) + F(T-s,\psi(s))\big) \d s, \quad t\in [0,T]
					\end{align} 
					admits a unique  non-continuable  solution $\psi \in C([0,T_{max}),\R^d)$ in the sense that $\psi$ satisfies   \eqref{eq:localexist} on $[0,T_{max})$ with $T_{max} \in (0,T]$ and $\sup_{t<T_{max}}|\psi( t)| = +\infty$, if $T_{max}<T$. 
				\end{itemize}
				\smallskip
				Moreover, let \(B:=\big(b_1,\ldots,b_d\big)\) and assume that \(B_{ij}\geq 0\) for all \(i,j=1,\ldots,d\) and $i\neq j$. 
				Assume further that $A_i = c_i a^i \, e_i e_i^\top,$ where \(e_i\) denote the \(i\)-th canonical basis vector of \(\mathbb{R}^d\), \(c_i \in \mathbb{R}\) and \(a^i : \mathbb{R}_+ \to \mathbb{R}_+\) is a continuous function.
				\begin{itemize} 
					\item[$(b)$] If \(f\) has nonpositive components, then the unique non-continuable continuous solution \((\psi, T_{\max})\) to Equation~\eqref{eq:localexist} satisfies \(\psi_i \le 0\) for \(i=1,\ldots,d\), that is, $\psi \in C([0,T_{max}),\R^d_-)$.
					 
				\item[$(c)$] if \(c_i\geq0\) for all \(i=1,\ldots,d\), then for any \(f \in C([0,T], \mathbb{R}^d_-)\), the Riccati--Volterra equation \eqref{eq:localexist} admits a unique global solution 
				\(\psi \in C([0,T], \mathbb{R}^d_-)\), that is, \(\psi_i \le 0\) for \(i=1,\ldots,d\).
				\end{itemize}  
			\end{Theorem}
			\noindent {\bf Proof:}
				\smallskip
				\noindent  {\sc Step~1} \textit{(Existence and uniqueness of local or maximally defined solution.)}  The first claim concerns the existence of local solutions to deterministic Volterra equations  of Hammerstein type: 
				It follows  from \cite[Theorem 12.2.6 or Theorem 12.1.1]{gripenberg1990} (see also \cite[ Theorem 3.1.2 ]{Brunner2017} ). In addition, note that $T_{max}$ is chosen to be  maximal, in the sense that the solution cannot be continued
				beyond $[0,T_{max})$.
				
				\smallskip
				\noindent
				{\bf Remark:} More generally, assume Equation~\eqref{eq:localexist} is defined on \(\R_+\). If \( f \in C(\mathbb{R}_+, \R^d) \) (resp.\( f \in  L^1_{\mathrm{loc}}(\mathbb{R}_+; \R^d) \) ), a non-continuable solution of~\eqref{eq:localexist} is a pair \((\psi, T_{\max})\) with \(T_{\max} \in (0,\infty]\) and $\psi \in C([0,T_{max}),\R^d)$
					(resp.\(\psi \in L^2_{\mathrm{loc}}([0, T_{\max})); \R^d)\), such that \(\psi\) satisfies~\eqref{eq:localexist} on 
					\([0, T_{\max})\) and $\sup_{t<T_{max}}|\psi( t)| = +\infty$ (resp. \(\|\psi\|_{L^2(0, T_{\max})} = \infty\) ) whenever \(T_{\max} < \infty\). 
					If \(T_{\max} = \infty\), we call \(\psi\) a global solution of~\eqref{eq:localexist}.
			
			\smallskip
			\noindent  {\sc Step~2} \textit{(Non-positivity of the solution.)} 
			We now deal with the non-positivity of  solutions to the deterministic Volterra equation~\eqref{eq:localexist}. For this, we consider two cases.\\
			1. If {\textit{\(c_i \geq 0\)}}, we observe that,  
			on the interval $[0,T_{\rm max})$, the function   $- \psi^i$ satisfies the linear equation
			\begin{equation}\label{eq:mod_chi}
				\chi_i(t) = \int_0^t K_i(t-s) \left( -f_i(T-s) + \big(B^\top(T-s) \chi(s)\big)_i+  c_ia^i(T-s)\psi^i (s) \chi^i(s)  \right) ds, \; t< T_{\max}\leq T.
			\end{equation}
			which has by Lemma~\ref{lemma:LinearVoltPositiv} a unique solution  $\chi\in C([0,T_{\rm max}),\R^d)$ with $\chi_i\ge0$, $i=1,\ldots, d,$ owing to assumption \ref{assump:kernelJumpVolterra}~$(i)$ on \(K\) and the fact that $-f_i \geq0$ (\(f\) has nonpositive components) and \(B_{ij}\geq 0\) for $i\neq j$.
			 Then, the function $\psi  \in C([0,T_{\rm max}), \R_{-}^d)$ solves the Riccati--Volterra equation \eqref{eq:localexist}. 
			 
			 \smallskip
			 \noindent 2. If {\textit{\(c_i < 0\)}}, its suffices to consider \(\tilde{\psi} := c_i\psi\). Then \(\tilde{\psi}\) is such that \(\tilde{\psi}^i\) satisfies the linear equation
			 \begin{equation}\label{eq:mod_chi2}
			 	\chi_i(t) = \int_0^t K_i(t-s) \left( c_if_i(T-s) + \big(B^\top(T-s) \chi(s)\big)_i +  a^i(T-s)\psi^i (s) \chi^i(s)  \right) d\,s, \; t< T_{\max}\leq T.
			 \end{equation}
			 which has, still by Lemma~\ref{lemma:LinearVoltPositiv} a unique solution  $\chi\in C([0,T_{\rm max}),\R^d)$ with $\chi_i\ge0$, $i=1,\ldots, d,$ owing to assumption \ref{assump:kernelJumpVolterra}~$(i)$ on \(K\) and the fact that $c_if_i \geq0$ (\(f\) has nonpositive components) and \(B_{ij}\geq 0\) for $i\neq j$.\\
			Finally, in both cases, there exists a
			unique maximally defined continuous solution $\psi  \in C([0,T_{\rm max}), \R_{-}^d)$ to the Riccati--Volterra equation \eqref{eq:localexist}
			
			\smallskip
			\noindent  {\sc Step~3} \textit{Global existence:} We are now going to show that any local solution can be extended
			to a local solution on a larger interval.
			Our aim is to prove that $T_{\rm max} \geq T$ for every \(T>0\) by showing that
			\begin{equation}\label{eq:temptmax}
				\sup_{t< T_{\rm max}} | \psi(t)| < \infty.  
			\end{equation}
			\noindent Let $\ell \in C([0,T],\R^d)$ be the solution of the linear deterministic Volterra equation
			\begin{align}
				\ell(t) = \int_0^t K(t-s) \big(f(T-s) + B^\top(T-s) \ell(s)\big) \d s, \quad t\in [0,T]
			\end{align}
			\noindent Observing that the function  $ \chi:=\psi - \ell$  satisfies the equation 
			\begin{equation}\label{eq:sumVolt}
				\chi_i(t) = \int_0^t K_i(t-s) \left( \big(B^\top(T-s) \chi(s)\big)_i+  c_ia^i(T-s) \psi^i(s)^2  \right) d\,s, \quad t< T_{\max}\leq T.
			\end{equation}
			on $[0,T_{\rm max})$, another application of Lemma~\ref{lemma:LinearVoltPositiv} yields that \(\ell^i \le \psi^i\) on \([0,T_{\max})\).
			In summary, we have shown that
			\[
			\ell^i \le \psi^i \le 0 
			\quad \text{on } [0,T_{\max}), 
			\qquad i=1,\ldots,d.
			\]
			Since $\ell$  is a global solution (apply Lemma~\ref{lemma:LinearVoltPositiv} with \(\chi:=-\ell\)) and thus have finite norm on any bounded interval, this implies \eqref{eq:temptmax} so that $T_{\rm max}\geq T$ as needed. This complete the proof.\qed
			\begin{Corollary}\label{Corol:Riccatilocalexistence} Let $T>0$ be fixed and suppose that the assumptions of 
				Theorem~\ref{Thm:Riccatilocalexistence}~$(c)$ hold.
				Assume moreover that the matrix-valued function 
				$B : \mathbb{R}_+ \to \mathcal{M}_d(\mathbb{R})$ is diagonal, i.e., $\diag{(B_1,\dots, B_d)}=:B :\R_+\to \mathcal M_d(\R)$. For $i=1,\ldots,d$, define \(\bar{\lambda}_i:=-\sup_{t\in [0,T]} B_i(t)\) and assume that $\bar{\lambda}_i \neq 0$.
			 Then the unique global solution 
			 \(\psi \in C([0,T], \mathbb{R}^d_-)\) of the Riccati--Volterra equation \eqref{eq:localexist} 
			 satisfies the following bound:
			 \begin{equation}
			 	\sup_{t \in [0,T]} |\psi^i(t)| \leq \frac{\|f_i\|_{\sup, T}}{\bar{\lambda}_i}\int_0^T f_{\bar{\lambda}_i}(s)ds = \frac{\|f_i\|_{\sup, T}}{\bar{\lambda}_i}(1- R_{\bar{\lambda}_i}(T)), \;i=1,\ldots,d.
			 \end{equation}
			 where $R_{\bar{\lambda}_i}$ is the \textit{ $\bar{\lambda}_i$-resolvent} associated to the real-valued kernel $K_i$ and $f_{\bar{\lambda}_i}$ its antiderivative (see e.g.,~\cite[Section 2.2]{EGnabeyeu2025}).
			\end{Corollary}
			\noindent {\bf Proof:} If the matrix \(B\) is diagonal (which is the case if for example the matrix $D$ in the drift of the volatility process is a diagonal matrix, i.e. $D=-\diag{(\lambda_1,\dots, \lambda_d)}$), namely  $\diag{(B_1,\dots, B_d)}=:B :\R_+\to \mathcal M_d(\R)$, then the vector valued equation~\eqref{eq:localexist} can be decomposed into $d$ real valued inhomogeneous Riccati-Volterra equations such that for the $i-$th component $\psi^i$ of $\psi$, $i=1,\ldots,d,$ we obtain 
			\vspace{-.3cm}
			\begin{equation}\label{comp}
				\psi^i(t)=\int_0^t K_i(t-s)\Big[ f_i(T-s) + B_i(T-s)\psi^i(s)+c_ia^i(T-s)\psi^i(s)^2\Big]ds, \quad t\leq T.
			\end{equation}
			Theorem~\ref{Thm:Riccatilocalexistence}~$(c)$ above still guarantees the existence of a unique continuous global solution of the equation \eqref{comp} on \([0,T]\). Now, let $g \in C([0,T],\R_+)$ be the solution of the linear deterministic Volterra equation
			\vspace{-.3cm}
			\begin{align}
				g(t) = \int_0^t K_i(t-s) \big(|f_i(T-s)| + B_i(T-s) g(s)\big) \d s, \quad t\in [0,T]
			\end{align}
			Since $f_i\leq0$, we get that $\psi^i$ is non-positive and the fact that the function  $ \chi:=\psi^i +g$  satisfies the equation~\eqref{eq:sumVolt} on \([0,T]\) leads (still owing to Lemma~\ref{lemma:LinearVoltPositiv} and noting that \(c_i\geq0\) for all \(i=1,\ldots,d\)) to \(-g \le \psi^i\leq 0\) on \([0,T]\). We obtain that \(|\psi^i|\leq g\). Moreover define \(h\) as the unique continuous solution of the following linear deterministic Wiener-Hopf equation
			\begin{align}\label{eq:wienerhopf}
				h(t) = \int_0^t K_i(t-s) \big(\|f_i\|_{\sup, T} -\bar{\lambda}_i h(s)\big) \d s, \quad t\in [0,T]
			\end{align}
			where \(\bar{\lambda}_i:=\inf_{t\in [0,T]} -B_i(t)\).
			Then, $h - g$ solves the equation \(\chi = K_i * (-\bar{\lambda}_i \chi + w),\)
			with \(w = (-\bar{\lambda}_i - B_i(T-\cdot))g + \|f_i\|_{\sup,T} - |f_i(T-\cdot)|,\)
			which is a non-negative function on $[0,T]$. Another application of Lemma~\ref{lemma:LinearVoltPositiv} now yields \(g \leq h\) so that \(|\psi^i|\leq g \leq h\).
			Finally, the claimed bound follows by noticing that, the unique solution of the linear deterministic Wiener-Hopf equation~\eqref{eq:wienerhopf} given by \cite[Proposition 2.4]{EGnabeyeu2025} reads:
			\begin{equation}\label{eq:sol_h}h(t) = ((K_i-f_{\bar{\lambda}_i}*K_i)*\|f_i\|_{\sup, T})(t)=\frac{1}{\bar{\lambda}_i}(f_{\bar{\lambda}_i}*\|f_i\|_{\sup, T})(t) =\frac{\|f_i\|_{\sup, T}}{\bar{\lambda}_i}(1- R_{\bar{\lambda}_i}(t)), \; t \in [0,T].
		   \end{equation}
		   Note that, the right hand side is always positive. Indeed, when \(\bar{\lambda}_i>0\), then by~\cite[Proposition 5.1.]{EGnabeyeu2025}, the function \(1- R_{\bar{\lambda}_i}\) is a Bernstein function, thus non-negative and non-decreasing. 
		   If we rather have \(\bar{\lambda}_i<0\), then we note that \(-1+ R_{\bar{\lambda}_i}\) is non-negative owing to the \textit{Neumann series expansion} of \(R_{\bar{\lambda}_i}\) given in~\cite[Equation 2.7]{EGnabeyeu2025}.
		   This complete the proof and we are done. \hfill $\square$

\end{document}